\documentclass{amsart}

\usepackage[utf8]{inputenc}
\usepackage[english]{babel}
\usepackage{graphicx}
\usepackage{amssymb}
\usepackage{amsmath}
\usepackage{tikz-cd}
\usepackage{adjustbox}
\usepackage{mathabx}
\usepackage{hyperref}
\usepackage{framed}
\usepackage{mathtools}
\usepackage{amsaddr}
\usepackage{bbm}
\usepackage{mathrsfs}
\usepackage{mathscinet}
\usepackage{rotating}

\usepackage[all]{xy}
\newtheorem{theorem}{Theorem}[section]
\newtheorem{lemma}[theorem]{Lemma}
\newtheorem{proposition}[theorem]{Proposition}
\newtheorem{corollary}[theorem]{Corollary}
\newtheorem{conjecture}[theorem]{Conjecture}

\theoremstyle{definition}
\newtheorem{definition}[theorem]{Definition}
\newtheorem{assumption}[theorem]{Assumption}
\newtheorem{example}[theorem]{Example}

\newtheorem{claim}[theorem]{Claim}

\newtheorem{warning}[theorem]{Warning}

\theoremstyle{remark}
\newtheorem{remark}[theorem]{Remark}

\numberwithin{equation}{section}

\newcommand{\ov}{\overline}

\newcommand{\Int}{\mathrm{Int}}

\newcommand{\der}{\mathrm{der}}
\newcommand{\mc}{\mathcal}

\newcommand{\Q}{\mathbb{Q}}

\newcommand{\ms}{\mathsf}
\newcommand{\mf}{\mathfrak}

\newcommand{\inv}{\mathrm{inv}}

\newcommand{\ad}{\mathrm{ad}}
\newcommand{\iso}{\mathrm{iso}}
\newcommand{\bas}{\mathrm{bas}}
\newcommand{\alg}{\mathrm{alg}}

\newcommand{\C}{\mathbb{C}}
\newcommand{\R}{\mathbb{R}}
\newcommand{\Gal}{\mathrm{Gal}}
\newcommand{\Z}{\mathbb{Z}}

\newcommand{\im}{\mathrm{im}}

\newcommand{\SL}{\mathrm{SL}}
\newcommand{\tr}{\mathrm{tr}}
\newcommand{\Trans}{\mathrm{Trans}}

\newcommand{\Ind}{\mathrm{Ind}}
\newcommand{\Irr}{\mathrm{Irr}}
\newcommand{\Jac}{\mathrm{Jac}}
\newcommand{\Gm}{{\mathbb{G}_{\mathrm{m}}}}
\newcommand{\Hom}{\mathrm{Hom}}

\newcommand{\ab}{\mathrm{ab}}

\newcommand{\rel}{\mathrm{rel}}
\newcommand{\Res}{\mathrm{Res}}

\newcommand{\reg}{\mathrm{reg}}

\newcommand{\GL}{\mathrm{GL}}
\newcommand{\op}{\mathrm{op}}

\newcommand{\co}[1]{\prescript{#1}{}}
\newcommand{\wh}{\widehat}
\newcommand{\res}{\mathrm{res}}
\newcommand{\LLC}{\mathrm{LLC}}
\newcommand{\Bun}{\mathrm{Bun}}
\newcommand{\pure}{\mathrm{pure}}
\newcommand{\Kal}{\mathrm{Kal}}

\newcommand{\D}{\mathbb{D}}
\newcommand{\Stab}{\mathrm{Stab}}

\newcommand{\PGL}{\mathrm{PGL}}

\newcommand{\Gbas}{G\mathchar`-\mathrm{bas}}
\newcommand{\Mbas}{M\mathchar`-\mathrm{bas}}
\newcommand{\acc}{\mathrm{acc}}
\newcommand{\msc}{\mathscr}
\newcommand{\emb}{\mathrm{emb}}
\newcommand{\Lie}{\mathrm{Lie}}
\newcommand{\Ad}{\mathrm{Ad}}
\newcommand{\st}{\mathrm{st}}
\newcommand{\Dist}{\mathrm{Dist}}
\newcommand{\SO}{\mathrm{SO}}
\newcommand{\Sp}{\mathrm{Sp}}
\newcommand{\xO}{\mathrm{O}}
\newcommand{\wM}{{}^{w}M}
\newcommand{\IndCoh}{\mathrm{IndCoh}}
\newcommand{\Par}{\mathrm{Par}}
\newcommand{\Rep}{\mathrm{Rep}}
\newcommand{\ul}{\underline}

\title{The B(G)-parametrization of the local Langlands correspondence}

\author{Alexander Bertoloni Meli}
\address{Department of Mathematics and Statistics, 665 Commonwealth Ave, Boston, MA 02215, USA.}
\email{abertolo@bu.edu}

\author{Masao Oi}
\address{Department of Mathematics, National Taiwan University, Astronomy Mathematics Building 5F, No.\ 1, Sec.\ 4, Roosevelt Rd., Taipei 10617, Taiwan.}
\email{masaooi@ntu.edu.tw}

\begin{document}
\begin{abstract}
    This article is on the parametrization of the local Langlands correspondence over local fields for non-quasi-split groups according to the philosophy of Vogan. We show that a parametrization indexed by the basic part of the Kottwitz set (which is an extension of the set of pure inner twists) implies a parametrization indexed by the full Kottwitz set. On the Galois side, we consider irreducible algebraic representations of the full centralizer group of the $L$-parameter (i.e., not a component group). When $F$ is a $p$-adic field, we discuss a generalization of the endoscopic character identity.
\end{abstract}

\maketitle

\tableofcontents

\section{Introduction}
For a quasi-split connected reductive group $G$ over a local field $F$, a \textit{local Langlands correspondence (LLC)} is a map
\begin{equation*}
    \LLC_G\colon \Pi(G)  \; \longrightarrow \; \Phi(G),
\end{equation*}
satisfying certain desiderata. Here, $\Pi(G)$ is the set of isomorphism classes of irreducible admissible $\C$-valued representations of $G(F)$ and $\Phi(G)$ is the set of $\widehat{G}$-conjugacy classes of $L$-parameters $\phi\colon W_F \times \SL_2 \to \co{L}G$ (throughout the text, we conflate $\wh{G}$ and $\co{L}{G}$ with their $\C$-points). The map $\LLC_G$ is not injective in general but it has finite fibers which are denoted $\Pi_{\phi}(G)$ and called $L$-\emph{packets}. The constituents of each $L$-packet are parametrized, after a choice of a Whittaker datum of $G$, by $\Irr(\pi_0(S_{\phi}/Z(\wh{G})^{\Gamma}))$, where $S_{\phi} := Z_{\wh{G}}(\im \phi)$ is a potentially disconnected reductive group and where $\Irr(\pi_0(S_{\phi}/Z(\wh{G})^{\Gamma}))$ denotes the set of irreducible representations of the finite group $\pi_0(S_{\phi}/Z(\wh{G})^{\Gamma})$ ($\Gamma$ denotes the absolute Galois group of $F$). The existence of an LLC was conjectured by Langlands and by now constructions are known in many cases. At this point, the literature is too rich to acknowledge every contribution, but we briefly mention some results here. Some further remarks are given in Remark \ref{rem: mainthmrems}. 
In the Archimedean case, an LLC for all groups is known by work of Langlands and Shelstad (see \cite{ShelstadLindistinguishability}). 
In the case of $p$-adic fields, LLC's have been constructed for $\GL_n$, by \cite{HT1} and \cite{HenniartLLC} and for $\Sp_{2n}$, $\SO_{2n+1}$, and $\xO_{2n}$ by \cite{Arthurbook}. Unitary groups and their inner twists were handled by \cite{KMSW}, \cite{Mok}.

In the case where $G$ is not quasi-split, Whittaker data are no longer defined, and in any case, the two sets are not always in bijection. Vogan realized (\cite{VoganLLC}) that instead of trying to parametrize the $L$-packets of $G$ on their own, one does better by simultaneously parametrizing $L$-packets of a collection of suitably rigidified inner twists of the unique quasi-split inner form $G^{\ast}$ of $G$. 
For various reasons, inner twists classified by $H^1(F, G_{\ad})$ are not suitable; for example, they can have outer-automorphisms that act non-trivially on representations of $G(F)$. Some natural suitable collections of inner twists are parametrized by $H^1(F,G)$,  $B(G)_{\bas}$, or $H^1(u \to W, Z(G) \to G)$ (see \cite{KalethaLLCnonQS} for the details).
In each case, one can conjecture an expected parametrization of $L$-packets. For instance, in the case of $B(G)_{\bas}$, Kottwitz conjectured a bijection (\cite[Conjecture F]{KalethaLLCnonQS}) 
\begin{equation*}
    \begin{tikzcd}
          \coprod\limits_{b \in B(G)_{\bas}} \Pi_{\phi}(G_b) \arrow[r, "\iota_{\mf{w}}"] & \Irr(S^{\natural}_{\phi}),
    \end{tikzcd}
\end{equation*}
where $S^{\natural}_{\phi} = S_{\phi} / (\wh{G}_{\der} \cap S_{\phi})^{\circ}$ ($\wh{G}_\der$ denotes the derived subgroup of $\wh{G}$).  The notation $B(G)_{\bas}$ refers to the \emph{basic elements} of the Kottwitz set $B(G)$, that is at the center of this work. This definition and combinatorial description of $B(G)$ are reviewed in \S\ref{ss: Kottwitzreview}. Most simply, $B(G)$ is given by the Frobenius-twisted conjugacy classes of $G(\breve{F})$ when $F$ is a $p$-adic field (here, $\breve{F}$ denotes the completion of the maximal unramified extension of $F$). For each $b \in B(G)$, there is a canonically associated group $G_b$, that is an inner twist of a standard Levi subgroup of $G$. The element $b$ is \emph{basic} precisely when $G_b$ is an inner twist of $G$ itself.

The main result of our paper is that a $B(L)_{\bas}$-parametrization of an LLC for each standard Levi subgroup $L \subset G$ implies a parametrization of a ``generalized LLC'' for $G$ involving the full Kottwitz set, $B(G)$, and each group $G_b$. Before stating our results precisely, we recall the work of \cite{fargues--scholze}, which is the main motivation for our paper.

Since its inception, $B(G)$ has been known to be central to the construction of Rapoport--Zink spaces and more generally, local Shimura varieties. Motivated by this, Fargues \cite{farguesgeometrization} outlined a program to geometrize the LLC for $p$-adic fields. Tremendous progress towards completing this program was made in \cite{fargues--scholze}. In particular, they conjecture that the LLC comes from an equivalence of categories. On the automorphic side, they define a v-stack $\Bun_G$ whose points are canonically in bijection with the Kottwitz set: $|\Bun_G| \cong B(G)$. They further define a derived category of sheaves $D(\Bun_G, \ov{\Q_{\ell}})$ containing, for each $b \in B(G)$, the derived category $\Rep(G_b(F))$ of smooth representations of $G_b(F)$. On the Galois side, they consider a variant of the stack of $L$-parameters $\Par_G$ first defined in \cite{DHKM}, and the ind-completion $\IndCoh(\Par_G)$ of its derived category of coherent sheaves.
\begin{conjecture}[{\cite[I.10.2]{fargues--scholze}}]{\label{conj: catconj}}

There exists a canonical equivalence of $\infty$-categories 
\begin{equation*}
    \IndCoh(\Par_G) \cong D(\Bun_G, \ov{\Q_{\ell}}).
\end{equation*}
\end{conjecture}

The category $D(\Bun_G, \ov{\Q_{\ell}})$ carries a perverse t-structure and the irreducible perverse sheaves are known to be in bijection with the set of pairs $(b, \pi)$, where $b \in B(G)$ and $\pi \in \Pi(G_b)$. 
This perverse $t$-structure must correspond to some $t$-structure on $\IndCoh(\Par_G)$, and Fargues and Scholze conjecture \cite[Remark I.10.3]{fargues--scholze} its irreducible objects are given by pairs $(\phi, \rho)$ for $\phi \in \Phi(G)$ and $\rho \in \Irr(S_{\phi})$. If this conjecture is true, then these sets of pairs must naturally be in bijection. Our motivation in this paper is therefore to show how such a bijection follows from the classical, $B(G)_{\bas}$ formulation of LLC. We remark that for $G$ with connected center, the $B(G)_{\bas}$-parametrization is known to be equivalent to the rigid parametrization of Kaletha by \cite{KallocalBGvsRig}. Kaletha further shows that knowing the $B(G)_{\bas}$-parametrization of LLC for all $G$ is equivalent to knowing the rigid parametrization for all $G$.

Our main result is then
\begin{theorem}[See \S\ref{s: basiccorrespondence}, \S\ref{s: theconstruction}]{\label{thm: intromainthm}}
Let $G$ be a quasi-split connected reductive group with a fixed Whittaker datum $\mf{w}$.
Suppose that there is an LLC for $G$ and its $B(G)_{\bas}$-inner twists as well as an LLC for each proper Levi subgroup $L \subset G$ and its $B(L)_{\bas}$-inner twists.
Then there is a natural LLC for the $B(G)$-twists of $G$ and a bijection
\begin{equation*}
    \begin{tikzcd}
          \coprod\limits_{b \in B(G)} \Pi_{\phi}(G_b) \arrow[r, "\iota_{\mf{w}}"] & \Irr(S_{\phi}),
    \end{tikzcd}
\end{equation*}
where $\Irr(S_{\phi})$ now denotes the set of irreducible algebraic representations of $S_{\phi}$.
(See \S \ref{s: basiccorrespondence} for the precise meaning of ``LLC".) 
\end{theorem}

One advantage of Theorem \ref{thm: intromainthm} is that, as suggested by the conjectures of Fargues and Scholze, only the group $S_{\phi}$ appears as opposed to its variants. On the other hand, the use of the disconnected reductive group $S_{\phi}$ as opposed to variants of its component group is one of the main subtleties we must contend with. The algebraic representations of a disconnected reductive group form a highest weight category (\cite{AcharHardestyRicherepthrydisconnected}) and this structure is central to our construction.

Showing that a formulation of the LLC is ``canonical'' in any sense is known to be a subtle question. On the other hand, we claim that our construction is the natural extension of the $B(G)_{\bas}$-LLC in the following sense. Given a pair $b \in B(G)$ and $\pi \in \Pi(G_b)$, there is a unique standard Levi subgroup $L$ and $b_L \in B(L)_{\bas}$ such that $b$ equals the image of $b_L$ under the map $B(L) \rightarrow B(G)$ and $b_L$ is ``$G$-dominant''. These notions are defined precisely in \S\ref{ss: Kottwitzreview}. Then we can consider $G_b$ as an inner twist of $L$ via $b_L$, and by the $B(L)_{\bas}$-LLC, there exists a corresponding pair $(\phi_L, \rho_L)$. Now we get an $L$-parameter $\phi$ of $G$ by composing $\phi_L$ with the map $\co{L}{L} \hookrightarrow \co{L}{G}$. On the other hand, by the representation theory of disconnected reductive groups of \cite{AcharHardestyRicherepthrydisconnected}, $\rho_L \in \Irr(S_{\phi_L})$ is determined by certain highest weight data $(\lambda, E)$. One has that the identity component $S^{\circ}_{\phi_L}$ is a Levi subgroup of $S^{\circ}_{\phi}$ and that the same data $(\lambda, E)$ can be used to define an irreducible representation of $S_{\phi}$. Then our $B(G)$-LLC is defined to be the unique correspondence that takes $(b, \pi)$ to $(\phi, \rho)$. In fact, this is essentially the definition of the correspondence. The more involved part is showing that this actually produces a bijection.

We explain how Theorem \ref{thm: intromainthm} can be seen as part of an \emph{extended Vogan philosophy}. Each of the classes of inner twists we mentioned above ($H^1(F,G)$, $B(G)_{\bas}$, $H^1(u \to W, Z(G) \to G)$) are related to cohomology of certain Galois gerbes. For us, Galois gerbes will be extensions of $\Gamma$ of $F$ by the $\ov{F}$-points of a certain pro-multiplicative $F$-group which we call the \emph{band}. For $H^1(F,G)$, the relevant gerbe $\mc{E}^{\pure} = \Gamma$ is banded by the trivial group. For $B(G)_{\bas}$ one has the gerbe $\mc{E}^{\iso}$ banded by the pro-torus $\D_F$ with character group $\Q$. Finally, $H^1(u \to W, Z(G) \to G)$ is associated to the Kaletha gerbe $\mc{E}^{\Kal}$ banded by a certain multiplicative pro-algebraic group $u$ (to be precise, when $F$ is a local function field and in the $H^1(u \to W, Z(G) \to G)$-parametrization, one has to instead work with geometric gerbes as in \cite{Dillerylocal}). In each case, the parametrizing set is given as the cohomology $H^1_{\bas}(\mc{E}, G(\ov{F}))$, where we are taking equivalence classes of $1$-cocycles $z$ of $\mc{E}$ whose restriction to $D(\ov{F})$ (here $D$ is the band) comes from an algebraic map $\nu_z: D \to G$. Further, the ``$\bas$'' signifies that we are only considering $\nu_z$ with central image in $G$. 

The expectation evidenced by Theorem \ref{thm: intromainthm} is that one gets a cleaner parametrization by dropping this centrality condition on $\nu_z$. In the $H^1(F,G)$ case, $D=1$ so dropping this assumption does nothing. For $B(G)_{\bas}$ one gets $B(G)$. We remark that the study of non-central cocycles of $\mc{E}^{\Kal}$ and their relation to the LLC has been initiated in \cite{dilleryschweinnonbasicrigidpacketsdiscrete}.

An obvious question is if one can recover Theorem \ref{thm: intromainthm} from Conjecture \ref{conj: catconj}. Unfortunately, this seems quite subtle at present. The interested reader is advised to study \cite{beijingnotes} for a detailed picture of the relationships currently conjectured. A detailed example for $G=\PGL_2$ is informally worked out in \cite{BMPGL2notes}. For an example of the difficulties involved, given a sheaf in $\IndCoh(\Par_G)$ conjecturally corresponding to an irreducible perverse sheaf in $D(\Bun_G, \ov{\Q_{\ell}})$, it is not clear at present how to recover the data $(\phi, \rho)$. Moreover, there exist many different incarnations of the pair $(b, \pi)$ as a sheaf in $D(\Bun_G, \ov{\Q_{\ell}})$ and it is not clear which sheaf is ``correct''. More precisely, there is a canonical identification with sheaves on the $b$-stratum $\Bun^b_G \xhookrightarrow{i_b} \Bun_G$ and the category of smooth representations of $G_b(F)$. On the other hand, there are several pushforward functors such as $i_{b,\ast}, i_{b, !}, i_{b, \#}$ as well as the intermediate extension functor $i_{b, !\ast}$ and in general these functors are all different. One answer is to consider, for $\phi$ a discrete parameter (though constructions for more general $\phi$ seem possible, see \cite[\S3.1]{beijingnotes} for details) Hecke eigensheaves $\mc{F}_{\phi}$ on $\Bun_G$ as originally conjectured by Fargues \cite{farguesgeometrization}. In cases where they are understood, the $\mc{F}_{\phi}$ appear to admit decompositions in terms of \emph{tilting-extensions} of the $\pi \in \Pi_{\phi}(G_b)$ along $i_b$. On the Galois side, these sheaves appear to admit decompositions in terms $\rho \in \Irr(S_{\phi})$.

In \S\ref{s: eci} we study how the endoscopic character identities in the $B(G)_{\bas}$-LLC generalize in the non-basic case. One motivation for this is that these identities should be related to the stalks of the Hecke eigensheaves $\mc{F}_{\phi}$ (for instance see \cite[Appendix A]{hamanneisenstein} and the remarks at the end of \cite[\S3.1]{beijingnotes}).

More precisely, we define the transfer to $G_b$ of the stable distribution $S\Theta_{\phi_H}^H$ attached to a tempered $L$-parameter $\phi_H$ of an endoscopic group $H$ of $G$. 
The transfer map is essentially a composition of the Jacquet functor from $H$ to certain Levi subgroups $H_L$ of $H$ that are simultaneously endoscopic groups of $G_b$, and then the endoscopic transfer from $H_L$ to $G_b$.
\begin{equation*}
    \begin{tikzcd}
        \Dist^{\st}(H) \arrow[rd, "{\Trans^{G_b}_H}"] \arrow[swap, d, "{\bigoplus \Jac}"]& \\ 
        \bigoplus \Dist^{\st}(H_L) \arrow[swap, r, "{\sum\Trans^{G_b}_{H_L}}"]& \Dist(G_b)
    \end{tikzcd}
\end{equation*}
The goal is then to describe $\Trans^{G_b}_H S\Theta_{\phi_H}^H$ in terms of $\Pi_{\phi}(G_b)$ for $\phi := \eta \circ \phi_H$, where $\eta$ denotes the $L$-embedding ${}^LH\hookrightarrow{}^LG$.

When $H=G$, this is essentially a question of understanding the compatibility of the local Langlands correspondence with Jacquet modules and already in this case, the description is quite complicated and not known in general. In particular, $\Trans^{G_b}_H S\Theta_{\phi_H}^H$ can contain representations of $G_b$ that are associated to different $L$-parameters of $G$ (see \cite{AtobeJacquet}, though the phenomenon appears even for $\GL_4$; Example \ref{ex:second}).

In this paper, we give the following partial description of $\Trans^{G_b}_H S\Theta_{\phi_H}^H$. We first define the regular part $[\Trans^{G_b}_H S\Theta_{\phi_H}^H]_{\reg}$ of $\Trans^{G_b}_H S\Theta_{\phi_H}^H$. Standard desiderata of LLC imply that whenever $\phi_H$ has trivial $\SL_2$-part, we have $[\Trans^{G_b}_H S\Theta_{\phi_H}^H]_{\reg} = \Trans^{G_b}_H S\Theta_{\phi_H}^H$, though in general they are different. We prove the following.
\begin{theorem}[Theorem \ref{thm: ECI}]{\label{thm: intro-eci}}
    We have an equality of distributions on $G_b$.
    \begin{equation*}
        [\Trans^{G_b}_H S\Theta_{\phi_H}^H]_{\reg}
        =
        e(G_b)\sum\limits_{\pi \in \Pi_{\phi}(G_b)} \langle \pi, \eta(s)\rangle_{\reg}\Theta_{\pi},
    \end{equation*}
    where $e(G_b)$ denotes the Kottwitz sign of $G_b$ and $\langle \pi, \eta(s)\rangle_{\reg}$ is a certain number defined in \S \ref{ss: regular-pairing-def} and $\Theta_{\pi}$ is the trace distribution attached to $\pi$.
\end{theorem}
In general, $[\Trans^{G_b}_H S\Theta_{\phi_H}^H]_{\reg}$ is the transfer to $G_b$ of a certain part, $ [J^H_{P_{H_L}}S\Theta_{\phi_H}^H]_{\reg}$, of the Jacquet module $J^H_{P_{H_L}}S\Theta_{\phi_H}^H$ ranging over various Levi subgroups $H_L$ of $H$. In Appendix \ref{s: appendixA} we describe $ [J^H_{P_{H_L}}S\Theta_{\phi_H}^H]_{\reg}$ for general linear groups and show that $[\cdot]_{\reg}$ is precisely the projection to the tempered part. It would be quite interesting to extend Theorem \ref{thm: intro-eci} beyond the regular case.

Finally, we remark that the Archimedean version of Theorem \ref{thm: intromainthm} should optimistically be related to the emerging categorical Langlands conjectures for real groups due to Scholze \cite{Scholzerealnotes}.

\subsection*{Acknowledgements}
We would like to thank David Hansen for suggesting we think about this question and also giving us a lot of constructive comments. We also thank the anonymous referee for their detailed suggestions. In addition, we thank Anne-Marie Aubert and Tasho Kaletha for several helpful discussions and Sug Woo Shin for inspiring us to a more general form of our results. We thank Jonathan Leake for helping us with some Weyl group computations. A.B.M.\ was partially supported by NSF grant DMS-1840234.
M.O.\ was partially supported by JSPS KAKENHI Grant Number 20K14287, Hakubi Project at Kyoto University, and the Yushan Young Fellow Program, Ministry of Education, Taiwan.

\section{Preliminaries}{\label{s: preliminaries}}
Let $F$ be a local field with a fixed choice of algebraic closure $\ov{F}$. We write $\Gamma$ (resp.\ $W_F$) for the absolute Galois group (resp.\ the Weil group) of $F$.

Let $G$ be a quasi-split connected reductive group over $F$. 
We fix an $F$-rational splitting $(T,B,\{X_\alpha\})$ of $G$ (we follow Kottwitz's terminology here, some other authors use \emph{pinning}).
We let $\wh{G}$ denote the Langlands dual group of $G$ where we remark that we routinely conflate $\wh{G}$ with its $\C$-points. Let $\co{L}{G}$ denote the $L$-group $\wh{G} \rtimes W_F$. By fixing a splitting $(\wh{T},\wh{B},\{\wh{X}_\alpha\})$ of $\wh{G}$, we get an action of $\Gamma$ on $\wh{G}$. 
To be more precise, let $\Psi(G)$ (resp.\ $\Psi(\wh{G})$) be the based root datum of $G$ (resp.\ $\wh{G}$) determined by the Borel pair contained in the fixed splitting.
Then, by fixing an isomorphism of based root data $\Psi(G)^{\vee}\cong\Psi(\wh{G})$, where $\Psi(G)^{\vee}$ is the dual to $\Psi(G)$, we obtain a unique action of $\Gamma$ on $\wh{G}$ which preserves the fixed splitting and is compatible with the Galois action on $\Psi(G)$ through the isomorphism $\Psi(G)^\vee\cong\Psi(\wh{G})$. 

For any algebraic group $H$, we write $H^\circ$ for the identity component of $H$ and $Z(H)$ for the center of $H$.
When $H$ acts on a set $X$, for any subset $Y\subset X$, we put $Z_H(Y):=\{h\in H\mid \text{$h\cdot y=y$ for any $y\in Y$}\}$ and $N_H(Y):=\{h\in H\mid \text{$h\cdot y\in Y$ for any $y\in Y$}\}$.

We fix the following additional notation. Let $A_T \subset T$ be the maximal split subtorus and denote $\mf{A}_T = X_{\ast}(A_T)_{\R}$. Let $\ov{C}$ denote the closed Weyl chamber in $\mf{A}_T$ associated to $B$ and let $\ov{C}_{\Q}$ denote its intersection with $X_{\ast}(A_T)_{\Q}$. 
For each standard Levi subgroup $M \subset G$, let $A_M$ denote the maximal split torus in the center of $M$.
We denote $\mf{A}_M = X_{\ast}(A_M)_{\R} \subset \mf{A}_T$. 
Let ${}^{L}M$ be the standard Levi subgroup of ${}^{L}G$ which corresponds to $M$ 
(see \cite[\S 3]{Bor2} for the details of the correspondence between Levi subgroups of $G$ and those of ${}^{L}G$).
We put $\wh{M}:={}^{L}M\cap \wh{G}$ and $A_{\wh{M}}:=Z(\wh{M})^{\Gamma, \circ}$.

We write $W:=W_G:=W_{G}(T)$ and $W^\rel:=W^\rel_G:=W_{G}(A_{T})\cong W^\Gamma$.
On the dual side, similarly, we write $\wh{W}:= \wh{W}_{G}:=W_{\wh{G}}(\wh{T})$ and $\wh{W}^\rel:=\wh{W}^\rel_{G}:=W_{\wh{G}}(A_{\wh{T}})\cong\wh{W}^\Gamma$.
Since we have fixed $F$-splittings of $G$ and $\wh{G}$, we have a $\Gamma$-equivariant identification $W\cong \wh{W}$ which induces $W^\rel\cong\wh{W}^\rel$.
(We refer the reader to \cite[\S 0.4.3]{KMSW} for the details.)
In this paper, we often implicitly use these identifications of Weyl groups.

For a standard parabolic subgroup $Q$ of $G$ with standard Levi $L$, we let $J^G_Q(-)$ (resp.\ $I^G_Q(-)$) denote the associated normalized Jacquet functor (resp.\ normalized parabolic induction) (see \cite[\S 2.3]{Ber1}; our $J^G_Q$ (resp.\ $I^G_Q$) is denoted by $r_{L,G}$ (resp.\ $i_{G,L}$ in \textit{loc}.\ \textit{cit}.).

\subsection{Review of the Kottwitz set}{\label{ss: Kottwitzreview}}
In this section, we briefly review the theory of the Kottwitz set $B(G)$ for local fields following \cite{Kot9}. We follow this source instead of \cite{KottwitzisoII} since we handle a general local field $F$. In each case, the set $B(G)$ is the first cohomology of a certain Galois gerbe. 

Let $\D_F$ be the $F$-(pro-)torus defined as in \cite[\S10.4]{Kot9}.
Note that $\D_F$ is isomorphic to $\Gm$ when $F$ is Archimedean and also that $X^\ast(\D_F)\cong\Q$ when $F$ is non-Archimedean (see Remark \ref{rem: Kot97-14}).
We have an extension
\begin{equation*}
    1 \rightarrow \D_F(\ov{F}) \rightarrow \mc{E}^{\iso}_F \xrightarrow{\pi} \Gamma \rightarrow 1
\end{equation*}
such that
\begin{itemize}
\item 
when $F$ is non-Archimedean, $\mc{E}^{\iso}_F$ corresponds to $1 \in  \wh{\Z} = H^2(\Gamma, \varprojlim\mu_n(\ov{F})) \to H^2(\Gamma, \D_F(\ov{F}))$ (see \cite[\S3.1]{KallocalBGvsRig}),
\item
when $F=\R$, $\mc{E}^{\iso}_F$ corresponds to the nontrivial class of $H^2(\Gamma, \Gm(\C))$, and
\item 
when $F=\C$, $\mc{E}^{\iso}_{F} = \Gm(\C)$. 
\end{itemize}

We then define $B(G)$ in all cases to be the set $H^1_{\alg}(\mc{E}^{\iso}_F, G(\ov{F}))$ of equivalence classes of algebraic cocycles $Z^1_{\alg}(\mc{E}^{\iso}_F, G(\ov{F}))$ (see \cite[\S2, \S10]{Kot9}).

For $z \in Z^1_{\alg}(\mc{E}^{\iso}_F, G(\ov{F}))$, we define an algebraic group $G_b$ over $F$ by
\[
G_b(R)
:=
\{g\in G(R\otimes_{F}\overline{F}) \mid \Int(z_e)(\gamma(g))=g, \forall e \in \mc{E}^{\iso}_F \text{ such that } \pi(e) = \gamma \},
\]
for any $F$-algebra $R$. Then $G_b$ is an inner form of a standard Levi subgroup of $G$. This Levi subgroup is given by the centralizer of the image of $b := [z]$ under the Newton map, which is mentioned next.

The Kottwitz set $B(G)$ has two important invariants. In the non-Archimedean and complex cases, these invariants completely determine the set $B(G)$.

The first invariant is the Kottwitz map 
\begin{equation*}
    \kappa_G: B(G) \to X^{\ast}(Z(\widehat{G})^{\Gamma}) \cong \pi_1(G)_{\Gamma},
\end{equation*}
where $(-)^\Gamma$ (resp.\ $(-)_\Gamma$) denotes the $\Gamma$-invariants (resp.\ $\Gamma$-coinvariants) (\cite[\S11]{Kot9}). When $G$ is a torus, the Kottwitz map is bijective (see \cite[\S13.2]{Kot9}).

The second invariant is the Newton map \cite[\S10.7]{Kot9}
\begin{equation*}
    \nu_G: B(G) \to \mf{A}_T,
\end{equation*}
which takes image in $\ov{C}_{\Q}$. 

In the non-Archimedean case, this is constructed by noting that by definition of an algebraic cocycle, the restriction of $z \in Z^1_{\alg}(\mc{E}^{\iso}_F, G(\ov{F}))$ to $\D_F(\ov{F})$ is induced from a homomorphism $\nu_z: \D_F \to G$ defined over $\ov{F}$ with $\Gamma$-invariant $G(\ov{F})$-conjugacy class. Modifying $z$ by a coboundary has the effect of conjugating $z$ by an element of $G(\ov{F})$, so we get that $[z] \mapsto [\nu_z] \in (\Hom_{\ov{F}}(\D_F, G) / G(\ov{F}))^{\Gamma}$, which corresponds to a unique element of $\ov{C}_{\Q}$. In the Archimedean case, we have an analogous construction with $\Gm$ taking the role of $\D_F$.

We define $B(G)_{\bas} \subset B(G)$ to be the preimage of $\mf{A}_G$ under the Newton map $\nu_G$ and recall that in the non-Archimedean case, $\kappa_G$ induces a bijection $B(G)_{\bas} \cong X^{\ast}(Z(\widehat{G})^{\Gamma})$ \cite[Proposition 13.1.(1)]{Kot9}.

In the real case, $\kappa_G|_{B(G)_{\bas}}$ is no longer injective or surjective so $B(G)_{\bas}$ is parametrized differently. Recall a fundamental torus of $G$ is defined to be a maximal torus of minimal split rank. Suppose $S \subset G$ is a fundamental torus. Then $B(S)_{\Gbas}$ is defined to be the subset of $B(S)$ whose image under $B(S) \to B(G)$ lies in $B(G)_{\bas}$. The map $B(S)_{\Gbas} \to B(G)_{\bas}$ is surjective (see \cite[Lemma 13.2]{Kot9} and its proof) and induces a bijection $B(S)_{\Gbas} / W_G(S)^{\Gamma} \cong B(G)_{\bas}$, where $W_G(S)$ denotes the Weyl group of $S$ in $G$.

Recall that for each standard Levi subgroup $M$, there is a map $X_{\ast}(A_M) \to X^{\ast}(A_{\wh{M}})$ given by 
\begin{equation*}
    X_{\ast}(A_M) \hookrightarrow X_{\ast}(Z(M)^{\circ}) \cong X^{\ast}(\wh{M}_{\ab}) = X^{\ast}(\wh{M}) \xrightarrow{\res} X^{\ast}(A_{\wh{M}}),
\end{equation*}
which induces an isomorphism after taking the tensor product with $\R$.
We write $\alpha_M$ for the inverse of this isomorphism:
\begin{equation}{\label{eqn: BGiso}}
    \alpha_M: X^{\ast}(Z(\widehat{M})^{\Gamma})_{\R} \xrightarrow{\sim} \mf{A}_M \subset \mf{A}_T,
\end{equation}
where we note that the restriction map induces an isomorphism $X^{\ast}(Z(\widehat{M})^{\Gamma})_{\R} \xrightarrow{\sim} X^{\ast}(A_{\wh{M}})_{\R}$.
We remark that the restriction of the Newton map $\nu_G$ on $B(G)_\bas$ is given by the composition of $\kappa_G$ and $\alpha_G$ (see \cite[Proposition 11.5]{Kot9}, cf.\ \cite[\S4.4]{KottwitzisoII}):
\begin{equation}\label{eqn: Kottwitz-Newton}
\begin{tikzcd}
B(G)_\bas \ar[r, swap, "\kappa_G"] \ar[rrr, bend left=15, "\nu_G"] & X^\ast(Z(\wh{G})^\Gamma) \ar[r] & X^\ast(Z(\wh{G})^\Gamma)_\R \ar[r, swap, "\alpha_G"]& \mf{A}_G
\end{tikzcd}
\end{equation}
(See also Remark \ref{rem: Kot97-14}.)

For any standard parabolic subgroup $P$ with Levi decomposition $P=MN$ such that $M\supset T$ (i.e., $M$ is a standard Levi subgroup), we put
\begin{equation*}
\mf{A}_P^+
:=
\{\mu\in\mf{A}_M \mid \text{$\langle\alpha,\mu\rangle>0$ for any root of $T$ in $N$}\}.
\end{equation*}

Then we have the decomposition
\begin{align}\label{eqn: chamberdecomp}
\overline{C}
=
\coprod_{P}\mf{A}_P^+,
\end{align}
where the index is the set of standard parabolic subgroups of $G$.
We define the subset $B(G)_P$ of $B(G)$ to be the preimage of $\mf{A}_P^+$ under the Newton map.
This gives the decomposition
\begin{align}{\label{eqn: BGdecomp}}
B(G)
=
\coprod_{P}B(G)_P.
\end{align}
Note that $B(G)_G=B(G)_\bas$.
For a general standard parabolic $P=MN$, $B(G)_P$ has the following description.
By noting that the image of the Newton map $\nu_M|_{B(M)_{\bas}}$ lies in $\mf{A}_M$, we define the ``$G$-dominant'' subset $B(M)_\bas^+$ of $B(M)_{\bas}$ by
\[
B(M)_\bas^+
:=
\{b\in B(M)_\bas \mid \nu_M(b)\in\mf{A}_P^+\}.
\]
Then the canonical map $B(M)\rightarrow B(G)$ induces a bijection $B(M)_\bas^+\xrightarrow{1:1}B(G)_P$ (see \cite[\S5.1]{KottwitzisoII} for the non-Archimedean case). Indeed, given a $b \in B(G)_P$, we choose a cocycle representative $z$ whose restriction to $\Gm$ or $\D_F$ is equal to $\nu_G(b)$. Then $z$ will factor through the centralizer of $\nu_G(b)$, which is $M$. Hence, it suffices to prove the injectivity. But if $z_1, z_2 \in B(M)^+_{\bas}$ are conjugate by some $g \in G(\ov{F})$, then we can assume their restrictions to $\Gm$ or $\D_F$ are equal. Then $g$ centralizes this restriction so lies in $M$.

\begin{remark}\label{rem: Kot97-14}
Let us give some comments on the difference between the convention used in \cite{Kot9} and ours (which is closer to the one in \ \cite{KottwitzisoII}).
Recall that $\D_F$ is defined to be $\varprojlim_{K/F}\Gm$, where the projective limit is taken over the directed set of finite Galois extensions $K/F$ and the transition map for $L\supset K\supset F$ is given by $\Gm\rightarrow\Gm\colon z\mapsto z^{[L:K]}$.
Thus the character group $X^\ast(\D_F)$ is given by $\varinjlim_{K/F}\Z$, where the transition map for $L\supset K\supset F$ is given by $\Z\rightarrow\Z\colon x\mapsto [L:K]x$.
The point is that we have a natural injective map
\begin{align}\label{eq: pro-torus}
X^\ast(\D_F)
=\varinjlim_{K/F}\Z
\cong \varinjlim_{K/F}\frac{1}{[K:F]}\Z
\hookrightarrow\Q,
\end{align}
where the middle isomorphism is given by $\Z\rightarrow \frac{1}{[K:F]}\Z\colon x\mapsto \frac{x}{[K:F]}$ at each $K/F$ and the last map is the one induced from the inclusion $\frac{1}{[K:F]}\Z\hookrightarrow \Q$ (note that the transition maps of $\varinjlim_{K/F}\frac{1}{[K:F]}\Z$ are natural inclusions).

\begin{enumerate}
\item 
In \cite[\S11.5]{Kot9}, the target of the Kottwitz map is given by 
\[
A(F,G):=\varinjlim_{K/F} X^\ast(Z(\wh{G}))_{\Gal(K/F)}.
\]
Here, the limit is taken over the directed set of finite Galois extensions $K/F$ such that the action of $\Gamma$ on $X^\ast(Z(\wh{G}))$ factors through $\Gal(K/F)$ and the transition maps are the isomorphisms induced from the identity maps.
Thus we naturally have $A(F,G)\cong X^\ast(Z(\wh{G}))_\Gamma \,(\cong X^\ast(Z(\wh{G})^\Gamma))$.
\item
In \cite[\S1.4.1]{Kot9}, the target of the Newton map restricted to the basic part is given by $(X^\ast(\wh{G}_\ab)\otimes X^\ast(\D_F))^\Gamma$.
Let us write 
\[
\nu'_G|_{B(G)_\bas}\colon B(G)_\bas\rightarrow(X^\ast(\wh{G}_\ab)\otimes X^\ast(\D_F))^\Gamma
\]
for this map in order to emphasize the difference of the conventions.
Using the above identification \eqref{eq: pro-torus}, we have
\[
(X^\ast(\wh{G}_\ab)\otimes X^\ast(\D_F))^\Gamma
\hookrightarrow
X^\ast(\wh{G}_\ab)_\Q^\Gamma
\cong
X_\ast(Z(G)^{\circ})_\Q^\Gamma
\cong
X_\ast(A_G)_\Q.
\]
Then our Newton map $\nu_G|_{B(G)_\bas}$ is nothing but the composition of $\nu'_G|_{B(G)_\bas}$ with the inclusion $(X^\ast(\wh{G}_\ab)\otimes X^\ast(\D_F))^\Gamma\hookrightarrow X_\ast(A_G)_\Q$.
\item 
In \cite[Definition 11.3]{Kot9}, a map 
\[
N\colon A(F,G)\rightarrow (X^\ast(Z(\wh{G}))\otimes X^\ast(\D_F))^\Gamma
\]
is constructed by taking the inductive limit of the norm map 
\[
N_{K/F}\colon X^\ast(Z(\wh{G}))_{\Gal(K/F)}\rightarrow X^\ast(Z(\wh{G}))^{\Gal(K/F)}
\]
given by $\sum_{\sigma\in\Gal(K/F)}\sigma$ at each finite level.
Note that, if we compose the map $N$ with the above identification \eqref{eq: pro-torus} (and also $A(F,G)\cong X^\ast(Z(\wh{G}))_\Gamma$), the resulting map $X^\ast(Z(\wh{G}))_\Gamma\rightarrow X^\ast(Z(\wh{G}))^\Gamma_\Q$ is given by $\frac{1}{[K:F]}\sum_{\sigma\in\Gal(K/F)}\sigma$ at each finite level.
Thus, by furthermore composing it with the quotient map $X^\ast(Z(\wh{G}))^\Gamma_\Q\rightarrow X^\ast(Z(\wh{G}))_{\Q,\Gamma}$, we get the natural map $X^\ast(Z(\wh{G}))_{\Gamma}\rightarrow X^\ast(Z(\wh{G}))_{\Q,\Gamma}$.
\item 
In fact, \cite[Proposition 11.5]{Kot9} mentioned before asserts that $N\circ \kappa_G$ is equal to $i\circ\nu'_G|_{B(G)_\bas}$, where $i$ denotes the natural map $(X^\ast(\wh{G}_\ab)\otimes X^\ast(\D_F))^\Gamma\rightarrow (X^\ast(Z(\wh{G}))\otimes X^\ast(\D_F))^\Gamma$.
By putting all the above observations into together, we obtain the assertion as in \eqref{eqn: Kottwitz-Newton} (after furthermore changing the coefficients from $\Q$ to $\R$).
\end{enumerate}
The situation can be summarized as follows:
\[
\begin{tikzcd}
B(G)_\bas \arrow[r, "\kappa_G"] \arrow[d, "\nu'_G|_{B(G)_\bas}"] \ar[ddd, bend right=90, "\nu_G|_{B(G)_\bas}", swap]& X^{\ast}(Z(\wh{G}))_\Gamma \arrow[d, "N"] \ar[ddd, bend left=80, "(-)\otimes\Q"]\\
(X^\ast(\wh{G}_\ab)\otimes X^\ast(\D_F))^\Gamma  \arrow[r, "i"] \arrow[d, hook]& (X^\ast(Z(\wh{G}))\otimes X^\ast(\D_F))^\Gamma \arrow[d, hook]\\
X^\ast(\wh{G}_\ab)_\Q^\Gamma \arrow[r] \ar[d, "\cong"]&  X^{\ast}(Z(\wh{G}))_{\Q}^\Gamma \arrow[d, two heads]\\
X_\ast(A_G)_\Q \ar[r, "\alpha_{G}^{-1}"]&  X^{\ast}(Z(\wh{G}))_{\Q,\Gamma} 
\end{tikzcd}
\]
The top square commutes by \cite[Proposition 11.5]{Kot9}.
It can be easily seen that the middle and bottom squares also commute.
Thus we get the commutativity of the outer big square, as stated in \eqref{eqn: Kottwitz-Newton}.
\end{remark}

\subsection{Representation theory of disconnected reductive groups}{\label{ss: repthrydisconnected}}
We now briefly recall the theory of algebraic representations of disconnected reductive groups as in \cite{AcharHardestyRicherepthrydisconnected}. 
For us, a disconnected reductive group is an algebraic group $\ms{G}$ whose identity component $\ms{G}^{\circ}$ is reductive. 
For an $L$-parameter $\phi$ of $G$, the group $S_{\phi}$ is disconnected reductive (see Lemma \ref{lem: S-group-Levis}) and we need to understand the algebraic representations of these groups.
For this reason, we always assume in this section our groups are defined over $\C$ and only consider $\C$-valued representations.

Suppose $\ms{G}$ is disconnected reductive and fix a maximal torus $\ms{T}$ and Borel subgroup $\ms{B}$ of $\ms{G}^{\circ}$ such that $\ms{T} \subset \ms{B} \subset \ms{G}^{\circ}$. 
We put $W_{\ms{G}}(\ms{T}):=N_{\ms{G}}(\ms{T})/\ms{T}$ and $W_{\ms{G}}(\ms{T},\ms{B}):=N_{\ms{G}}(\ms{T},\ms{B})/\ms{T}$, where $N_{\ms{G}}(\ms{T},\ms{B}):=\{n\in \ms{G} \mid \co{n}{(\ms{T},\ms{B})}=(\ms{T},\ms{B})\}$. 

\begin{lemma}\label{lem: pi0embedweyl}
\begin{enumerate}
\item
We have a canonical bijection $\pi_0(\ms{G})\xrightarrow{\cong}W_{\ms{G}}(\ms{T},\ms{B})$.
\item
We have $W_\ms{G}(\ms{T})=W_{\ms{G}^\circ}(\ms{T})\rtimes W_\ms{G}(\ms{T},\ms{B})$.
\end{enumerate}
\end{lemma}

\begin{proof}
Let us first show (1).
For $\ov{g} \in \pi_0(\ms{G})$, we can choose a representative $g \in \ms{G}$. 
Then the conjugation map $\Int(g)$ takes $(\ms{T},\ms{B})$ to some pair $\co{g}(\ms{T}, \ms{B})$. 
All pairs are conjugate in $\ms{G}^{\circ}$ so we can find some $g^{\circ} \in \ms{G}^{\circ}$ such that $\Int(g^{\circ})$ takes $\co{g}(\ms{T},\ms{B})$ to $(\ms{T},\ms{B})$. 
Then we let $\ov{g}$ act on $\ms{T}$ by $\Int(g^{\circ}g)$ where we have $g^{\circ}g \in N_{\ms{G}}(\ms{T},\ms{B})$. 
Any two such $g^{\circ}$ differ by an element of $\ms{T}$ so this indeed gives a well-defined action. 
Suppose that $g_{1}, g_2 \in \ms{G}$ give the same element of $W_{\ms{G}}(\ms{T},\ms{B})$.
Then we have elements $g^{\circ}_{1},g^{\circ}_{2}\in \ms{G}^{\circ}$ and $t\in \ms{T}$ satisfying $g_{1}^{\circ}g_{1}=g_{2}^{\circ}g_{2}t$.
This means that $g_{1}$ and $g_{2}$ are equal in $\pi_{0}(\ms{G})$.
The surjectivity of the map is obvious.

We next show (2).
Since $\ms{G}^\circ$ is normal in $\ms{G}$, so is $W_{\ms{G}^\circ}(\ms{T})$ in $W_\ms{G}(\ms{T})$.
We have $W_{\ms{G}^\circ}(\ms{T})\cap W_{\ms{G}}(\ms{T},\ms{B})=W_{\ms{G}^\circ}(\ms{T},\ms{B})=\{1\}$.
Thus it is enough only to show that any element of $W_{\ms{G}}(\ms{T})$ can be written as a product of elements of $W_{\ms{G}^\circ}(\ms{T})$ and $W_{\ms{G}}(\ms{T},\ms{B})$. Choose $w \in W_{\ms{G}}(\ms{T})$. Then $\co{w}(\ms{T}, \ms{B}) = (\ms{T}, \co{w}{\ms{B}})$ and we can choose $w_0 \in W_{\ms{G}^{\circ}}(\ms{T})$ such that $\co{w_0}{\ms{B}} = \co{w}{\ms{B}}$. Then $w^{-1}_0w \in W_{\ms{G}}(\ms{T},\ms{B})$, which concludes the proof.
\end{proof}

For each dominant $\lambda \in X^{\ast}(\ms{T})^+$, we let $\mc{L}(\lambda)$ denote the irreducible algebraic representation of $\ms{G}^{\circ}$ with highest weight $\lambda$. 
We define $A^{\lambda} \subset \pi_0(\ms{G})$ to be the stabilizer of $\lambda$ under the action just described. 
We let $\ms{G}^{\lambda}$ be the pre-image in $\ms{G}$ of $A^{\lambda}$. 
For each $a\in A^\lambda$, we fix a representative $\iota(a)$ of $a$ in $\ms{G}^\lambda$ and a $\ms{G}^\circ$-equivariant isomorphism
\[
\theta_a \colon \mc{L}(\lambda)\xrightarrow{\cong}{}^{\iota(a)}\mc{L}(\lambda),
\]
such that $\iota(1)=1$ and $\theta_{1}=\mathrm{id}$.
Then the data $\{\theta_a\}_{a\in A^\lambda}$ defines a $2$-cocycle $\alpha(-,-)\colon A^\lambda\times A^\lambda\rightarrow\C^\times$ (see \cite[\S 2.4]{AcharHardestyRicherepthrydisconnected} for the details).
We define a twisted group algebra $\mc{A}^{\lambda}$ to be the $\C$-vector space $\C[A^{\lambda}]$ spanned by symbols $\{\rho_a\mid a\in A^\lambda\}$ with multiplication given by $\rho_a\cdot\rho_b=\alpha(a,b)\rho_{ab}$ for $a,b\in A^\lambda$.
%Then \cite{AcharHardestyRicherepthrydisconnected} prove that $\End_{\ms{G}^{\lambda}}(\Ind^{\ms{G}^{\lambda}}_{\ms{G}^{\circ}} \mc{L}(\lambda)) \cong (\mc{A}^{\lambda})^{\op}$ (see \cite[Proposition 2.6]{AcharHardestyRicherepthrydisconnected}).

For each simple $\mc{A}^{\lambda}$-module $E$, we have an irreducible representation $\mc{L}(\lambda,E)$ of $\ms{G}$ given by $\Ind^{\ms{G}}_{\ms{G}^{\lambda}} (E \otimes\mc{L}(\lambda))$. 
Here, the $\ms{G}^\lambda$-module structure on $E\otimes\mc{L}(\lambda)$ is given by
\[
(\iota(a)g)\cdot (u\otimes v)
=
(\rho_{a}u)\otimes(\theta_{a}^{-1}(gv))
\]
for $a\in A^\lambda$ and $g\in \ms{G}^\circ$.

An $a \in \pi_0(\ms{G})$ induces an isomorphism $\mc{L}(\lambda, E) \cong \mc{L}(\co{a}\lambda,\co{a}E)$, for a certain simple $\mc{A}^{\co{a}\lambda}$-module $\co{a}E$, and we have the following theorem.
\begin{theorem}[{\cite[Theorem 2.16]{AcharHardestyRicherepthrydisconnected}}]{\label{thm: acharmainthm}}
There is a bijection
\begin{equation*}
    \{(\lambda, E)\}/\pi_0(\ms{G}) \leftrightarrow \Irr(\ms{G}),
\end{equation*}
given by $(\lambda, E) \mapsto \mc{L}(\lambda, E)$, where $\{(\lambda, E)\}$ denotes the set of pairs of $\lambda \in X^{\ast}(\ms{T})^+$ and an isomorphism class of simple $\mc{A}^{\lambda}$-modules $E$ and $\Irr(\ms{G})$ denotes the set of isomorphism classes of irreducible algebraic representations of $\ms{G}$.
\end{theorem}

%\begin{definition}{\label{def: highestweightnumber}}
%Let $\rho = \mc{L}(\lambda, E)$ be an irreducible representation of $\ms{G}$. We define the \emph{highest weight number}, $h_\rho$, by
%\begin{equation*}
%   h_\rho :=
%|W_{\ms{G}}(\ms{T})\cdot\lambda|\cdot\dim(E).
%\end{equation*}
%\end{definition}

\begin{lemma}\label{lem: disconn-irrep-paramet}
The set $\{(\lambda, E)\}/\pi_0(\ms{G})$ can be identified with the set
\[
\coprod_{\lambda \in X^{\ast}(\ms{T})^+/W_\ms{G}(\ms{T},\ms{B})}
\{\text{$E$: simple $\mc{A}^{\lambda}$-module}\}/{\cong},
\]
where the index set is over a(ny) complete set of representatives of $X^{\ast}(\ms{T})^+/W_\ms{G}(\ms{T},\ms{B})$ and each summand is the set of isomorphism classes of simple $\mc{A}^\lambda$-modules.
\end{lemma}

\begin{proof}
We first note that if two dominant characters $\lambda_1$ and $\lambda_2$ satisfy $\lambda_2=w\cdot\lambda_1$ for $w\in W_{\ms{G}^\circ}(\ms{T})$, then we must have $\lambda_1=\lambda_2$ (see, e.g., \cite[Lemma 10.3.B]{Hum3}).
Thus, by Lemma \ref{lem: pi0embedweyl}, we have $X^{\ast}(\ms{T})^+/\pi_0(\ms{G})=X^{\ast}(\ms{T})^+/W_\ms{G}(\ms{T},\ms{B})$.
By fixing a complete set of representatives of $X^{\ast}(\ms{T})^+/W_\ms{G}(\ms{T},\ms{B})$, we get a surjective map
\[
\coprod_{\lambda \in X^{\ast}(\ms{T})^+/W_\ms{G}(\ms{T},\ms{B})}
\{\text{$E$: simple $\mc{A}^{\lambda}$-module}\}/{\cong}
\twoheadrightarrow
\{(\lambda, E)\}/\pi_0(\ms{G})
\colon E\mapsto (\lambda,E).
\]
Let us consider the fibers of this map.
For any simple $\mc{A}^{\lambda_1}$-module $E_1$ and simple $\mc{A}^{\lambda_2}$-module $E_2$, $(\lambda_1,E_1)$ and $(\lambda_2,E_2)$ are equivalent under the $\pi_0(\ms{G})$-action if and only if $\lambda_2=\lambda_1$ (as $\lambda_1$ and $\lambda_2$ are representatives of $X^{\ast}(\ms{T})^+/\pi_0(\ms{G})=X^{\ast}(\ms{T})^+/W_\ms{G}(\ms{T},\ms{B})$) and $E_2\cong {}^{w}E_1$ for some $w\in \pi_0(\ms{G})$ stabilizing $\lambda_1$.
Since $\Stab_{\pi_0(\ms{G})}(\lambda_1)=A^{\lambda_1}$ by definition, we have ${}^{w}E_1\cong E_1$.
In other words, the above map is in fact bijective.
\end{proof}

%We will also need the following.
%\begin{proposition}[{\cite[\S2.6]{AcharHardestyRicherepthrydisconnected}}]{\label{prop: Gcircdecomp}}
%Let $g_1, ..., g_r$ be representatives in $\ms{G}$ of $\ms{G}/\ms{G}^{\lambda}$. 
%As a representation of $\ms{G}^{\circ}$, we have 
%\begin{equation*}
%    \mc{L}(\lambda, E) \cong \bigoplus\limits_i \mc{L}(\co{g_i}\lambda)^{\oplus \dim E}.
%\end{equation*}
%\end{proposition}

\subsection{\texorpdfstring{$S$}{S}-groups of \texorpdfstring{$L$}{L}-parameters as disconnected reductive groups}{\label{ss: S-group}}

Let $\phi\colon L_F \to {}^LG$ be an $L$-parameter of $G$, where $L_F = W_F \times \SL_2$ in the non-Archimedean case and $W_F$ in the Archimedean case.
Let ${}^{L}M$ be a smallest Levi subgroup of ${}^{L}G$ such that $\phi$ factors through the $L$-embedding ${}^{L}M\hookrightarrow{}^{L}G$.
Up to replacing $\phi$ with a conjugate, we can and do assume that ${}^{L}M$ is a standard Levi subgroup.
Let $M$ be the $F$-rational standard Levi subgroup of $G$ which corresponds to ${}^{L}M$. 
Then, $\phi$ is discrete as an $L$-parameter of $M$.
For each $F$-rational standard Levi subgroup $L$ containing $M$, we define $S_{\phi,L}:=Z_{\wh{L}}(\im\phi)$.

\begin{lemma}{\label{lem: S-group-Levis}}
\begin{enumerate}
\item
The group $S_\phi^\circ$ is a connected reductive group.
\item
We have $S_{\phi,M}^{\circ}=A_{\widehat{M}}$ and this is a maximal torus of $S_\phi^\circ$.
\item
For any $F$-rational standard Levi subgroup $L$ containing $M$, the group $S_{\phi,L}^\circ$ is a Levi subgroup of $S_\phi^\circ$ and satisfies $S_{\phi,L}^\circ=S_{\phi,L}\cap S_\phi^\circ$.
\end{enumerate}
\[
\begin{tikzcd}
S_{\phi,M}\ar[r,hook]&S_{\phi,L}\ar[r,hook]&S_\phi\\
S_{\phi,M}^\circ=A_{\wh{M}}\ar[r,hook]\ar[u,hook]&S_{\phi,L}^\circ\ar[u,hook]\ar[r,hook]&S_\phi^\circ\ar[u,hook]
\end{tikzcd}
\]

\end{lemma}

\begin{proof}
See \cite[10.1.1, Lemma]{Kot5} (and also a comment in \cite[\S 12, p.648]{Kot5}) for the assertion (1).

The equality $S_{\phi,M}^{\circ}=A_{\widehat{M}}$ follows from the fact that $\phi$ is discrete as an $L$-parameter of $M$ (see \cite[10.3.1, Lemma]{Kot5}).
We note that $\wh{L}=Z_{\wh{G}}(A_{\wh{L}})$ (see \cite[\S0.4.1]{KMSW}).
We have
\[
S_{\phi,L}^\circ
=
(\wh{L}\cap S_\phi^\circ)^\circ
=
\bigl(Z_{\wh{G}}(A_{\wh{L}})\cap S_\phi^\circ\bigr)^\circ
=
Z_{S_\phi^\circ}(A_{\wh{L}})^\circ.
\]
As $S_\phi^\circ$ is a connected reductive group, the centralizer $Z_{S_\phi^\circ}(A_{\wh{L}})$ of a torus $A_{\wh{L}}$ is a Levi subgroup of $S_\phi^\circ$ (in particular, connected).
This also shows that $A_{\wh{M}}$ is a maximal torus of $S_\phi^\circ$.

Let us finally verify the equality $S_{\phi,L}^\circ=S_{\phi,L}\cap S_\phi^\circ$.
The inclusion $S_{\phi,L}^\circ \subset S_{\phi,L}\cap S_\phi^\circ$ is obvious, so it suffices to check the converse inclusion $S_{\phi,L}^\circ \supset S_{\phi,L}\cap S_\phi^\circ$.
For this, it is enough to show that $S_{\phi,L}\cap S_\phi^\circ$ is connected.
We have
\[
S_{\phi,L}\cap S_\phi^\circ
=
(\wh{L}\cap S_\phi) \cap S_\phi^\circ
=
Z_{\wh{G}}(A_{\wh{L}}) \cap S_\phi^\circ
=
Z_{S_\phi^\circ}(A_{\wh{L}}).
\]
Thus $S_{\phi,L}\cap S_\phi^\circ$ is connected as shown above.
\end{proof}

Note that Lemma \ref{lem: S-group-Levis} implies that for the fixed $L$-parameter $\phi$, the Levi subgroup $M$ is determined canonically up to conjugation.
Indeed, suppose that ${}^{L}M'$ is another smallest Levi subgroup of ${}^{L}G$ such that $\phi$ factors through ${}^{L}M'$.
Let us assume that ${}^{g}({}^{L}M)$ and ${}^{g'}({}^{L}M')$ are standard.
Then, by the above lemma, $A_{{}^{g}\wh{M}}$ and $A_{{}^{g'}\wh{M'}}$ are maximal tori of $S_{{}^{g}\phi}^\circ$ and $S_{{}^{g'}\phi}^\circ$, respectively.
Noting that ${}^{gg^{\prime-1}}S_{{}^{g'}\phi}^\circ=S_{{}^{g}\phi}^\circ$, both $A_{{}^{g}\wh{M}}$ and ${}^{gg^{\prime-1}}A_{{}^{g'}\wh{M'}}=A_{{}^{g}\wh{M'}}$ are maximal tori of $S_{{}^{g}\phi}^\circ$, hence conjugate by $S_{{}^{g}\phi}^\circ$.
This implies that $A_{\wh{M}}$ and $A_{\wh{M'}}$ are conjugate by $S_{\phi}^\circ$.
By using ${}^L{M}=Z_{{}^L{G}}(A_{\wh{M}})$ and ${}^L{M'}=Z_{{}^L{G}}(A_{\wh{M'}})$ (\cite[\S4.0.1]{KMSW}), we also see that ${}^{L}M$ and ${}^{L}M'$ are conjugate by $S_{\phi}^\circ$.
Thus, $M$ and $M'$ are conjugate in $G$.

\subsection{Weyl group constructions}{\label{ss: weylgroupconstructions}}

Let $\phi$ and $M$ be as in the previous section.
Suppose that $\lambda\in X^{\ast}(A_{\wh{M}})$ is given.
Then we have $\alpha_M(\lambda) \in \mf{A}_M\subset\mf{A}_T$ where $\alpha_M$ is the map of \eqref{eqn: BGiso}.
We take an element $w\in W^\rel$ such that $w\cdot\alpha_M(\lambda)\in\overline{C}\subset\mf{A}_T$. 
Let us write $M':={}^{w}M$ and $\lambda':={}^{w}\lambda$.
Thus we have $w\cdot\alpha_M(\lambda)=\alpha_{M'}(\lambda')$.
According to the decomposition \eqref{eqn: chamberdecomp}, there exists a unique standard parabolic subgroup $Q_\lambda$ of $G$ satisfying $\alpha_{M'}(\lambda')\in\mf{A}_{Q_\lambda}^+$.
We let $L_\lambda$ be the $F$-rational standard Levi subgroup of $G$ associated to $Q_\lambda$ (hence we have $T\subset M'\subset L_\lambda\subset G$).
Equivalently, $L_\lambda$ is the centralizer of an element of $X_\ast(A_T)$ given by some suitable scaling of $\alpha_{M'}(\lambda')$.
We simply write $Q$ and $L$ for $Q_\lambda$ and $L_\lambda$ in the following, respectively.
We note that the map $\alpha_T^{-1}\circ\alpha_{M'}\colon X^\ast(A_{\wh{M'}})_\R \hookrightarrow X^\ast(A_{\wh{T}})_\R$ gives a section to the restriction map $X^\ast(A_{\wh{T}})_\R \twoheadrightarrow X^\ast(A_{\wh{M'}})_\R$.
\begin{equation*}
    \begin{tikzcd}
    \mf{A}_{M'} \arrow[d, hook]  & X^{\ast}(A_{\wh{M'}})_{\R} \arrow[l, "\alpha_{M'}"]\\
    \mf{A}_{T} \arrow[r, "\alpha_{T}^{-1}"] &  X^{\ast}(A_{\wh{T}})_{\R} \arrow[u, "\res", swap]
    \end{tikzcd}
\end{equation*}
By furthermore noting that the isomorphism $\alpha_T$ is equivariant with respect to the action of $W^\rel\cong\wh{W}^\rel$, we get the following.

\begin{lemma}{\label{lem: lambdacentweyl}}
We have $\Stab_{W^\rel}(\alpha_{M'}(\lambda'))=W^\rel_L$ and $\Stab_{\wh{W}^\rel}(\alpha_T^{-1}\circ\alpha_{M'}(\lambda'))=\wh{W}^\rel_L$.
\end{lemma}

In the following, by choosing a representative $\dot{w}\in N_{\wh{G}}(A_{\wh{T}})$ of $w\in W^\rel\cong\wh{W}^\rel$ and replacing $\phi$ with ${}^{\dot{w}}\phi$, let us write $M$ and $\lambda$ for $M'$ and $\lambda'$, respectively.

We fix a Borel subgroup $B_{\phi} \subset S^{\circ}_{\phi}$ containing $A_{\widehat{M}}$.
We put
\begin{itemize}
    \item $W_\phi:=W_{S_\phi}(A_{\wh{M}}):=N_{S_{\phi}}(A_{\wh{M}})/A_{\wh{M}}$,
    \item $W_\phi^\circ:=W_{S_\phi^\circ}(A_{\wh{M}}):=N_{S_{\phi}^\circ}(A_{\wh{M}})/A_{\wh{M}}$,
    \item $R_\phi:=W_{S_\phi}(A_{\wh{M}},B_\phi):=N_{S_{\phi}}(A_{\wh{M}},B_\phi)/A_{\wh{M}}$.
\end{itemize}
Then, by Lemma \ref{lem: pi0embedweyl}, we have an identification $\pi_0(S_{\phi})\cong R_\phi$ and the semi-direct product decomposition $W_\phi=W_\phi^\circ\rtimes R_\phi$.
Note that we have a natural map
\begin{align}\label{map:Wphi}
W_\phi
=N_{S_{\phi}}(A_{\wh{M}})/A_{\wh{M}}
\rightarrow N_{\wh{G}}(A_{\wh{M}})/\wh{M}
=W_{\wh{G}}(A_{\wh{M}}).
\end{align}

\begin{lemma}{\label{lem: weylinj}}
We have a natural injective map $W_{\widehat{G}}(A_{\widehat{M}}) \hookrightarrow \wh{W}^\rel$.
Moreover, via this injection, the restriction map $X^\ast(A_{\wh{T}})_\R \twoheadrightarrow X^\ast(A_{\wh{M}})_\R$ is equivariant with respect the actions of $W_{\widehat{G}}(A_{\widehat{M}})$ on $X^\ast(A_{\wh{M}})_\R$ and $\wh{W}^\rel$ on $X^\ast(A_{\wh{T}})_\R$.
\end{lemma}

\begin{proof}
The construction of the injective map can be found in \cite[\S 0.4.3 and \S 0.4.7]{KMSW}.
For the sake of completeness, we explain it.
We note that $W_{\wh{G}}(A_{\wh{T}}) = W_{\wh{G}}(\wh{T})^{\Gamma}$ (see \cite[\S 0.4.3]{KMSW}) and that the same fact holds replacing $\wh{T}$ with an $F$-rational standard Levi subgroup of $\wh{G}$. We will first prove that we have an injective map $W_{\widehat{G}}(A_{\widehat{M}}) \hookrightarrow W_{\widehat{G}}(\widehat{T})$, and then show that this map is $\Gamma$-equivariant, which will finish the proof of the first assertion.

Set $\widehat{B}_{\widehat{M}}:=\widehat{B}\cap \widehat{M}$.
Then $(\widehat{T},\wh{B}_{\widehat{M}})$ is a Borel pair of $\widehat{M}$. 
Let $n\in N_{\widehat{G}}(A_{\wh{M}})$, hence we have ${}^{n}A_{\wh{M}}=A_{\wh{M}}$.
As $A_{\wh{M}}\subset {}^{n}\wh{T}$, we get $\wh{M}\supset {}^{n}\wh{T}$ by taking centralizers in $\wh{G}$.
Since ${}^{n}\wh{B}$ is a Borel subgroup of $\wh{G}$ containing ${}^{n}\wh{T}$, it follows that ${}^{n}\wh{B}\cap \wh{M} ={}^{n}\wh{B}_{\wh{M}}$ is a Borel subgroup of $\wh{M}$ containing ${}^{n}\wh{T}$.
Thus ${}^{n}(\wh{T},\wh{B}_{\wh{M}})$ is also a Borel pair of $\wh{M}$.
Hence there exists an element $m$ of $\wh{M}$ (unique up to right $\wh{T}$-multiplication) such that ${}^{m}(\wh{T},\wh{B}_{\wh{M}})={}^{n}(\wh{T},\wh{B}_{\wh{M}})$, which implies that $m^{-1}n\in N_{\wh{G}}(\wh{T})$.
In other words, we have obtained a well-defined map $N_{\wh{G}}(A_{\wh{M}})/\wh{M}\rightarrow N_{\wh{G}}(\wh{T})/\wh{T}$ (given by $n\mapsto m^{-1}n$).

Let us suppose that two elements $n_{1},n_{2}\in N_{\wh{G}}(A_{\wh{M}})$ map to the same element of $N_{\wh{G}}(\wh{T})/\wh{T}$.
By the definition of the map, this means that there exist $m_{1},m_{2}\in \wh{M}$ such that $m_{1}^{-1}n_{1}=m_{2}^{-1}n_{2}t$ with some $t\in \wh{T}$, or equivalently, $m_{1}^{-1}n_{1}=t'm_{2}^{-1}n_{2}$ with some $t'\in \wh{T}$.
In particular, we have $n_{2}n_{1}^{-1}=m_{2}t^{\prime-1}m_{1}^{-1}\in \wh{M}$.
Thus $n_{1}$ and $n_{2}$ are equal in $N_{\wh{G}}(A_{\wh{M}})/\wh{M}$.

We now prove $\Gamma$-equivariance. Fix $\gamma \in \Gamma$ and consider $\gamma(m^{-1}n) \in N_{\wh{G}}(\wh{T})$. It suffices to show that $\co{\gamma(m^{-1}n)}{\wh{B}} = \co{m^{-1}n}{\wh{B}}$. In fact, since $m^{-1}n$ preserves $\wh{B}_{\wh{M}}$, we need only show that $\co{\gamma(m^{-1}n)}{U_{\wh{P}}} = \co{m^{-1}n}{U_{\wh{P}}}$, where $U_{\wh{P}}$ denotes the unipotent radical of the standard parabolic $\wh{P}$ with Levi component $\wh{M}$.
But since $m^{-1}n$ gives a $\gamma$-invariant element of $W_{\wh{G}}(\wh{M})$, we have $\gamma(m^{-1}n)=m^{-1}nm'$ for some $m' \in \wh{M}$. Then the result follows from the fact that $\wh{M}$ normalizes $U_{\wh{P}}$.

By this construction, the second assertion for the restriction map is obvious.
\end{proof}
We also need the following.
\begin{lemma}{\label{lem: weylinj2}}
The map $\alpha_T^{-1}\circ\alpha_M\colon X^\ast(A_{\wh{M}})_\R \hookrightarrow X^\ast(A_{\wh{T}})_\R$ is equivariant with respect the action of $W_{\widehat{G}}(A_{\widehat{M}}) \hookrightarrow \wh{W}^\rel$.
\end{lemma}

\begin{proof}
Similarly to the previous lemma, it can be also checked that we have a natural inclusion $W_G(A_M)\hookrightarrow W^\rel$ and that the inclusion map $\mf{A}_M\hookrightarrow \mf{A}_T$ is equivariant with respect to the action of $W_G(A_M)\hookrightarrow W^\rel$.
Then the statement follows by checking that $\alpha_T$ (resp.\ $\alpha_M$) is equivariant with respect to the actions of $W^\rel\cong\wh{W}^\rel$ (resp.\ $W_G(A_M)\cong W_{\wh{G}}(A_{\wh{M}})$) and that the inclusions $W_G(A_M)\hookrightarrow W^\rel$ and $W_{\widehat{G}}(A_{\widehat{M}}) \hookrightarrow \wh{W}^\rel$ are consistent under the identifications $W^\rel\cong\wh{W}^\rel$ and $W_G(A_M)\cong W_{\wh{G}}(A_{\wh{M}})$.
\end{proof}

Following the notation of \S \ref{ss: repthrydisconnected}, we denote the stabilizer of $\lambda$ in $\pi_0(S_{\phi})$ by $A^{\lambda}$.
Here, recall that $\pi_0(S_\phi)$ acts on $X^\ast(A_{\wh{M'}})$ through the identification $\pi_0(S_\phi)\cong R_\phi$.
We denote the stabilizer of $\lambda$ in $\pi_0(S_{\phi, L})$ by $A^{\lambda}_L$.
We define the groups $W_{\phi,L}$, $W_{\phi,L}^\circ$, and $R_{\phi,L}$ in the same way as $W_{\phi}$, $W_{\phi}^\circ$, and $R_{\phi}$, respectively.
Note that $\pi_0(S_{\phi,L})$ can be regarded as a subgroup of $\pi_0(S_{\phi})$ by Lemma \ref{lem: S-group-Levis} (3).

\begin{proposition}{\label{prop: AlambdaL}}
We have $\pi_0(S_{\phi,L})= A^\lambda_L$ and the natural map $A^{\lambda}_L \hookrightarrow A^\lambda$ is surjective, hence bijective.
\[
\begin{tikzcd}
\pi_0(S_{\phi,L})\ar[r,hook]&\pi_0(S_\phi)\\
A_L^\lambda\ar[u,equal]\ar[r,hook,"="]&A^\lambda\ar[u,hook]
\end{tikzcd}
\]
\end{proposition}

\begin{proof}
Our task is to show that, for any $g\in \pi_0(S_\phi)$, $g$ stabilizes $\lambda$ if and only if $g\in\pi_0(S_{\phi,L})$.
By letting $w\in R_\phi$ be the image of $g\in \pi_0(S_\phi)$ under the identification $\pi_0(S_\phi)\cong R_\phi$, it suffices to check that $w$ stabilizes $\lambda$ if and only if $w\in R_{\phi,L}$.
We note that, by construction, the maps of Lemma \ref{lem: weylinj} for $G$ and $L$ are compatible.
\[
\begin{tikzcd}
\pi_0(S_\phi)\ar[r, "\sim"]&R_\phi\ar[r]&W_{\wh{G}}(A_{\wh{M}})\ar[r,hook]&\wh{W}^\rel\\
\pi_0(S_{\phi,L})\ar[r, "\sim"]\ar[u,hook]&R_{\phi,L}\ar[r]\ar[u,hook]&W_{\wh{L}}(A_{\wh{M}})\ar[u,hook]\ar[r,hook]&\wh{W}_L^\rel\ar[u,hook]
\end{tikzcd}
\]
Since the map \eqref{map:Wphi} is injective on $R_\phi$, it is enough to check that the image of $w$ in $W_{\wh{G}}(A_{\wh{M}})$ under the map \eqref{map:Wphi} (say $\overline{w}$) stabilizes $\lambda$ if and only if $\overline{w}$ lies in $W_{\wh{L}}(A_{\wh{M}})$.
If we let $\tilde{w}$ be the image of $\overline{w}\in W_{\wh{G}}(A_{\wh{M}})$ in $\wh{W}^\rel$, then $\overline{w}$ stabilizes $\lambda$ if and only if $\tilde{w}$ stabilizes $\alpha_T^{-1}\circ\alpha_{M}(\lambda)$ by Lemma \ref{lem: weylinj2}.
By Lemma \ref{lem: lambdacentweyl}, this is equivalent to $\tilde{w}\in \wh{W}_L^\rel$.
This completes the proof.
\end{proof}

\section{Review of the \texorpdfstring{$B(G)_{\bas}$}{B(G)bas} form of the conjectural correspondence.}\label{s: basiccorrespondence}
In this section we review the conjectural local Langlands correspondence parametrized in terms of $B(G)_{\bas}$ following \cite[\S2.5]{KalethaLLCnonQS}.
Recall that we fixed an $F$-splitting $(T,B,\{X_\alpha\})$ of $G$. Fix also a nontrivial additive character $\psi: F \to \C^{\times}$.
This defines a Whittaker datum for $G$ which we denote by $\mf{w}$.
For an $L$-parameter $\phi$ of $G$, we let $S_{\phi} = Z_{\widehat{G}}(\im \phi)$ and define $S^{\natural}_{\phi}$ to equal $S_{\phi}/(\widehat{G}_{\der} \cap S_{\phi})^{\circ}$.

The local Langlands correspondence with $B(G)_\bas$-parametrization is as follows:
\begin{conjecture}\label{conj: basic-LLC}
For each $b\in B(G)_\bas$, there exists a finite-to-one map
\[
\LLC_{G_b}\colon\Pi(G_b) \rightarrow \Phi(G),
\]
or, equivalently, a partition
\[
\Pi(G_b)
=
\coprod_{\phi\in\Phi(G)}\Pi_\phi(G_b),
\]
where $\Pi_\phi(G_b)$ denotes the finite set $\LLC_{G_b}^{-1}(\phi)$ (“$L$-packet").
Furthermore, for each $\phi\in\Phi(G)$, the union of $\Pi_\phi(G_b)$ over $b\in B(G)_\bas$ is equipped with a bijective map $\iota_{\mf{w}}$, depending only on the choice of a Whittaker datum $\mf{w}$, to $\Irr(S_\phi^\natural)$ such that the following diagram commutes:
\begin{equation}{\label{eqn: basic correspondence}}
    \begin{tikzcd}
    \coprod\limits_{b \in B(G)_{\bas}} \Pi_{\phi}(G_b) \arrow[d] \arrow[r, "\iota_{\mf{w}}"] & \Irr(S^{\natural}_{\phi}) \arrow[d] \\
    B(G)_{\bas} \arrow[r, "\kappa_G"]& X^{\ast}(Z(\widehat{G})^{\Gamma}),
    \end{tikzcd}
\end{equation}
where the left vertical map is the obvious projection and the right vertical map takes central character.
\end{conjecture}

In the following, we refer to Conjecture \ref{conj: basic-LLC} as “$B(G)_\bas$-LLC".

\begin{remark}\label{rem: tempered}
We note that in \cite{KalethaLLCnonQS}, Conjecture \ref{conj: basic-LLC} was stated for tempered $L$-parameters and that the proof of \cite[Theorem 2.5]{BMHN} shows that if Conjecture \ref{conj: basic-LLC} holds for all tempered $L$-parameters of each Levi subgroup of $G$, then it holds for all $L$-parameters of $G$.
\end{remark}

\subsection{The enhanced Archimedean basic correspondence}{\label{ss: enhancedbasicarchllc}}

In the Archimedean case, the Kottwitz map $\kappa_G$ is not injective.
Thus, when $\pi$ belongs to $\Pi_\phi(G_b)$ for $b \in B(G)_{\bas}$, Conjecture \ref{conj: basic-LLC} does not allow us to recover $b$ from the $Z(\wh{G})^\Gamma$-central character of $\iota_{\mf{w}}(\pi)$. We explain how to remedy this. Recall that when $F$ is an Archimedean local field, we have $W_F = \mc{E}^{\iso}_F$.

We first consider the simplest case when $F=\C$. Then the Newton map gives a bijection $\nu_G: B(G) \xrightarrow{\sim} X_*(T)^{+}$. An $L$-parameter is determined by two elements $\mu, \nu \in X_*(\wh{T})_{\C}$ such that $\mu - \nu \in X_*(\wh{T})$ via the formula $\phi(z) = z^{\mu} \ov{z}^{\nu}$. This implies the centralizer group $S_{\phi}$ is a Levi subgroup of $\wh{G}$ (\cite[Corollary 5.5]{VoganLLC}) and hence is connected. In particular, $S^{\natural}_{\phi} \cong \wh{G}_{\ab}$. The classical Langlands correspondence for $\C$ (see \cite[Theorem 5.3]{VoganLLC}) gives a bijection between $\Pi(G)$ and $\Phi(G)$. Since $B(G)_{\bas}$ is identified via $\nu_G$ with $X_*(A_G)$, which is canonically isomorphic to $X^*(\wh{G}_{\ab}) = \Irr(S^{\natural}_{\phi})$, we have the following commutative diagram, where every map is a bijection and the top horizontal arrow is defined to be the unique one such that the diagram commutes:
\begin{equation}{\label{eqn: complexbasicLLC}}
    \begin{tikzcd}
        \coprod\limits_{b \in B(G)_{\bas}} \Pi_{\phi}(G_b) \arrow[r] \arrow[d] & \Irr(S^{\natural}_{\phi}) \arrow[d] \\
        B(G)_{\bas} \arrow[r] & X^*(\wh{G}_{\ab}).
    \end{tikzcd}
\end{equation}

Now let $F=\R$. Let $\phi\colon W_F \to \co{L}{G}$ be an $L$-parameter. Let $M$ be a minimal Levi subgroup through which $\phi$ factors. By possibly replacing $\phi$ with a conjugate, we can assume $M$ is a standard Levi subgroup.  Then $\phi(W_F)$ normalizes a maximal torus of $\wh{M}$ (see \cite[pg.\ 126]{LanglandsClassification}), which we can assume is $\wh{T}$, again possibly replacing $\phi$ by a conjugate. We have an element $\mu \in X_*(\wh{T})_{\C}$ with $\mu - \phi(j)(\mu) \in X_*(\wh{T})$ such that $\phi(z) = z^{\mu} \ov{z}^{\phi(j)(\mu)} \rtimes z$ for $z \in \C^{\times}$, where $j \in W_{\R}$ projects to the nontrivial element of $\Gamma$ and satisfies $j^2 = -1$. The group $A_{\wh{M}}$ is a maximal torus of $S^{\circ}_{\phi}$ (see \S \ref{ss: S-group}) and we fix also a Borel subgroup $B_{\phi}$ of $S^{\circ}_{\phi}$ containing $A_{\wh{M}}$.

We explain first the discrete case where $G=M$  (our exposition parallels that of \cite[\S5.6]{Kal2}).
Then we have  $Z_{\wh{G}}(\phi(\C^{\times})) = \wh{T}$ (see \cite[Lemma 3.3]{LanglandsClassification}) and note that $\phi$ induces an action of $\Gamma$ on $\wh{T}$, which will in general be distinct to the given action of $\Gamma$. 
This data specifies an $\R$-rational torus $S$ whose dual is identified with $\wh{T}$ with the $\Gamma$-action coming from $\phi$. Our fixed Borel pair induces $(T,B)$ and gives us an embedding $S \to T \subset G$ defined over $\C$ whose $G(\C)$-conjugacy class is $\Gamma$-stable. Since $G$ is quasi-split, there exists an embedding $i: S \to G$ defined over $\R$ in this conjugacy class.

Now fix an inner twist $\varphi: G \to G'$. Since $i(S)$ is a fundamental torus, $\varphi \circ i$ has a $G'(\C)$-conjugate defined over $\R$ and we call such an embedding \emph{admissible}. Shelstad proves that the $L$-packet $\Pi_{\phi}(G')$ is in bijection with the set of $G'(\R)$-conjugacy classes of admissible embeddings $S \to G'$. 

Using Shelstad's bijection, we now show how to construct $\iota_{\mf{w}}$ in \ref{conj: basic-LLC}. Fix $b \in B(G)_{\bas}$ and choose a cocycle $z$ representing $b$ and let $(G_b, \varphi, z)$ be an extended pure inner twist. In particular, this means $\varphi: G \to G_b$ and $\varphi^{-1} \circ \gamma(\varphi) = \Int(z_e)$ for each $e \in \mc{E}^{\iso}_F$ projecting to $\gamma \in \Gamma$. There exists a unique $\mf{w}$-generic element $\pi_{\mf{w}}$ of the packet $\Pi_{\phi}(G)$ which corresponds to an embedding $i_{\mf{w}}: S \to G$.  Then choose any $\pi \in \Pi_{\phi}(G_b)$ and take its corresponding embedding $i_{\pi}: S \to G_b$. Then take $g \in G(\C)$ such that $i_{\pi} = \varphi \circ \Int(g) \circ i_{\mf{w}}$ and let $\inv[z](\pi_{\mf{w}}, \pi) \in B(S)_{\Gbas}$ be the cohomology class corresponding to the cocycle $e \mapsto i^{-1}_{\mf{w}}(g^{-1}z_ee(g))$. Then observe $B(S) = X^*(\wh{S}^{\Gamma_{\phi}})=X^*(\wh{T}^{\Gamma_{\phi}})$ (where the $\phi$-subscript reminds us that the invariants are with respect to the $\Gamma$-action induced by $\phi$). Since $
\phi$ is discrete, we have $S_{\phi} = S^{\natural}_{\phi}= \wh{T}^{\Gamma_{\phi}}$. So this gives an element of $\Irr(S^{\natural}_{\phi})$ that recovers $b \in B(G)_{\bas}$ via the map $B(S)_{\Gbas} \to B(G)_{\bas}$. Conversely, given an element of $B(S)$ whose image under $i_{\mf{w}}$ equals $b \in B(G)_{\bas}$, we get an admissible embedding $S \to G_b$ (proof analogous to \cite[Lemma 3.5]{BM3}). Finally, we claim that $B(S)_{\Gbas}$ is in bijection with $\Irr(S^{\natural}_{\phi})$ which follows from the fact that $B(S)_{\Gbas}= B(S)$ (since $i_{\mf{w}}(S) \subset G$ is an elliptic torus because $\phi$ is discrete).

We now explain how to handle the tempered case as in  \cite{ShelstadLindistinguishability}, following the notation of \cite[\S5.6]{Kal2}. 
\begin{remark}
    One could construct $S$ by taking $Z_{\wh{M}}(\phi(\C^{\times}))$ in analogy with the discrete case. However, this construction will in general give the ``wrong'' $\Gamma$-action on $\wh{S}$. A simple example of this is the parameter $\phi\colon W_{\R} \to \co{L}\SL_2$ where the composition of $\phi$ with the projection $\co{L}{\SL_2} \to \wh{\SL_2}$ has kernel equal to $\C^{\times}$ and $\phi(j) = \begin{pmatrix} 1 & 0 \\ 0 & -1 \end{pmatrix} \rtimes j$ where $j \in W_{\R}$ projects to the nontrivial element of $\Gamma$ and satisfies $j^2 = -1$. Then the ``naive'' construction of $S$ yields $\Gm$, but the construction we are about to describe produces $U(1)$. 
\end{remark}
 In the tempered case, Shelstad (\cite[\S5.3-\S5.4]{ShelstadLindistinguishability}) defines a Levi subgroup $M_1 \supset M$, an element $s \in \wh{G}$ and a parameter $\phi_1 = \Int(s) \circ \phi$ that is therefore equivalent in $\co{L}{G}$ to $\phi$ and such that $\phi_1$ is a limit of discrete series parameter for $M_1$. We have $\phi_1(W_{\R})$ normalizes $\wh{T}$ and $\phi_1(\C^{\times}) \subset \wh{T}$. Hence $\phi_1$ induces an action of $\Gamma$ on $\wh{T}$ which gives a torus $S$ which is elliptic in $M_1$.
 
 Using this, Shelstad proves that for each group $G'$ that is an inner form of $G$, there is an $L$-packet $\Pi_{\phi}(G')$ that is in bijection with the admissible embeddings $i: S \to G'$ such that $i(\Delta_{\phi})$ consists entirely of non-compact imaginary roots, where 
 \begin{equation*}
     \Delta_{\phi} = \{ \alpha \in X^*(S) \cong X_*(\wh{T}) \mid \alpha^{\vee} \in R(\wh{T}, \wh{G}), \langle  \mu,  \alpha^{\vee} \rangle = 0, \sum\limits_{r \in R_{\phi}} r\cdot \alpha^{\vee} = 0\},
 \end{equation*}
 ($R(\wh{T}, \wh{G})$ denotes the set of roots of $\wh{T}$ in $\wh{G}$).
We recall that the group $R_{\phi}$ acts on $X^\ast(\wh{T})$ through the map $R_\phi\rightarrow W_{\wh{G}}(A_{\wh{M}})\hookrightarrow \wh{W}^\rel$ (see \S \ref{ss: weylgroupconstructions}).

We give a few details on this construction. Fix an inner twist $\varphi: G \to G'$ as before and assume that $M$ transfers to some standard Levi $M'$ of $G'$ (if it does not, the $L$-packet will be trivial), and potentially change $\varphi$ by conjugation so that it restricts to an inner twist $\varphi: M \to M'$. Then the $\Gamma$-cocycle given by $\sigma \mapsto \varphi^{-1} \circ \sigma(\varphi)$ takes values in $M_{\ad}(\C)$ and hence it follows that if we define $M'_1 = \varphi(M_1)$, then $\varphi: M_1 \to M'_1$ is also an inner twist. Then for each admissible embedding $i: S \to M'_1$, we obtain a distribution on $M'_1$ by taking a limit at $\mu$ of the character formula for an essentially discrete series representation. Next, we take the parabolic induction to $G'$ and this is either $0$ or an irreducible character. The $L$-packet $\Pi_{\phi}(G')$ corresponds to the set of these characters which are in bijection with certain $M'_1(\R)$-conjugacy classes of admissible embeddings $i: S \to M'_1$. We claim the set of all $M'_1(\R)$-conjugacy classes of admissible embeddings is the same as the set of all $G'(\R)$-conjugacy classes of admissible embeddings. Indeed the former (resp.\ latter) set is in bijection with $\ker(H^1(\R, S) \to H^1(\R,M'_1))$ (resp.\ $\ker(H^1(\R, S) \to H^1(\R, G'))$) and it is a standard fact that $H^1(\R, M'_1) \hookrightarrow H^1(\R, G')$ (see, \cite[\S1.3.5]{Kestutistorsors} for instance). Thus, we have a bijection between $\Pi_{\phi}(G')$ and $G'(\R)$-conjugacy classes of admissible embeddings $i: S \to G'$ such that $i(\Delta_{\phi})$ consists of non-compact roots. 

We need to characterize the set of embeddings $i$ satisfying this non-compactness condition. There is a unique $\mf{w}$-generic constituent $I^G_{P_1}(\pi_{M_1, \mf{w}})$ of $\Pi_{\phi}(G)$. We let $i_{\mf{w}}: S \to G$ denote the corresponding embedding. Now let $b \in B(G)_{\bas}$ and choose an extended pure inner twist $(G_b, \varphi, z)$, where $z$ is an algebraic cocycle representing $b$. Now, $i_{\mf{w}}: S \to G$ is known to satisfy that $i_{\mf{w}}(\Delta_{\phi})$ consists of non-compact roots.
The condition we need on some embedding $i_{\pi}: S \to G_b$ is that the image of $\inv[z](\pi_{\mf{w}}, \pi) \in B(S)$ in $H^1(\R, S_{\ad}) \cong \pi_0(\wh{S_{\ad}}^{\Gamma_{\phi_1}})^{\vee}$ pairs to an even integer with each $\alpha^{\vee}$ such that $\alpha \in \Delta_{\phi}$. Indeed,  note that in the notation of \textit{loc.\ cit.}, a root $\alpha$ is non-compact relative to the embedding $S \to G$ if and only if $f_{(G,S)}(\alpha)=1$. Then by \cite[Proposition 4.3.(1)]{KalEpi} and using $\iota_{\mf{w}}$ as our base-point, we need only determine when $\kappa_{\alpha}(\eta_{t, \alpha})=1$. By \cite[Proposition 4.3.(2)]{KalEpi}, this is equivalent to our claimed expression (recalling that $\Gamma = \Gamma_{\pm \alpha}$ since the roots in question are symmetric). 
%By diagram \eqref{eqn: Kottwitz-Newton}, we have a diagram
%\begin{equation}{\label{eqn: realKSVS}}
%    \begin{tikzcd}
%        B(S) \arrow[r, "\kappa_S"] \arrow[d, swap, "\nu_S"] & X^*(\wh{S}^{\Gamma_{\phi_1}}) \arrow[dl, "N"] \\
%        X_*(S)^{\Gamma_{\phi_1}}
%    \end{tikzcd}
%\end{equation}
%and the set $B(S)_{\Gbas}$ is those elements of $B(S)$ whose image under $\nu_S$ factors through $X_*(A_G) \to X_*(S)^{\Gamma_{\phi_1}}$. In particular, the set $B(S)_{\Gbas}$ corresponds to the subgroup of $X^*(\wh{S}^{\Gamma_{\phi_1}})$ which is the pre-image under $N$ of $\im(X_*(A_G)) \subset X_*(S)^{\Gamma_{\phi_1}}$. By the anti-equivalence of categories between multiplicative groups and finitely generated abelian groups, we get a subgroup $\wh{S}_{\Gbas} \subset \wh{S}^{\Gamma_{\phi_1}}$ such that the elements of $B(S)_{\Gbas}$ correspond via $\kappa_S$ to the subset $X^*(\wh{S}^{\Gamma_{\phi_1}}/\wh{S}_{\Gbas})$ of elements of $X^*(\wh{S}^{\Gamma_{\phi_1}})$ that vanish on $\wh{S}_{\Gbas}$.

Note that $B(S)_{\Gbas}$ is those elements of $B(S)$ whose image under the Newton map $\nu'_S$ in the sense of \cite{Kot9} belongs to $(X^\ast(\wh{G}_\ab)\otimes X^\ast(\D_F))^\Gamma$ (see \cite[Definition 10.2]{Kot9} and also the discussion in Remark \ref{rem: Kot97-14}).
Thus, by the diagram \eqref{eqn: Kottwitz-Newton} and Remark \ref{rem: Kot97-14}, we have a diagram
\begin{equation}{\label{eqn: realKSVS}}
    \begin{tikzcd}
        B(S)_{\Gbas} \arrow[r, hook] \ar[d, "\nu'_S"]& B(S) \arrow[r, "\kappa_S"] \arrow[d, swap, "\nu'_S"] & X^*(\wh{S}^{\Gamma_{\phi_1}}) \arrow[dl, "N"] \\
        (X^\ast(\wh{G}_\ab)\otimes X^\ast(\D_F))^\Gamma \arrow[r, hook]& (X^\ast(\wh{S})\otimes X^\ast(\D_F))^\Gamma & 
    \end{tikzcd}
\end{equation}
In particular, the set $B(S)_{\Gbas}$ corresponds to the subgroup of $X^\ast(\wh{S}^{\Gamma_{\phi_1}})$ which is the pre-image under $N$ of $(X^\ast(\wh{G}_\ab)\otimes X^\ast(\D_F))^\Gamma$.
By the anti-equivalence of categories between multiplicative groups and finitely generated abelian groups, we get a subgroup $\wh{S}_{\Gbas} \subset \wh{S}^{\Gamma_{\phi_1}}$ such that the elements of $B(S)_{\Gbas}$ correspond via $\kappa_S$ to the subset $X^*(\wh{S}^{\Gamma_{\phi_1}}/\wh{S}_{\Gbas})$ of elements of $X^*(\wh{S}^{\Gamma_{\phi_1}})$ that vanish on $\wh{S}_{\Gbas}$.

Now for each $\alpha \in \Delta_{\phi}$, we get an element $\alpha^{\vee}(-1) \in \wh{S}$. The nontrivial element $\sigma \in \Gamma_{\R}$ is known to satisfy $\phi_1(\sigma)(\alpha) = - \alpha$ and so we have $\alpha^{\vee}(-1) \in \wh{S}^{\Gamma_{\phi_1}}$. Let $\Omega(\Delta_{\phi})$ be the group generated by the reflections $w_{\alpha}$ for $\alpha \in \Delta_{\phi}$. Then we define a map $\Omega(\Delta_{\phi}) \times \wh{S}_{\Gbas} \to \wh{S}^{\Gamma_{\phi_1}}$ where the map on the first factor is given by $w_{\alpha} \mapsto \alpha^{\vee}(-1)$ and the map on the second factor is the natural inclusion. Then it is clear that an embedding $i_{\pi}:S \to G_b$ satisfies that $i_{\pi}(\Delta_{\phi})$ are non-compact if and only if $\inv[z](\pi_{\mf{w}}, \pi) \in B(S) \cong X^*(\wh{S}^{\Gamma_{\phi_1}})$ vanishes on $\im(\Omega(\Delta_{\phi}) \times \wh{S}_{\Gbas})$.

\begin{lemma}{\label{lem: mainreallem}}
    We have an exact sequence
    \begin{equation*}
        \Omega(\Delta_{\phi}) \times \wh{S}_{\Gbas} \xrightarrow{r} \wh{T}^{\Gamma_{\phi_1}} \xrightarrow{p} S^{\natural}_{\phi} \to 1.
    \end{equation*}
\end{lemma}

\begin{proof}
We first construct the map $p: \wh{T}^{\Gamma_{\phi_1}} \to S^{\natural}_{\phi}$ and prove it is surjective. Recall \cite[Proposition 5.4.3]{ShelstadLindistinguishability}, that $\phi_1(\C^{\times}) \subset \wh{T}$ and that $\phi_1(W_{\R})$ normalizes $\wh{T}$. Hence, $S_{\phi_1} \cap \wh{T} = \wh{T}^{\Gamma_{\phi_1}}$. Shelstad proves (\cite[Theorem 5.4.4]{ShelstadLindistinguishability}) that we have a surjection $\wh{T}^{\Gamma_{\phi_1}} \twoheadrightarrow \pi_0({S_{\phi_1}})$. We also claim the natural map $\wh{T}^{\Gamma_{\phi_1}} \cap S^{\circ}_{\phi_1} \to S^{\circ}_{\phi_1} / (\wh{G}_{\der} \cap S_{\phi_1})^{\circ}$ is surjective. Indeed, it suffices to show that $Z(\wh{G})^{\Gamma} \cap S^{\circ}_{\phi_1}  \twoheadrightarrow S^{\circ}_{\phi_1} / (\wh{G}_{\der} \cap S_{\phi_1})^{\circ}$ and this follows from the fact (\cite[Lemma 0.4.13]{KMSW}) that $Z(\wh{G})^{\Gamma}$ surjects onto $S^{\circ}_{\phi_1}Z(\wh{G})^{\Gamma} / (\wh{G}_{\der} \cap S_{\phi_1})^{\circ}$. Now that the claim is proven, we can combine the two surjections to get a surjection $\wh{T}^{\Gamma_{\phi_1}} \to S^{\natural}_{\phi_1}$. Finally, we post-compose with $\Int(s^{-1})$ to get the desired map $p$.

Since $\wh{S}^{\Gamma_{\phi_1}} = \wh{T}^{\Gamma_{\phi_1}}$, the map $r$ is as constructed immediately before the statement of the lemma. It remains to prove exactness in the middle. We first show that $p \circ r$ is trivial. To do so, we let $\chi \in X^*(S^{\natural}_{\phi})$ and show the pullback to $\wh{T}^{\Gamma_{\phi_1}}$ vanishes on $\im(r)$. By conjugating by $s$, we get $\chi' \in X^*(S^{\natural}_{\phi_1})$. Then $\chi'$ by definition vanishes on $(\wh{G}_{\der} \cap S_{\phi_1})^{\circ}$ and hence $(\wh{G}_{\der} \cap \wh{T}^{\Gamma_{\phi_1}})^{\circ}$. Let $\wh{T}_{\der}$ denote the torus given by $(\wh{T} \cap \wh{G}_{\der})^{\circ}$. Then we have that $\chi'$ vanishes on $\wh{T}^{\Gamma_{\phi_1 , \circ}}_{\der}$.  We now have the following commutative diagram
\begin{equation}{\label{eqn: normderdiagram}}
    \begin{tikzcd}
        X^*(\wh{T}^{\Gamma_{\phi_1}}) \arrow[r, "\res"] \arrow[d, "N"]& X^*(\wh{T}^{\Gamma_{\phi_1}}_{\der})  \arrow[r, "\res"] \arrow[d, "N"] &  X^*(\wh{T}^{\Gamma_{\phi_1}, \circ}_{\der}) \\
        (X^*(\wh{T})\otimes X^\ast(\D_F))^{\Gamma_{\phi_1}} \arrow[r, "\res"] & (X^*(\wh{T}_\der)\otimes X^\ast(\D_F))^{\Gamma_{\phi_1}}.
    \end{tikzcd}
\end{equation}
We claim that the image of $\chi'$ in $(X^*(\wh{T}_\der)\otimes X^\ast(\D_F))^{\Gamma_{\phi_1}}$ is trivial. 
Indeed the restriction of $\chi'$ to $\wh{T}^{\Gamma_{\phi_1}}_{\der}$ is a character of $\pi_0(\wh{T}^{\Gamma_{\phi_1}}_{\der})$, and by the classification of tori over $\R$, the elements in the component group all have order $2$ and hence are killed by the norm map. 
Finally, we observe that $\wh{T}/\wh{T}_{\der}\cong \wh{G}/\wh{G}_\der=\wh{G}_\ab$. 
Hence it follows that $N(\chi')$ lies in $(X^\ast(\wh{G}_\ab)\otimes X^\ast(\D_F))^\Gamma$, which implies that $\chi'$ vanishes on $\wh{S}_\Gbas$.
Now we show that $\chi'$ vanishes on the image of $\Omega(\Delta_{\phi})$. But the image of this map also lies in $\wh{T}^{\Gamma_{\phi_1}}_{\der}$, so we are done.

Finally, we need to show that if $\chi \in X^*(\wh{T}^{\Gamma_{\phi_1}})$ vanishes on $\im(r)$, then it factors through $p$. Now, we have a surjection $\wh{T}^{\Gamma_{\phi_1}} \to S^{\natural}_{\phi_1}$ so it suffices to show that $\chi$ vanishes on $(\wh{G}_{\der} \cap S_{\phi_1})^{\circ} \cap \wh{T}^{\Gamma_{\phi_1}} = (\wh{G}_{\der} \cap \wh{T}^{\Gamma_{\phi_1}})^{\circ} = \wh{T}^{\Gamma_{\phi_1}, \circ}_{\der}$. We are assuming $\chi$ vanishes on $\wh{S}_{\Gbas}$ and so by the previous paragraph, $N(\chi)$ vanishes on $\wh{T}_{\der}$. Now, since all tori over $\R$ are a product of $\Gm, U(1), \Res_{\C/\R}\Gm$, we have that the center vertical norm map is injective when restricted to $\wh{T}^{\Gamma_{\phi_1},\circ}_{\der}$. Hence $\chi$ must vanish on $\wh{T}^{\Gamma_{\phi_1},\circ}_{\der}$ as desired.

\end{proof}

We define $W_G(S)^{\Gamma}$ by fixing an embedding $i: S \to G$ defined over $\R$ and defining $W_G(S)^{\Gamma} := W_G(i(S))^{\Gamma}$. We note that this definition is independent of $i$ since any two such embeddings are conjugate by some $g \in G(\C)$ which can be taken to be in $N_G(i(S))$ and whose $\Gamma$-invariance in $W_G(S)$, comes from both embeddings being defined over $\R$.
As a consequence of Lemma \ref{lem: mainreallem}, we have constructed for all tempered parameters $\phi$ a commutative diagram
  \begin{equation}
    \begin{tikzcd}
    \coprod\limits_{b \in B(G)_{\bas}} \Pi_{\phi}(G_b) \arrow[d] \arrow[r, "\iota_{\mf{w}}"] & \Irr(S^{\natural}_{\phi}) \arrow[d] \\
    B(G)_{\bas} & B(S)_{\Gbas} / W_G(S)^{\Gamma} \arrow[l],
    \end{tikzcd}
\end{equation}  
where $\iota_{\mf{w}}$ is bijective. 
More precisely, for any $\rho\in \Irr(S_\phi^\natural)$, the pull back of $\rho$ along the map $p$ is trivial on $\im(r)$ by Lemma \ref{lem: mainreallem}.
In particular, it gives rise to an element of $X^*(\wh{S}^{\Gamma_{\phi_1}}/\wh{S}_{\Gbas})$.
By noting that we have a bijection $\kappa_S\colon B(S)_{\Gbas}\rightarrow X^*(\wh{S}^{\Gamma_{\phi_1}}/\wh{S}_{\Gbas})$, we get an element $b$ of $B(S)_{\Gbas}$.
This association $\rho\mapsto b$ is the right vertical map.
Moreover, the $G_b(\R)$-rational conjugacy class of admissible embeddings $i\colon S\rightarrow G_b$ corresponding to the element $b\in B(S)_{\Gbas}$ satisfying the condition that $i(\Delta_\phi)$ are non-compact by the triviality of $p^*\rho$ on $r(\Omega(\Delta_\phi))$.
Hence $i$ corresponds to an element $\pi$ of $\Pi_\phi(G_b)$.
This association $\rho\mapsto\pi$ is the top horizontal map.

We now extend this construction to the non-tempered case. This is done via the Langlands classification and Langlands classification for $L$-parameters as in \cite[Appendix A]{SilbergerZink}. Fix $G'$ a connected reductive group over $\R$, a minimal $\R$-parabolic $P_0 \subset G'$ with Levi subgroup $M_0$ and maximal $\R$-split torus $A_0$. Let $\mf{a}^*_{M_0} = X^*(M_0)^{\Gamma}_{\R}$. On the one hand we have a bijection
\begin{theorem}[Langlands Classification]{\label{thm: LanglandsClassification}}
    \begin{equation*}
        \{(P, \sigma, \nu)\} \leftrightarrow \Pi(G'),
    \end{equation*}
    where $(P, \sigma, \nu)$ is a triple where $P \supset P_0$ is a standard parabolic subgroup with standard Levi $M$ and unipotent radical $N$, where $\sigma \in \Pi(M)$ is tempered, and $\nu \in \mf{a}^*_M \xhookrightarrow{\res} \mf{a}^*_{M_0}$ pairs positively with any root of $A_0$ in $N$.
\end{theorem}
On the $L$-parameter side, we have
\begin{theorem}[{\cite[A.2]{SilbergerZink}}]{\label{thm: LanglandsClassificationLparam}}
      \begin{equation*}
        \{(P, \co{t}{\phi}, \nu)\} \leftrightarrow \Phi(G'),
    \end{equation*}
    where $(P, \co{t}{\phi}, \nu)$ is a triple where $P \supset P_0$ is a standard parabolic subgroup with standard Levi $M$ and unipotent radical $N$, where $\co{t}{\phi}$ is a tempered $L$-parameter of $M$ up to equivalence, and $\nu \in \mf{a}^*_M \xhookrightarrow{\res} \mf{a}^*_{M_0}$ pairs positively with any root of $A_0$ in $N$.
\end{theorem}
With these theorems, we define $\iota_{\mf{w}}$ as follows. Choose $b \in B(G)_{\bas}$ and choose an extended pure inner twist $(G_b, \varphi, z)$ such that $[z]=b$. Let $\phi \in \Phi(G_b)$ and suppose $\phi$ corresponds to $(P_b,\co{t}{\phi}, \nu)$ by Theorem \ref{thm: LanglandsClassificationLparam}. We have that $P_b=M_bN_b \subset G_b$ where $M_b \subset G_b$ is a standard Levi subgroup corresponding to a standard Levi $M \subset G$. Then by \cite[Lemma 2.4]{BMHN} (this Lemma is proven for $F=\Q_p$ in \textit{loc.\ cit.}\ but the proof works also for $F=\R$), there is a unique equivalence class of extended pure inner twists $(M_b, \varphi_M, z_M)$ with class $b_M \in B(M)$ whose class in $B(G)$ is $b$. We define $\Pi_{\phi}(G_b)$ to consist of all elements of $\Pi(G_b)$ with corresponding triple $(P, \sigma, \nu)$ such that $\sigma \in \Pi_{\co{t}{\phi}}(M_b)$.

Following \cite[\S7]{SilbergerZink} we have $S_{\phi} = S_{\phi, M} = S_{\co{t}{\phi}, M}$ so we define $\iota_{\mf{w}}$ on $G$ by declaring that for $\pi \in \Pi(G_b)$ corresponding to $(P_b, \sigma, \nu)$, we have $\iota_{\mf{w}}(\pi) := \iota_{\mf{w}_M}(\sigma)$ where $\mf{w}_M$ is the Whittaker datum of $M$ given by restricting $\mf{w}$ and we are temporarily thinking of both sides of this equality as representations of $S_{\phi} = S_{\co{t}{\phi}, M}$. Then the proof of \cite[Theorem 2.5]{BMHN} shows that $\iota_{\mf{w}}(\pi)$ factors to give a representation of $S^{\natural}_{\phi}$.

If we pullback $\iota_{\mf{w}}(\pi)$ to $\wh{S}^{\Gamma_{\phi_1}}$ via $p$, then we get an element $b_S \in B(S)_{\Mbas}$ whose image in $B(M)$ is $b_M$. Hence, the image in $B(G)$ is $b$ and therefore $b_S \in B(S)_{\Gbas}$ and gives a class in $B(S)_{\Gbas}/W_G(S)^{\Gamma}$ which recovers $b$. 
To prove $\iota_{\mf{w}}$ is a bijection, we construct an inverse. Note that given a representation $\rho \in \Irr(S_{\phi}^\natural)$ whose pullback to $S_{\phi_1}$ yields $b_S \in B(S)_{\Gbas}$ mapping to $b \in B(G)$, such a representation factors to give a representation of $S^{\natural}_{\co{t}{\phi}, M}$ and by the uniqueness result (\cite[Lemma 2.4]{BMHN}) we must have that $b_S$ maps to $b_M \in B(M)_{\bas}$.

In particular, we have proven the following theorem.

\begin{theorem}{\label{thm: realBGbasthm}}
We have the following commutative diagram
  \begin{equation}{\label{eqn: archimedean basic correspondence}}
    \begin{tikzcd}
    \coprod\limits_{b \in B(G)_{\bas}} \Pi_{\phi}(G_b) \arrow[d] \arrow[r, "\iota_{\mf{w}}"] & \Irr(S^{\natural}_{\phi}) \arrow[d] \\
    B(G)_{\bas} & B(S)_{\Gbas} / W_G(S)^{\Gamma} \arrow[l],
    \end{tikzcd}
\end{equation}  
where $\iota_{\mf{w}}$ is bijective.
\end{theorem}

The bottom map is explained in \S\ref{ss: Kottwitzreview}. The right vertical map comes from pullback along the map $p$ of Lemma \ref{lem: mainreallem} and uses the constructions in that lemma to show that we indeed get an element of $B(S)_{\Gbas}$. This element of $B(S)$ depends on the choice of $i_{\pi}$ in its $G(\R)$-conjugacy class. This ambiguity corresponds to modifying our element of $B(S)$ by an element of $N_{G(\R)}(i_{\pi}(S))$ and this ambiguity is removed when we take a quotient by $W_G(S)^{\Gamma}$. 

\subsection{Statement of main theorem}

We now return to considering a general local field $F$. Our aim in this paper is, by assuming the  $B(G)_\bas$-LLC (Conjecture \ref{conj: basic-LLC}) and its refinement in the Archimedean case, to establish its ``$B(G)$-version" in a reasonable way:
\begin{theorem}\label{thm: main}
We assume Conjecture \ref{conj: basic-LLC} for $G$ and all standard Levi subgroups of $G$.
For each $b\in B(G)$, there exists a finite-to-one map
\[
\LLC_{G_b}\colon\Pi(G_b) \rightarrow \Phi(G),
\]
given by the composition
\begin{equation}
    \Pi(G_b) \rightarrow \Phi(L) \rightarrow \Phi(G),
\end{equation}
where $L \subset G$ is the standard Levi subgroup that is the quasi-split inner form of $G_b$, the first map is from Conjecture \ref{conj: basic-LLC} for $L$, and the second map comes from $\co{L}{L} \hookrightarrow \co{L}{G}$.
Furthermore, for each $\phi \in \Phi(G)$, the union of $\Pi_\phi(G_b):=\LLC_{G_b}^{-1}(\phi)$ over $b\in B(G)$ is equipped with a bijection $\iota_\mf{w}$ to $\Irr(S_\phi)$ such that the following diagram commutes:
\begin{equation}
    \begin{tikzcd}
    \coprod\limits_{b \in B(G)} \Pi_{\phi}(G_b) \arrow[d] \arrow[r, "\iota_{\mf{w}}"] & \Irr(S_{\phi}) \arrow[d] \\
    B(G) \arrow[r, "\kappa_G"]& X^{\ast}(Z(\widehat{G})^{\Gamma}),
    \end{tikzcd}
\end{equation}
where the left vertical map is the obvious projection and the right vertical map takes central character. In particular, we note that since $\iota_{\mf{w}}$ is bijective, one can recover $b \in B(G)$ from $\iota_{\mf{w}}(\pi) \in \Irr(S_{\phi})$ for $\pi \in \Pi_{\phi}(G_b)$.
\end{theorem}
Some remarks are in order.
\begin{remark}{\label{rem: mainthmrems}}
\begin{enumerate}
    \item     The set $\Pi_{\phi}(G_b)$ is trivial if $\phi$ does not factor through the canonical embedding $\co{L}{G_b} \rightarrow \co{L}{G}$. In particular, when $\phi$ is discrete, nothing new happens: $\Pi_{\phi}(G_b)$ is trivial for all non-basic $b$, and so we reduce to Conjecture \ref{conj: basic-LLC}.
    \item    From Corollary \ref{cor: surjectivity}, we get that the $\Pi_{\phi}(G_b)$ are unions of $L$-packets for $G_b$ considered as an inner twist of its quasi-split inner form.
    \item  In many cases, Theorem \ref{thm: main} is unconditional because Conjecture \ref{conj: basic-LLC} is known for all standard Levi subgroups. For instance, when $F$ is non-archimedean, this is true for $\GL_n$ by \cite{HT1}, \cite{HenniartLLC}, \cite{DeligneKazhdanVigneras}, \cite{LaumonRapoportStuhler}. For $p$-adic  $\SL_n$ this essentially follows from \cite{HiragaSaito}; here, the meaning of ``essentially" is that a Levi of $\SL_n$ is an intermediate group between a product of general linear groups and a product of special linear groups, hence we need to consider such groups inductively as well.  For $p$-adic unitary groups, this follows from \cite{Mok}, \cite{KMSW}, and \cite{AGIKMS}. The case of $p$-adic $\SO_{2n+1}$ is known by \cite{Arthurbook} and \cite{Ishimoto}. The archimedean case is known for all groups as discussed in \ref{ss: enhancedbasicarchllc}.
    \item We can also check that the maps $\LLC_{G_b}$ and $\iota_\mf{w}$ of Theorem \ref{thm: main} satisfy an expected property on duality. See the end of Section \ref{ss: properties}.
\end{enumerate}
\end{remark}

\begin{example}
    The simplest non-trivial example is for $G=\GL_2$ where $\phi=\phi_1 \oplus \phi_2$ is a sum of two characters of $W_F$ that do not differ by the norm character $|\cdot |$. Let $\chi_1, \chi_2$ be the corresponding characters of $F^{\times}$ by local class field theory. We fix the standard splitting of $G$ using the diagonal torus $T$, upper triangular Borel $B$, and standard choice of a simple root vector.
    Then $S_{\phi}$ can be identified with the diagonal torus of $\GL_2$ and we have $X^*(S_{\phi}) = \Z^2$. The Kottwitz set $B(\GL_2)= B(G)_G \coprod B(G)_B$ and we have $B(G)_G=\Z$ and $B(G)_B = \Z^2_>= \{ (x,y) \in \Z \mid x>y \}$. Then for $b \in B(G)_G$, we have $\Pi_{\phi}(G_b)$ is empty if $b$ is odd (so $G_b$ is non-split) and contains the irreducible representation $I^G_B(\chi_1 \boxtimes \chi_2)$ when $b$ is even. For $b=(x,y) \in B(G)_B$, we have $G_b$ is isomorphic to the diagonal torus of $G$ and $\Pi_{\phi}(G_b) = \{ \chi_1 \boxtimes \chi_2 , \chi_2 \boxtimes \chi_1 \}$. These representations correspond to two different elements of $X^*(S_{\phi})=\Z^2$. The first has weights $(x,y)$ and the second has weights $(y,x)$.
\end{example}
\begin{example}
    The next interesting example to consider is the parameter of $G=\SL_2$ corresponding to a degree two extension $E/F$ and such that the map $W_F \rightarrow \wh{G} = \PGL_2$ factors through $W_F/W_E$ and takes the non-trivial element to $\begin{pmatrix} 1 & 0 \\ 0 & -1 \end{pmatrix}$. Then $S_{\phi} = T \coprod nT$ where $T$ is the diagonal torus and $n = \begin{pmatrix} 0 & 1 \\ 1 & 0 \end{pmatrix} $. The abelianization of $S_{\phi}$ is the component group which is $\Z/2\Z$ and hence there are two irreducible characters which correspond to the two representations of $\SL_2(F)$ in the packet of the unique basic element of $B(G)$. The other irreducibles of $S_{\phi}$ are $2$-dimensional and each one restricted to $T$ is a sum of two non-trivial characters of weight $n$ and $-n$ for some integer $n > 0$. We call these $\pi_n$. The elements of $B(G)_B$ are in bijection with positive integers and the element $b \in B(G)_B$ corresponding to positive integer $n$ satisfies $\Pi_{\phi}(G_b) = \{ \pi_n \}$.
\end{example}

\section{The construction}{\label{s: theconstruction}}
In this section, $F$ is an arbitrary local field. Recall that we fixed an $F$-splitting $(T,B,\{X_\alpha\})$ of $G$, which gives rise to a Whittaker datum $\mf{w}$ of $G$.
For each standard Levi subgroup $L \subset G$, the Whittaker datum $\mf{w}$ restricts to give a Whittaker datum $\mf{w}_L$ of $L$.

\subsection{The easy map}{\label{ss: theeasymap}}

Fix a pair $(b, \pi_b)$ of $b\in B(G)$ and $\pi_b\in\Pi(G_b)$.
By \eqref{eqn: BGdecomp}, there exists a unique standard Levi subgroup $L$ of $G$ and $b_L\in B(L)_\bas^+$ such that $b_L$ is identified with $b\in B(G)$.
We may regard $\pi_b$ as an element of $\Pi(L_{b_L})$ via the identification $G_b\cong L_{b_L}$ as discussed later; see Lemma \ref{lem: Gbequals}.
Then, by the $B(L)_\bas$-LLC, we can associate to $\pi_b$ the pair $(\phi,\rho_L)$ of an $L$-parameter $\phi$ of $L$ and an irreducible representation $\rho_L$ of $S_{\phi,L}^\natural$ (i.e., $\rho_L=\iota_{\mf{w}_L}(\pi_b)$).

Let ${}^{L}M$ be a smallest Levi subgroup of ${}^{L}G$ such that $\phi$ factors through the $L$-embedding ${}^{L}M\hookrightarrow{}^{L}L$.
We regard $\phi$ also as an $L$-parameter of $G$ by composing it with the embedding ${}^{L}L\hookrightarrow{}^{L}G$.
Then, by Lemma \ref{lem: S-group-Levis}, $S_{\phi}^\circ$ is a connected reductive group and $S_{\phi,L}^\circ$ is its Levi subgroup with a maximal torus $A_{\wh{M}}$.
Hence, by representation theory of disconnected reductive groups (\S \ref{ss: repthrydisconnected}), $\rho_L$ is given by $\mc{L}_L(\lambda,E)$, where $\lambda\in X^\ast(A_{\wh{M}})^+$ is a dominant character and $E$ is a simple $\mc{A}_L^\lambda$-module with the notation as in \S \ref{ss: repthrydisconnected}.

Since $A_L^\lambda=A^\lambda$ by Proposition \ref{prop: AlambdaL}, $E$ can be regarded as a simple $\mc{A}^\lambda$-module.
Thus, we get an irreducible representation $\rho:=\mc{L}(\lambda,E)$ of $S_\phi$.
We put $\iota_{\mf{w}}(\pi_b):=\rho$ and this completes the construction of our map.

\subsection{The map in the other direction}{\label{ss: themap}}

We now construct a map in the other direction.
Let $[\phi]\in\Phi(G)$, i.e., $[\phi]$ is a $\wh{G}$-conjugacy class of $L$-parameters of $G$. 
(In this section, we use the symbol $[\phi]$ in order to emphasize that it is a $\wh{G}$-conjugacy class.)
We fix a representative $\phi$ of $[\phi]$. 
The group $S_{\phi}$ is a possibly disconnected reductive group. 
Our aim is to associate to $\rho\in\Irr(S_{\phi})$ a pair $(b, \pi_b)$ for $b \in B(G)$ and $\pi_b \in \Pi(G_b)$. 

Let ${}^{L}{M}$ be a minimal Levi subgroup through which $\phi$ factors and as in \S \ref{ss: weylgroupconstructions}, we replace $\phi$ with a conjugate such that we can assume ${}^{L}M$ is a standard Levi.
Let $M$ be the standard Levi subgroup of $G$ corresponding to ${}^{L}M$.
We fix a Borel subgroup $B_\phi$ of $S_\phi^\circ$ containing the maximal torus $A_{\wh{M}}$.

Let $\rho\in\mathrm{Irr}(S_\phi)$. 
By the classification of irreducible representations of disconnected reductive groups (Theorem \ref{thm: acharmainthm}), there exists a weight $\lambda \in X^{\ast}(A_{\wh{M}})^+$ (dominant relative to $B_{\phi}$) and a simple $\mc{A}^{\lambda}$-module $E$ such that $\rho \cong \mc{L}(\lambda,E)$ with the notations as in \S\ref{ss: repthrydisconnected}.
We associate $w\in W^\rel$, $Q=Q_\lambda$, and $L=L_\lambda$ to $\lambda$ according to the construction given in \S \ref{ss: weylgroupconstructions}.
Let us write $M':=\wM$.
Choose a representative $\dot{w} \in N_{\wh{G}}(A_{\wh{T}})$ of $w\in W^\rel\cong\wh{W}^\rel$ and consider the conjugate $L$-parameter $\phi':=\Int(\dot{w})\circ\phi$ of $\phi$. 
By construction, $\phi'$ factors through $\co{L}M'$ and hence $\co{L}L$. 
Conjugation by $\dot{w}$ induces an isomorphism $\Int(\dot{w}): S_{\phi} \cong S_{\phi'}$ and hence we get a corresponding representation $\rho' \in \Irr(S_{\phi'})$ and weight $\lambda':=\co{\dot{w}}\lambda \in X^{\ast}(A_{\wh{M'}})^+$ (dominant relative to $\co{\dot{w}}B_{\phi}$). 
We have $\rho'\cong \mc{L}(\lambda',E')$, where $E'$ is the simple $\mc{A}^{\lambda'}$-module corresponding to $E$ under the identification $S_{\phi} \cong S_{\phi'}$.
 
We let $\mc{L}_L(\lambda')$ be the irreducible  representation of $S^{\circ}_{\phi',L}$ with highest weight $\lambda'$.
Proposition \ref{prop: AlambdaL} says that the natural map from $\pi_0(S_{\phi',L})=A^{\lambda'}_L$ to $A^{\lambda'}$ is a bijection.
Thus we may regard $E'$ as a simple $\mc{A}_L^{\lambda'}$-module, for which we write $E'_L$.
Again by the classification of irreducible representations of disconnected reductive groups, applied to $S_{\phi',L}$, we get an irreducible representation $\mc{L}_L(\lambda',E'_L)$ of $S_{\phi',L}$.
We denote this representation by $\rho_L$.

\begin{lemma}{\label{lem: naturalfactoring}}
The representation $\rho_L \in \Irr(S_{\phi',L})$ factors through $S^{\natural}_{\phi',L}$, to give a representation which by abuse of notation we also denote $\rho_L$. 
\end{lemma}

\begin{proof}
We first study the representation $\mc{L}_L(\lambda') \in \Irr(S^{\circ}_{\phi',L})$. 
Note that $\alpha_{M'}(\lambda')$ belongs to $\mf{A}_L$ by construction and the following diagram commutes.
\begin{equation*}
     \begin{tikzcd}
      X^{\ast}(A_{\wh{M'}})_{\R} \arrow[rr, bend right=-15, "\alpha_{M'}"]  &X^{\ast}(\wh{M'}_{\ab})_{\R} \arrow[l, "\res"] & \mf{A}_{M'}\arrow[l]   \\
    &X^{\ast}(\wh{L}_{\ab})_{\R} \arrow[u, hook] & \mf{A}_L \arrow[l] \arrow[u, hook]
      \end{tikzcd}
\end{equation*}
Let $m\in\Z_{>0}$ be a positive integer such that $\alpha_{M'}(m\lambda')$ belongs to $X_\ast(A_{L})$. Hence, by the above diagram, there is a character of $\wh{L}$ whose restriction to $A_{\wh{M'}}$ is $m\lambda'$.
Then the irreducible representation $\mc{L}_L(m\lambda') \in \Irr(S^{\circ}_{\phi',L})$ with highest weight $m\lambda'$ is actually just this character acting through $S^{\circ}_{\phi',L} \subset \wh{L}$.
This implies that the irreducible representation $\mc{L}_L(\lambda') \in \Irr(S^{\circ}_{\phi',L})$ with highest weight $\lambda'$ is also a character of $S^{\circ}_{\phi',L}$.
(This can be checked by, e.g., comparing the dimensions of $\mc{L}_L(m\lambda')$ and $\mc{L}_L(\lambda')$; through the Weyl dimension formula, we can easily see that $\dim\mc{L}_L(\lambda')\leq \dim\mc{L}_L(m\lambda')$.)

Since $\mc{L}_L(m\lambda')$ is the restriction of a character of $\wh{L}$, the representation $\mc{L}_L(m\lambda')$ is clearly trivial on $(\wh{L}_{\der} \cap S^{\circ}_{\phi',L})^{\circ}$.
In other words, the $m$-th power of the character $\mc{L}_L(\lambda')|_{(\wh{L}_{\der} \cap S^{\circ}_{\phi',L})^{\circ}}$ is trivial.
As the finitely generated abelian group $X^\ast((\wh{L}_{\der} \cap S^{\circ}_{\phi',L})^{\circ})$ is torsion-free, this implies that $\mc{L}_L(\lambda')|_{(\wh{L}_{\der} \cap S^{\circ}_{\phi',L})^{\circ}}$ is trivial.
Therefore $\rho_L$ is trivial on $(\wh{L}_{\der} \cap S^{\circ}_{\phi',L})^{\circ}=(\wh{L}_{\der} \cap S_{\phi',L})^{\circ}$.
This concludes the proof of the lemma.
\end{proof}

Now, by the $B(L)_\bas$-LLC (Conjecture \ref{conj: basic-LLC} for $F$ non-Archimedean, Diagram \eqref{eqn: complexbasicLLC} for $\C$, Theorem \ref{thm: realBGbasthm} for $\R$), we get $b_L \in B(L)_{\bas}$ and $\pi_{b_L} \in \Pi(L_{b_L})$ corresponding to $\rho_L\in\Irr(S_{\phi',L}^\natural)$ (i.e., $\iota_{\mf{w}_L}(\pi_{b_L})=\rho_L$).
Denote by $b$ the image of $b_L$ in $B(G)$. 

\begin{lemma}{\label{lem: dominance}}
We have $b_{L}\in B(L)^+_\bas$.
\end{lemma}

\begin{proof}
Recall that the natural map $B(L)\rightarrow B(G)$ induces a bijection $B(L)_\bas^+\xrightarrow{1:1}B(G)_Q$ and that the subset $B(G)_Q$ of $B(G)$ is defined to be the preimage of $\mf{A}_Q^+$ under the Newton map (\S \ref{ss: Kottwitzreview}):
\[
\begin{tikzcd}
B(L) \ar[r] & B(G) \ar[r, "\nu_G"] & \overline{C}\\
B(L)_\bas^+ \ar[r, "1:1"] \ar[u, hook] & B(G)_Q \ar[r] \ar[u,hook] & \mf{A}_Q^+ \ar[u, hook]
\end{tikzcd}
\]
Thus our task is to check that $\nu_G(b)$ belongs to $\mf{A}_Q^+$.
Since the Newton map is functorial, i.e., we have $\nu_G(b)=\nu_L(b_L)$, it suffices to show that $\nu_L(b_L)$ belongs to $\mf{A}_Q^+$.

Recall that we have $\nu_L(b_L)=\alpha_L\circ\kappa_L(b_L)$ since $b_L$ is basic (\eqref{eqn: Kottwitz-Newton} for $L$):
\begin{equation*}
\begin{tikzcd}
B(L)_\bas \ar[r, "\kappa_L"] \ar[rrr, bend left=15, "\nu_L"] & X^\ast(Z(\wh{L})^\Gamma) \ar[r] & X^\ast(Z(\wh{L})^\Gamma)_\R \ar[r, "\alpha_L"]& \mf{A}_L=X_\ast(A_L)_\R
\end{tikzcd} 
\end{equation*}
By the commutative diagram \eqref{eqn: basic correspondence} (applied to $L$), $Z(\wh{L})^{\Gamma}$ acts on $\rho_{L}$ via $\kappa_{L}(b_{L})\in X^{\ast}(Z(\wh{L})^{\Gamma})$.
Since $\wh{T}\subset\wh{M'}\subset\wh{L}$, we have $\wh{T}\supset A_{\wh{M'}}\supset A_{\wh{L}}$.
By construction, $A_{\wh{M'}}$ acts on $\rho_{L}$ via $\lambda'$.
Hence the element $\kappa_{L}(b_{L})\in X^{\ast}(Z(\wh{L})^{\Gamma})_{\R}$ is nothing but the image of $\lambda'$ under the map
\[
X^{\ast}(A_{\wh{M'}})_{\R}
\xrightarrow{\mathrm{res}} X^{\ast}(A_{\wh{L}})_{\R}
=X^{\ast}(Z(\wh{L})^{\Gamma})_{\R}.
\]

Now recall that the standard parabolic subgroup $Q$ with standard Levi $L$ is chosen so that $w\cdot \alpha_{M}(\lambda)=\alpha_{M'}(\lambda')$ belongs to $\mf{A}_Q^+$.
We note that the natural inclusion map $\mf{A}_L\hookrightarrow\mf{A}_{M'}$ gives a section of the restriction map $X^\ast(A_{\wh{M'}})_\R\twoheadrightarrow X^\ast(A_{\wh{L}})_\R$ under the identifications via $\alpha_{L}$ and $\alpha_{M'}$.
\[
\begin{tikzcd}
X^\ast(A_{\wh{M'}})_\R \ar[r, "\alpha_{M'}"] \ar[d, "\mathrm{res}"] & \mf{A}_{M'}\\
X^\ast(A_{\wh{L}})_\R \ar[r, "\alpha_{L}"] & \mf{A}_{L} \ar[u, hook]
\end{tikzcd}
\]
Hence the image of $\lambda'\in X^\ast(A_{\wh{M'}})_\R$ in $X^\ast(A_{\wh{L}})_\R$, which equals $\kappa_L(b_L)$ by the argument in the previous paragraph, is equal to $\alpha_{L}^{-1}\circ\alpha_{M'}(\lambda')$.
Thus we get $\alpha_{L}\circ\kappa_L(b_L)=\alpha_{M'}(\lambda')$.
This implies that $\nu_L(b_L)\,(=\alpha_L\circ\kappa_L(b_L)=\alpha_{M'}(\lambda'))$ lies in $\mf{A}_Q^+$.
\end{proof}

\begin{lemma}{\label{lem: Gbequals}}
We have $L_{b_L}=G_b$.
\end{lemma}

\begin{proof}
By the definition of the groups $G_b$ and $L_{b_L}$, we have that $L_{b_L}$ is naturally embedded in $G_b$.
The group $G_b$ is an inner form of a Levi subgroup of $G$ given by the centralizer of $\nu_G(b)$ in $G$ (\S\ref{ss: Kottwitzreview}).
Similarly, the group $L_{b_L}$ is an inner form of a Levi subgroup of $L$ given by the centralizer of $\nu_L(b_L)$ in $L$.
Thus, since we have $\nu_L(b_L)=\nu_G(b)$, it is enough to show that the centralizer of $\nu_G(b)$ in $G$ is equal to $L$.
Noting that $\nu_G(b)$ belongs to $\mf{A}_Q^+$ by Lemma \ref{lem: dominance}, this can be easily checked by looking at the definition of $\mf{A}_Q^+$.
\end{proof}

By this lemma, we may regard $\pi_{b_L}$ as a representation of $G_b(F)$.
We define $\pi_b \in \Pi(G_b)$ to be this representation. 
Hence we have finally constructed $(b, \pi_b)$ as desired. This concludes the construction. 

It is moreover easy to see that applying this map to the $\rho$ produced by \S\ref{ss: theeasymap} returns the original $(b, \pi_b)$ up to equivalence. Hence the map $([\phi], \rho) \mapsto (b, \pi_b)$ is surjective onto $\coprod_b \Pi(G_b)$.

\subsection{Independence of choices}

Recall that, for fixed $[\phi]\in\Phi(G)$ and $\rho\in\Irr(S_\phi)$, several choices were made in the construction of $(b,\pi_b)$ as follows. 
\begin{enumerate}
\item 
We fixed a representative $\phi$ of $[\phi]$.
\item
We chose a smallest Levi subgroup ${}^{L}M$ such that $\phi$ factors through ${}^{L}M\hookrightarrow{}^{L}G$.
Furthermore, we replaced $\phi$ with its conjugate ${}^{x}\phi$ so that ${}^x({}^{L}M)$ is a standard Levi subgroup.
(We put ${}^{L}M$ to be ${}^x({}^{L}M)$.)
\item
We took a weight $\lambda\in X^\ast(A_{\wh{M}})^+$ and a simple $\mc{A}^\lambda$-module $E$ such that $\rho\cong\mc{L}(\lambda,E)$.
\item
We took $w\in W^\rel$ such that $w\cdot\alpha_M(\lambda)$ belongs to $\overline{C}$ and defined the standard parabolic $Q$ with standard Levi $L$ to be the unique one satisfying $w\cdot\alpha_M(\lambda)\in \mf{A}_Q^+$.
\item
Then, by taking a representative $\dot{w}$ of the element $w\in W^\rel\cong\wh{W}^\rel$, we applied the $B(L)_\bas$-LLC to $({}^{\dot{w}}\phi,\rho_L)$, where $\rho_L:=\mc{L}_L({}^{\dot{w}}\lambda,{}^{\dot{w}}E_L)$.
\end{enumerate}

We now explain that our construction is independent of these.

We first discuss (5).
Any two choices $\dot{w}, \dot{w}' \in N_{\wh{G}}(A_{\wh{T}})$ differ by an element of $\wh{T}$. This means that $(\co{\dot{w}}\phi,\mc{L}_L({}^{\dot{w}}\lambda,{}^{\dot{w}}E_L))$ and $(\co{\dot{w}'}\phi,\mc{L}_L({}^{\dot{w}'}\lambda,{}^{\dot{w}'}E_L))$ differ by conjugation by an element of $\wh{T}\subset\wh{L}$.
Hence, the resulting $(b,\pi_b)$ does not change since the basic correspondence is assumed to be well-defined.

We next discuss (4)
If $w'\in W^\rel$ is another element such that $w'\cdot\alpha_M(\lambda)\in\overline{C}$, then we must have $w\cdot\alpha_M(\lambda)=w'\cdot\alpha_M(\lambda)$ (see, e.g., \cite[Lemma 10.3.B]{Hum3}).
In particular, the standard Levi subgroup $L$ does not change.
Furthermore, $w'w^{-1}$ stabilizes $w\cdot\alpha_M(\lambda)$ and hence lies in $W_L^\rel$ by Lemma \ref{lem: lambdacentweyl}.
This will modify $({}^{\dot{w}}\phi,\rho_L)$ up to $\wh{L}$-conjugacy, which does not affect $(b,\pi_b)$.

Let us discuss (3).
We take another weight $\lambda'\in X^\ast(A_{\wh{M}})^+$ and simple $\mc{A}^{\lambda'}$-module $E'$ such that $\rho\cong\mc{L}(\lambda',E')$.
By Lemma \ref{lem: disconn-irrep-paramet}, we may assume that $\lambda'$ is $R_\phi$-conjugate to $\lambda$ (say $\lambda'=w\cdot\lambda$) and $E$ and $E'$ are identified under the isomorphism $\mc{A}^\lambda\cong\mc{A}^{w\cdot\lambda}$.
Recall that the action of $w\in R_\phi$ factors through $R_\phi\rightarrow W_{\wh{G}}(A_{\wh{M}})$ (see \eqref{map:Wphi}) and that $W_{\wh{G}}(A_{\wh{M}})$ is identified with a subgroup of $\wh{W}^\rel$ (Lemma \ref{lem: weylinj}).
Thus, by Lemma \ref{lem: weylinj2}, $w$ does not affect the definition of $L$ and $\rho_L$.

Let us discuss (2).
Let ${}^{L}M$ and ${}^{L}M'$ be two smallest Levi subgroups of $G$ such that $\phi$ factors through ${}^{L}M$ and ${}^{L}M'$, respectively.
As explained in \S\ref{ss: S-group}, ${}^{L}M$ and ${}^{L}M'$ are conjugate by an element of $S^\circ_\phi$, say ${}^{s}({}^{L}M)={}^{L}M'$.
Thus using $M'$ instead of $M$ amounts to using ${}^s\rho$ instead of $\rho$.
Since ${}^s\rho\cong\rho$, this does not change the rest of the construction of $(b,\pi_b)$.

We finally discuss (1).
Let us choose ${}^{g}\phi$ conjugate to $\phi$ via $g\in\wh{G}$.
Then ${}^g({}^{L}M)$ is a smallest Levi subgroup such that ${}^g\phi$ factors through ${}^g({}^{L}M)\hookrightarrow{}^{L}G$.
Thus, both $\phi$ and ${}^g\phi$ are conjugate to ${}^x\phi$, whose image is contained in a standard Levi subgroup ${}^x({}^LM)$.

\subsection{Properties of the correspondence}\label{ss: properties}
We now verify that the construction in \S\ref{ss: themap} is well behaved.

\begin{proposition}{\label{prop: injectivity}}
The map $([\phi],\rho)\mapsto(b,\pi_b)$ constructed in \S \ref{ss: themap} is injective. To be more precise, suppose $\phi_1, \phi_2$ are $L$-parameters of $G$ and $\rho_i \in \Irr(S_{\phi_i})$ and that our map takes $\rho_i$ to $(\pi_i , b_i)$ with $b_1 = b_2$ and $\pi_1 \cong \pi_2$. Then $\phi_1 \sim \phi_2$ and $\rho_1 \sim \rho_2$.
\end{proposition}

Here, the meaning of ``$\rho_1 \sim \rho_2$" in the statement is as follows.
Since we have $\phi_1 \sim \phi_2$, we can take $g\in \widehat{G}$ such that ${}^g\phi_2=\phi_1$, which implies that ${}^g S_{\phi_2}=S_{\phi_1}$.
Then we have ${}^g\rho_2 \cong \rho_1$.
Note that this condition is independent of the choice of $g$ as any other choice $g'$ can differ from $g$ only by an element of $S_{\phi_2}$.

\begin{proof}
For $i=1,2$, let $L_i$ be the Levi subgroup associated to $(\phi_i,\rho_i)$ as in \S \ref{ss: themap}.
Similarly, we let $\rho_{i,L_i}\in \Irr(S^{\natural}_{{}^{\dot{w}_i}\phi_i,L_i})$ denote the representation associated to $(\phi_i,\rho_i)$ as in \S \ref{ss: themap}.
Recall that $b_i$ and $\pi_i\in\Pi(G_{b_i})$ are obtained by applying the $B(L_i)_\bas$-LLC to $\rho_{i,L_i}\in \Irr(S^{\natural}_{{}^{\dot{w}_i}\phi_i,L_i})$.

Note that $L_i$ is characterized as the unique standard Levi subgroup of $G$ such that $b_i\in B(G)$ is contained in $B(L_i)^+_\bas$ by Lemma \ref{lem: dominance} and the decomposition \eqref{eqn: BGdecomp}.
Thus the assumption that $b_1=b_2$ implies that $L_1=L_2$.
Let us simply write $L$ for $L_1=L_2$ in the following.

Since the $B(L)_\bas$-LLC is bijective, the assumption $\pi_1\cong\pi_2$ implies that ${}^{\dot{w}_1}\phi_1$ and ${}^{\dot{w}_2}\phi_2$ are equivalent as $L$-parameters of $L$.
Hence $\phi_1$ and $\phi_2$ are equivalent as $L$-parameters of $G$.
In the following, we fix an element $l\in\widehat{L}$ satisfying ${}^{\dot{w}_2}\phi_2={}^{l\dot{w}_1}\phi_1$ (hence we get ${}^{l}S^{\natural}_{{}^{\dot{w}_1}\phi_1,L}=S^{\natural}_{{}^{\dot{w}_2}\phi_2,L}$ and ${}^{l}S^{\natural}_{{}^{\dot{w}_1}\phi_1}=S^{\natural}_{{}^{\dot{w}_2}\phi_2}$).

Let us show that the representations ${}^{l\dot{w}_{1}}\rho_{1}$ and ${}^{\dot{w}_{2}}\rho_{2}$ of $S^{\natural}_{{}^{\dot{w}_2}\phi_2}$ are isomorphic.
For this, for each $i=1,2$, we take an element $\lambda_{i}\in X^{\ast}(A_{\wh{M}})^{+}$ and a simple $\mc{A}^{\lambda_{i}}$-module $E_{i}$ such that $\rho_{i}\cong \mc{L}(\lambda_{i},E_{i})$.
Then, by construction, $\rho_{i,L}$ is the unique irreducible representation of $S_{{}^{\dot{w}_{i}}\phi_{i},L}$ associated with the pair $({}^{\dot{w}_{i}}\lambda_{i},{}^{\dot{w}_{i}}E_{i,L})$, where ${}^{\dot{w}_{i}}E_{i,L}$ is ${}^{\dot{w}_{i}}E_{i}$ regarded as a simple $\mc{A}_{L}^{{}^{\dot{w}_{i}}\lambda}$-module via the bijection $A_{L}^{{}^{\dot{w}_{i}}\lambda}\cong A^{{}^{\dot{w}_{i}}\lambda}$.
As the assumption $\pi_1\cong\pi_2$ also implies that the representations ${}^{l}\rho_{1,L}$ and $\rho_{2,L}$ of $S^{\natural}_{{}^{\dot{w}_2}\phi_2,L}$ are isomorphic, we have ${}^{l\dot{w}_{1}}\lambda_{1}={}^{\dot{w}_{2}}\lambda_{2}$ and ${}^{l\dot{w}_{1}}E_{1,L}\cong{}^{\dot{w}_{2}}E_{2,L}$.
Thus we see that ${}^{l\dot{w}_{1}}E_{1}\cong{}^{\dot{w}_{2}}E_{2}$ and conclude that ${}^{l\dot{w}_{1}}\rho_{1}\cong{}^{\dot{w}_{2}}\rho_{2}$.
\end{proof}

We denote by $\Pi_{\phi}(G_b)$ the set of all $\pi \in \Pi(G_b)$ attached to some $\rho \in \Irr(S_{\phi})$. As a result of Proposition \ref{prop: injectivity}, we can define a bijective map $\iota_{\mf{w}}$.
\begin{equation}{\label{eqn: iotaw}}
    \coprod\limits_{b \in B(G)} \Pi_{\phi}(G_b) \xrightarrow{\iota_{\mf{w}}} \Irr(S_{\phi}). 
\end{equation}
\begin{proposition}
The map $\iota_{\mf{w}}$ fits into a commutative diagram.
\begin{equation*}
     \begin{tikzcd}
    \coprod\limits_{b \in B(G)} \Pi_{\phi}(G_b) \arrow[d] \arrow[r, "\iota_{\mf{w}}"] & \Irr(S_{\phi}) \arrow[d] \\
    B(G) \arrow[r, "\kappa_G"]& X^{\ast}(Z(\widehat{G})^{\Gamma}).
    \end{tikzcd}
\end{equation*}
\end{proposition}

\begin{proof}
Suppose that $\pi_b\in\Pi_\phi(G_b)$ is mapped to $\rho\in\Irr(S_\phi)$ under the map $\iota_\mf{w}$.
Let $\omega_\rho\in X^\ast(Z(\wh{G})^\Gamma)$ be the image of $\rho$ under the map $\Irr(S_\phi)\rightarrow X^\ast(Z(\wh{G})^\Gamma)$, i.e., $Z(\wh{G})^\Gamma$ acts on $\rho$ via $\omega_\rho$.
Our task is to show that $\omega_\rho=\kappa_G(b)$.
In the following, we follow the notation of \S \ref{ss: themap}.

By our construction, $b\in B(G)$ is the image of $b_L\in B(L)_\bas^+$ in $B(G)$ and $\pi_b=\pi_{b_L}$ (under the identification $G_b\cong L_{b_L}$), where $\pi_{b_L}$ corresponds to $\rho_L\in\Irr(S^\natural_{{}^{\dot{w}}\phi,L})$ under the $B(L)_\bas$-LLC.
Let $\omega_{\rho_L}\in X^\ast(Z(\wh{L})^\Gamma)$ be the image of $\rho_L$ under the map $\Irr(S^\natural_{{}^{\dot{w}}\phi,L})\rightarrow X^\ast(Z(\wh{L})^\Gamma)$, i.e., $Z(\wh{L})^\Gamma$ acts on $\rho_L$ via $\omega_{\rho_L}$.
Then the commutativity in the basic case \eqref{eqn: basic correspondence} implies that $\omega_{\rho_L}$ is given by $\kappa_L(b_L)$.
By the functoriality of the Kottwitz homomorphism (see \cite[\S 4.9]{KottwitzisoII}), $\kappa_L(b_L)\in X^\ast(Z(\wh{L})^\Gamma)$ is mapped to $\kappa_G(b)\in X^\ast(Z(\wh{G})^\Gamma)$ under the natural map $X^\ast(Z(\wh{L})^\Gamma)\rightarrow X^\ast(Z(\wh{G})^\Gamma)$.
In other words, $\omega_{\rho_L}|_{Z(\widehat{G})^\Gamma}$ is given by $\kappa_G(b)$.
Hence it suffices to show that $\omega_\rho=\omega_{\rho_L}|_{Z(\widehat{G})^\Gamma}$, i.e., $Z(\wh{G})^\Gamma$ acts on both $\rho$ and $\rho_L$ via the same character.

Recall that $\rho\cong\mc{L}(\lambda,E)$.
Since the conjugate action of $Z(\wh{G})^\Gamma$ on $S_\phi$ is trivial, $Z(\wh{G})^\Gamma$ is contained the preimage $S_\phi^\lambda$ of $A^\lambda$ under the map $S_\phi\twoheadrightarrow\pi_0(S_\phi)$.
As $\mc{L}(\lambda,E)$ is defined to be the induction of $E\otimes\mc{L}(\lambda)$ from $S_\phi^\lambda$ to $S_\phi$, we see that $Z(\wh{G})^\Gamma$ acts on $\mc{L}(\lambda,E)$ and $E\otimes\mc{L}(\lambda)$ via the same character $\omega_\rho$.

Recall that, in \S \ref{ss: repthrydisconnected}, we choose a representative $\iota(a)$ of $a\in A^\lambda$ in $S_\phi^\lambda$ and an $S_\phi^\circ$-equivariant isomorphism $\theta_a \colon \mc{L}(\lambda)\xrightarrow{\cong}{}^{\iota(a)}\mc{L}(\lambda)$ such that $\iota(1)=1$ and $\theta_{1}=\mathrm{id}$.
Let $Z^\lambda$ be the image of $Z(\wh{G})^\Gamma\subset S_\phi$ in $A^\lambda$.
For any $a\in Z^\lambda$, we may and do choose $\iota(a)$ to be an element of $Z(\wh{G})^\Gamma$ and $\theta_{a}$ to be the identity map.
Then, for any element $z\in Z(\wh{G})^\Gamma$, its action on $u\otimes v\in E\otimes\mc{L}(\lambda)$ is given by 
\[
z\cdot (u\otimes v)=(\rho_a u)\otimes(gv),
\]
where $a$ denotes the image of $z$ in $Z^\lambda\subset A^\lambda$, $\rho_a$ is the associated element of $\mc{A}^\lambda$ (see \S \ref{ss: repthrydisconnected}), and $g:=\iota(a)^{-1}z\in S^{\circ}_{\phi}\cap Z(\wh{G})^\Gamma$.
Since $S^{\circ}_{\phi}\cap Z(\wh{G})^\Gamma$ is a central subgroup of the connected reductive group $S^{\circ}_{\phi}$, $S^{\circ}_{\phi}\cap Z(\wh{G})^\Gamma$ is contained in the maximal torus $A_{\wh{M}}$ of $S^\circ_\phi$.
In particular, we have $gv=\lambda(g)v$, hence we get $z\cdot (u\otimes v)=(\rho_a u)\otimes (\lambda(g)v)$.
By the same argument, we can also check that the action of $Z(\wh{G})^\Gamma$ on $\rho_L\cong\mc{L}_L({}^{\dot{w}}\lambda)\otimes{}^{\dot{w}}E_L$ and ${}^{\dot{w}}E_L$ is given by the same formula. 
\end{proof}

The surjectivity we remarked on at the end of \S\ref{ss: themap} gives us the desired finite-to-one map
\[
\LLC_{G_b}\colon\Pi(G_b) \rightarrow \Phi(G).
\]
\begin{corollary}{\label{cor: surjectivity}}
We have an equality of sets
\begin{equation*}
    \coprod\limits_{\phi} \Pi_{\phi}(G_b) = \Pi(G_b).
\end{equation*}
\end{corollary}

Let us also discuss the compatibility of our construction with duality.
Let $\wh{\mc{C}}$ be a Chevalley involution of $\wh{G}$ with respect to our fixed splitting $(\wh{T},\wh{B},\{\wh{X}_\alpha\})$ of $\wh{G}$, which extends to an involution ${}^L \mc{C}=\wh{\mc{C}}\rtimes\mathrm{id}$ of ${}^L G=\wh{G}\rtimes W_F$.
What we are interested in is the composite ${}^L\mc{C}\circ\phi$ of the involution ${}^L\mc{C}$ and an $L$-parameter $\phi\in\Phi(G)$.
Here note that $S_{{}^L\mc{C}\circ\phi}=\wh{\mc{C}}(S_\phi)$, hence we also have an isomorphism $\wh{\mc{C}}\colon S_\phi^\natural\cong S_{{}^L\mc{C}\circ\phi}^\natural$. 
The following is expected to be satisfied by the $B(G)_\bas$-LLC (see \cite[Section 2]{AV16} and also \cite{Kal13-gen} for more details):
\begin{conjecture}\label{conj:basic-duality}
Let $b\in B(G)_\bas$.
Suppose that an irreducible tempered representation $\pi_b\in\Pi(G_b)$ corresponds to $(\phi,\rho)$, where $\phi\in\Phi(G)$ and $\rho=\iota_{\mf{w}}(\pi_b)\in\Irr(S_\phi^\natural)$.
Then the $L$-parameter associated to the contragredient $\pi_b^\vee$ of $\pi_b$ is given by ${}^L\mc{C}\circ\phi$ and we have $\iota_{\mf{w}^{-1}}(\pi_b^\vee)=\rho^\vee\circ\wh{\mc{C}}^{-1}$.
Here, $\mf{w}^{-1}$ denotes the Whittaker datum whose Borel is the same as that of $\mf{w}$ but generic character is inverted.
\end{conjecture}

\begin{proposition}
Suppose that Conjecture \ref{conj:basic-duality} is true for $G$ and all standard Levi subgroups of $G$.
Let $b\in B(G)$ and $\pi_b\in\Pi(G_b)$ be an irreducible tempered representation.
If $\pi_b$ corresponds to $(\phi,\rho)$, where $\phi\in\Phi(G)$ and $\rho=\iota_{\mf{w}}(\pi_b)\in\Irr(S_\phi)$ under the map $\iota_{\mf{w}}$ constructed in \S\ref{ss: themap}, then the $L$-parameter associated to $\pi_b^\vee$ is given by ${}^L\mc{C}\circ\phi$ and we have $\iota_{\mf{w}^{-1}}(\pi_b)=\rho^\vee\circ\wh{\mc{C}}^{-1})$.
\end{proposition}

\begin{proof}
With the notation as in \S\ref{ss: theeasymap}, let $\phi$ be the $L$-parameter of $L$ and $\rho_L$ be the irreducible representation of $S_{\phi,L}^\natural$ associated to $\pi_b\cong\pi_{b_L}$ under the $B(L)_\bas$-LLC.
Since $\wh{\mc{C}}$ maps any root $\alpha$ of $\wh{T}$ to $-\alpha$ (see \cite[Section 2]{AV16}), ${}^L\mc{C}$ preserves any standard Levi subgroup of ${}^L G$ (in particular, ${}^L M$ and ${}^L L$) and induces the Chevalley involution with respect to the restriction of the splitting $(\wh{T},\wh{B},\{\wh{X}_\alpha\})$.
Thus, by Conjecture \ref{conj:basic-duality}, $\pi_b^\vee\cong\pi_{b_L}^\vee$ corresponds to $({}^L\mc{C}\circ\phi,\rho_L^\vee\circ\wh{\mc{C}}^{-1})$.
Hence the only task is to check that the representation of $S_\phi$ determined by $\rho_L^\vee\circ\wh{\mc{C}}^{-1}$ as in the manner of \S\ref{ss: theeasymap} is equal to $\rho^\vee\circ\wh{\mc{C}}^{-1}$.
But this directly follows from the construction (just note that $\wh{\mc{C}}$ also induces $S_{\phi,L}\cong S_{{}^L\mc{C}\circ\phi,L}$, $S_{\phi,M}\cong S_{{}^L\mc{C}\circ\phi,M}$, and so on).
\end{proof}

\section{Endoscopic character identity}{\label{s: eci}}

In this section, we restrict to the case where $F$ is a $p$-adic field. It seems to us that analogous results must hold for all local fields.

\subsection{Setup}{\label{ss: setup}}

Recall that \textit{a refined endoscopic datum $\mf{e}$ of $G$} is a tuple $(H, \mc{H}, s, \eta)$ consisting of \begin{itemize}
\item
$H$ is a quasi-split connected reductive group over $F$,
\item
$\mc{H}$ is a split extension of $W_F$ by $\wh{H}$ such that the induced action of $W_F$ on $\wh{H}$ coincides with the one coming from the $F$-rational structure of $\wh{H}$, 
\item
$s$ is an element of $Z(\wh{H})^\Gamma$, and
\item
$\eta\colon \mc{H}\rightarrow {}^{L}G$ is an $L$-homomorphism which restricts to an isomorphism $\hat{H}\rightarrow Z_{\wh{G}}(\eta(s))^\circ$
\end{itemize}
Recall also that an isomorphism of refined endoscopic data from $(H, \mc{H}, s, \eta)$ to $(H', \mc{H}', s', \eta')$ is an element $g \in \wh{G}$ such that
\begin{enumerate}
    \item we have $(\Int(g) \circ \eta)(\mc{H}) = \eta'(\mc{H}')$, and 
    \item $\Int(g)(\eta(s)) = \eta'(s')$.
\end{enumerate}
(see \cite[Definition 2.11]{BM2}, \cite[Definition 2.3.4]{BMShinIG2} and also \cite[\S 1.3 and \S 4.1]{KalethaLLCnonQS}). We let $\ms{E}^{\iso}(G)$ be the set of refined endoscopic data for $G$ and let $\msc{E}^{\iso}(G)$ denote the set of isomorphism classes.

We fix a refined endoscopic datum $\mf{e}=(H, \mc{H}, s, \eta)$ in the following.
For simplicity, we assume throughout that $\mc{H}={}^L{H}$.
We fix an $F$-splitting $(T_H, B_H, \{X_{H, \alpha}\})$ of $H$ and a $\Gamma$-stable splitting $(\wh{T}_{H}, \wh{B}_{H}, \{X_{\wh{H}, \alpha}\}$ of $\wh{H}$ in addition to the splittings of $G$ and $\wh{G}$ we fixed in \S\ref{s: preliminaries}. We assume that $\eta(\wh{T}_{H}) = \wh{T}$ and $\eta(\wh{B}_{H}) \subset \wh{B}$.

Temporarily fix $b \in B(G)_{\bas}$ and choose a cocycle $z \in Z^1_{\alg}(\mc{E}^{\iso}_F, G(\ov{F}))$ and $\varphi: G \to G_b$ such that $(G_b, \varphi, z)$ is an extended pure inner twist of $G$. Recall that we can define the notion of \textit{matching orbital integrals} between test functions $f_b\in C_c^\infty(G_b(F))$ and $f_H\in C_c^\infty(H(F))$.
For any test function $f_b\in C_c^\infty(G_b(F))$, there always exists a test function $f_H\in C_c^\infty(H(F))$ (\textit{transfer}) which has matching orbital integrals with $f_b$ (see \cite[Theorem 4]{KalethaLLCnonQS}).
Accordingly, for any stable distribution $D$ on $H(F)$, we may consider its \textit{transfer} $\Trans_H^{G_b} D$ to $G_b(F)$ by, for any test function $f_b\in C_c^\infty(G_b(F))$, 
\[
\Trans_{H}^{G_b} D(f_b)
:= D(f_H),
\]
where $f_H\in C_c^\infty(H(F))$ is a transfer of $f_b$ to $H(F)$. Note that the notion of transfer of functions (and distributions) requires fixing a transfer factor $\Delta[\mf{w}, z]$ depending on our fixed Whittaker datum $\mf{w}$ and cocycle $z$. We use the $\Delta^{\lambda}_D$-normalization as in \cite[\S5.5]{KS2}.

Let $\phi$ be a tempered $L$-parameter of $G$.
We assume that $\phi$ factors through $\eta$; let $\phi_H$ be an $L$-parameter of $H$ such that $\phi=\eta\circ\phi_H$.

In the following, we assume the existence of the basic case of the local Langlands correspondence (Conjecture \ref{conj: basic-LLC}).
Hence, by Theorem \ref{thm: main}, we have a bijective map
\[
\iota_\mf{w} \colon \coprod_{b\in B(G)} \Pi_\phi(G_b)\xrightarrow{1:1}\Irr(S_\phi)
\]
which extends the bijection of the $B(G)_{\bas}$-LLC
\[
\iota_{\mf{w}} \colon \coprod_{b\in B(G)_\bas} \Pi_\phi(G_b)\xrightarrow{1:1}\Irr(S_\phi^\natural).
\]
In the following, for any $\pi\in \Pi_\phi(G_b)$, we let $\langle\pi,-\rangle$ denote the irreducible character of $S_\phi$ corresponding to $\pi$ under $\iota_\mf{w}$, i.e., 
\[
\langle\pi,s\rangle
:=
\tr(s \mid \iota_\mf{w}(\pi))
\]
for $s\in S_\phi$.
For any $b\in B(G)$ and $s\in S_\phi$, we put
\begin{equation*}{\label{eq: basiceci}}
  \Theta_\phi^{G_b,s}
:=
e(G_b)\sum_{\pi\in\Pi_{\phi}(G_{b})}\langle\pi,s\rangle\Theta_\pi,  
\end{equation*}
where $e(G_b)$ denotes the Kottwitz sign of $G_b$.
When $s=1$, we write $S\Theta^{G_b}_{\phi}$ for $\Theta_\phi^{G_b,1}$. 

It is expected that the basic case of the local Langlands correspondence satisfies the \textit{stability} and the \textit{endoscopic character identity} (cf.\ \cite[Conjecture F]{KalethaLLCnonQS}).
We assume these properties in the following:

\begin{assumption}[stability and endoscopic character identity]\label{assumption:ECI-basic}
Let $b\in B(G)_\bas$.
\begin{enumerate}
\item
The distribution $S\Theta^H_{\phi_H}$ on $H(F)$ is stable.
\item
We have the following equality as distributions on $G_b(F)$:
\begin{align}\label{eq:ECI-basic}
\Trans_{H}^{G_b} S\Theta^H_{\phi_H}
=
\Theta_\phi^{G_b,\eta(s)}.
\end{align}
\end{enumerate}
\end{assumption}

\begin{remark}
    As a sanity check, we observe that if we multiply $s$ by the pre-image of an element $c \in Z(\wh{G})^{\Gamma}$, then the right-hand side of \eqref{eq: basiceci} is multiplied by $\kappa(b)(c)$ because $\iota_{\mf{w}}(\pi)|_{Z(\wh{G})^{\Gamma}} = \kappa(b)^{\oplus\dim \iota_{\mf{w}}(\pi)}$. 
    On the left-hand side, multiplying $s$ by $c$ does not change $H$, but it does change the transfer factor and hence the notion of transfer of functions between $G_b$ and $H$. It is relatively simple to check that the transfer factor is multiplied by the quantity $\langle \inv[z], c \rangle = \kappa(b)(c)$.
\end{remark}

Our aim in this section is to generalize the identity \eqref{eq:ECI-basic} to any $b\in B(G)$.
For this, we additionally assume the following standard properties of the basic case of the local Langlands correspondence.

\begin{assumption}\label{assumption:LIR-basic}
Let $Q$ be a standard parabolic subgroup of $G$ with standard Levi $L$.
If a tempered $L$-parameter $\phi$ of $G$ factors through the $L$-embedding of ${}^{L}L$ into ${}^{L}G$, then we have the following equality as distributions on $G(F)$:
\[
S\Theta^G_{\phi} =  I^G_Q(S\Theta^L_{\phi}).
\]
\end{assumption}

\begin{remark}
By using the transitivity of parabolic induction, we can reduce the property of Assumption \ref{assumption:LIR-basic} to the case when $L$ is a minimal Levi subgroup through which $\phi$ factors, i.e., $\phi$ is discrete as an $L$-parameter of $L$.
Then, this is a special case of the local intertwining relation (see \cite[Theorem 2.4.1]{Arthurbook}, for instance). 
\end{remark}

\begin{assumption}\label{assumption:twist}
Suppose that $\alpha\colon G\rightarrow G$ is an $F$-rational automorphism of $G$.
Let ${}^L\alpha\colon {}^LG\rightarrow {}^LG$ be the dual to $G$.
Then, for any $L$-parameter $\phi\colon L_F\rightarrow {}^LG$, we have
\[
\Pi_{{}^L\alpha\circ\phi}(G)=\alpha^\ast\Pi_{\phi}(G),
\]
where $\alpha^\ast\Pi_{\phi}(G)$ denotes the pull-back of $\Pi_{\phi}(G)$ via $\alpha\colon G(F)\rightarrow G(F)$.
\end{assumption}

\begin{remark}
Assumption \ref{assumption:twist} should be standard (see, for example, \cite[Conjecture 4.9]{Haines}) and can be also thought of as a special case of the compatibility of the local Langlands correspondence with isogeny; for example, see \cite[10.3 (5)]{Bor2}, \cite[\S IX.6.1]{fargues--scholze}, \cite[Th\'eor\`eme 0.1]{Genestier--Lafforgue}, etc.
\end{remark}

\subsection{Motivation}{\label{ss: motivation}}

We now describe what we believe is the correct way to formulate the endoscopic character identity for a general $b \in B(G)$.
To begin, we want to define a transfer of functions from $C^{\infty}_c(G_b(F))$ to $C^{\infty}_c(H(F))$ for any $b \in B(G)$ and an endoscopic group $H$ of $G$.
We suspect this will not be possible in full generality, but it will be for $\nu_b$-\emph{acceptable} functions. We recall their definition (see \cite[\S2.7]{BMShinIG2}).

Let $\nu: \D_F \to G$ be a homomorphism of groups and let $M_{\nu}$ be the centralizer of $\nu$ in $G$. 
The homomorphism $\nu$ defines a parabolic subgroup $P_{\nu} = M_{\nu}N_{\nu}$ whereby the positive roots of $P_{\nu}$ are those such that $\langle \nu, \alpha \rangle <0$. 
\begin{warning}
We often take $\nu$ to be $\nu_b := \nu_G(b)$ and the opposite parabolic $P_\nu^\op$ is standard in this case.
\end{warning}
We say that $\gamma \in M_{\nu}(\ov{F})$ is $\nu$-\emph{acceptable} if the adjoint action of $\gamma$ on $N_{\nu}(\ov{F})$ is dilating, namely each eigenvalue $\lambda$ of this action satisfies $|\lambda|>1$.
The set of $\nu$-acceptable elements is nonempty and open in $M_{\nu}(\ov{F})$.  Since $\nu$-acceptability only depends on the stable conjugacy class of $\gamma$ in $M_\nu$, we can define for an inner twist $\varphi_M : M_{\nu} \to M'_{\nu}$ that $\gamma' \in M'_{\nu}(\ov{F})$ is $\nu$-acceptable if $\varphi^{-1}(\gamma')$ is $\nu$-acceptable. 
We let $C^{\infty}_{c, \acc}(M_{\nu}(F)) \subset C^{\infty}_c(M_{\nu}(F))$ (resp. $C^{\infty}_{c, \acc}(M'_{\nu}(F)) \subset C^{\infty}_c(M'_{\nu}(F))$) denote the subset of functions supported on $\nu$-acceptable elements. We remark that there are enough $\nu$-acceptable functions to separate $\Pi(M'_{\nu})$ (see the argument of \cite[Lemma 6.4]{Shi4}, cf.\ \cite[Lemma 2.7.5]{BMShinIG2}) so it is sufficient to restrict our attention to them. The relevant proposition is as follows.
\begin{proposition}[{\cite[Lemma 3.1.2]{KretShinIG}}]{\label{prop: shinascentlem}}
    Let $f_{\nu} \in C^{\infty}_{c, \acc}(M_{\nu}(F))$. Then there exists an $f \in C^{\infty}_c(G(F))$ satisfying the following properties.
\begin{itemize}
    \item  For every semisimple element $g \in G(F)$, we have the following identity of orbital integrals
    \begin{equation*}
        O^G_g(f) = \delta^{-1/2}_{P_{\nu}}(m)\cdot O^{M_{\nu}}_m(f_\nu),
    \end{equation*}
    if there exists a $\nu$-acceptable $m \in M_{\nu}(F)$ that is conjugate to $g \in G(F)$ and $O^G_g(f)=0$ otherwise.
    \item We have
    \begin{equation*}
        \tr(f \mid \pi) = \tr( f_\nu \mid J^G_{P^{\op}_{\nu}}(\pi)),
    \end{equation*}
    for $\pi \in \Pi(G)$.
\end{itemize}
\end{proposition}

We need to study the relation between the endoscopy of $G$ and its Levi subgroups. Fix $L \subset G$ a standard Levi subgroup of $G$ (later, especially, we take $L$ to be the standard Levi subgroup of $G$ such that $G_b$ is its inner twist). 

\begin{definition}\label{def:emb-end-datum}
    An \emph{embedded endoscopic datum} for $G$ is a tuple $(H_L, \mc{H}_L, H, \mc{H}, s, \eta)$, where 
\begin{itemize}
    \item $(H, \mc{H}, s, \eta)$ is a refined endoscopic datum of $G$ with a fixed $F$-splitting $(T_H, B_H, \{X_{H,\alpha}\})$ of $H$,
    \item $H_L$ is a standard Levi subgroup of $H$,
    \item $\mc{H}_L$ is a Levi subgroup of $\mc{H}$, namely $\mc{H}_L$ surjects onto $W_F$ and its intersection with $\wh{H}$ is a Levi subgroup of $\wh{H}$,
\end{itemize}    
    such that $\wh{H_L} = \mc{H}_L \cap \wh{H}$ and $(H_L, \mc{H}_L, s, \eta|_{\mc{H}_L})$ is a refined endoscopic datum of $L$. 

    An isomorphism of embedded data from $(H_L, \mc{H}_L, H, \mc{H}, s, \eta)$ to $(H'_L, \mc{H}'_L, H', \mc{H}', s', \eta')$ is a $g \in \wh{G}$, which simultaneously produces isomorphisms 
    $$(H_L, \mc{H}_L, s, \eta|_{\mc{H}_L}) \xrightarrow{\sim} (H'_L, \mc{H}'_L, s', \eta'|_{\mc{H}'_L})\quad \mbox{and} \quad (H, \mc{H}, s, \eta) \xrightarrow{\sim} (H', \mc{H'}, s', \eta').$$
    We denote the set of embedded endoscopic data by $\ms{E}^{\emb}(L,G)$ and the set of isomorphism classes by $\msc{E}^{\emb}(L,G)$.
    \end{definition}

We have the natural restrictions $X: \ms{E}^{\emb}(L,G) \to \ms{E}^{\iso}(L)$ and $Y^{\emb}: \ms{E}^{\emb}(L,G) \to \ms{E}^{\iso}(G)$. These induce maps of isomorphism classes, and the map induced by $X$ is a bijection 
by \cite[Proposition 2.20]{BM2}. We recall from \cite[Construction 2.15]{BM2} that there is a natural map $Y: \ms{E}^{\iso}(L) \to \ms{E}^{\iso}(G)$ such that the following diagram commutes
\begin{equation}
\begin{tikzcd}
&\msc{E}^{\iso}(G)&\\
\msc{E}^{\emb}(L,G) \arrow[ur, " Y^{\emb}"] \arrow[rr, "X"] && \msc{E}^{\iso}(L) \arrow[ul, swap, "Y"].
\end{tikzcd}
\end{equation}

\begin{definition}
For a refined endoscopic datum $(H, \mc{H}, s, \eta)$ of $G$, we define $\msc{E}^{\emb}(L,G;H)$ to be the set of isomorphism classes of embedded endoscopic data whose image under 
\begin{equation}
Y^{\emb}: \msc{E}^{\emb}(L,G) \to \msc{E}^{\iso}(G)
\end{equation}
is the isomorphism class of $(H,\mc{H}, s,\eta)$. 
We define the set of \emph{inner classes of embedded endoscopic data} relative to $H$, denoted by $\msc{E}^i(L,G;H)$, to be the set of equivalence classes of elements of the form $(H_L, \mc{H}_L, H, \mc{H}, s, \Int(n) \circ \eta)$ of $\ms{E}^{\emb}(L,G)$, for $n \in N_{\wh{G}}(\wh{T})$. The isomorphism class of such elements lies in $\msc{E}^{\emb}(L,G;H)$ and two such data are considered equivalent if they are isomorphic by an  inner isomorphism $\alpha$ of the group $H$ inducing an isomorphism of embedded endoscopic data.
\end{definition}

In the following, we fix a refined endoscopic datum $\mf{e} = (H, \mc{H}, s, \eta)$.
Although we believe that our result can be established for general $\mc{H}$, we focus only on the case $\mc{H}={}^LH$ in the following.
We also fix $b \in B(G)$ and an extended pure inner twist $(G_b, \varphi, z)$ of $L$, where $L \subset G$ is the standard Levi subgroup given by the centralizer of $\nu_b := \nu_G(b)$ with standard parabolic $Q$ and $z$ is a cocycle corresponding to $b_L \in B(L)^+_{\bas}$.

We furthermore fix $X^{\mf{e}}_L$, a set of representatives of $\msc{E}^i(L,G;H)$. For each $\mf{e}_L \in X^{\mf{e}}_L$, we get a character $\nu_{\mf{e}_L}: \D_F \xrightarrow{\nu_b} A_{L} \subset T' \cong T_H$ where $T' \subset G$ and we note the isomorphism $T' \cong T_H$ is determined by $\mf{e}_L$ and canonical up to our choice of splittings.  The following diagram records the relationships between the various groups that appear.

\begin{equation}{\label{eq: endodiagram}}
\begin{tikzcd}
&& G \\
G_b \arrow[r, "\text{inner}", leftrightarrow] & L \arrow[ru, "\text{Levi}", hook]& H \arrow[u, "\text{endo.}"', dash]\\
&H_L \arrow[lu, "\text{endo.}", dash] \arrow[ru, "\text{Levi}"', hook] \arrow[u, "\text{endo.}" description, dash]&
\end{tikzcd}
\end{equation}

Fix $f_b \in C^{\infty}_{c, \acc}(G_b(F))$. We produce a matching $f_H \in C^{\infty}_c(H(F))$. 
\begin{enumerate}
\item Define $f_b^0 := f_b \otimes \ov{\delta}^{1/2}_{P_{\nu_b}}$, where $\ov{\delta}_{P_{\nu_b}}$ is the character on $G_b$ defined such that $\ov{\delta}_{P_{\nu_b}}(\gamma') = \delta_{P_{\nu_b}}(\gamma)$ for $\gamma \in L(F)$ matching $\gamma' \in G_b(F)$.
\item For each $\mf{e}_L \in X^{\mf{e}}_L$, define $f_{\mf{e}_L}\in C_c^\infty(H_L)$ to be a transfer of $f_b^0$ from $G_b$ to $H_L$ using the Whittaker normalized $\Delta[\mf{w}_L, z]$ transfer factor (we use the $\Delta^{\lambda}_D$ normalization as in \cite[\S5.5]{KS2}, these transfer factors are explained in \cite[\S3]{BM3} generalizing \cite[(4.3)]{KalTai}, though note that \cite{KalTai} uses the $\Delta'_{\lambda}$ normalization). By multiplying with the indicator function on the set of $\nu_{\mf{e}_L}$-acceptable elements, we can and do assume that $f_{\mf{e}_L} \in C^{\infty}_{c, \acc}(H_L(F))$. Note that the Levi subgroup of $H$ determined by $\nu_{\mf{e}_L}$ is precisely $H_L$. 
\item We now apply Proposition \ref{prop: shinascentlem} to each $f_{\mf{e}_L}\in C^{\infty}_{c, \acc}(H_L(F))$ to get functions $f_{H,\mf{e}_L}\in C^{\infty}_c(H(F))$. 
\item We finally let $f_H = \sum\limits_{X^{\mf{e}}_L} f_{H,\mf{e}_L}\in C^{\infty}_c(H(F))$.
\end{enumerate}

Now take $\gamma_H \in H(F)$ that is $G$-strongly regular semisimple.
We compute the stable orbital integral $SO^H_{\gamma_H}(f_H)$.
If there is no $\mf{e}_L \in X^{\mf{e}}_L$ and $\nu_{\mf{e}_L}$-acceptable $\gamma_{H_L} \in H_L(F)$ conjugate to $\gamma_H$ in $H(F)$, then $SO^H_{\gamma_H}(f_H) = 0$ by Proposition \ref{prop: shinascentlem}.
Otherwise, we have that
\begin{equation}{\label{eqn: HHLascent}}
    SO^H_{\gamma_H}(f_H) = \sum\limits_{\mf{e}_L} \delta^{-1/2}_{P_{\nu_{\mf{e}_L}}}(\gamma_{H_L})\cdot SO^{H_L}_{\gamma_{H_L}}(f_{\mf{e}_L}),
\end{equation}
where the sum is over some subset of $X^{\mf{e}}_L$. 
Here we used the fact that the identity of orbital integrals in Proposition \ref{prop: shinascentlem} induces the identity of stable orbital integrals (\cite[Lemma 3.5]{Shi3}).
Crucially, by \cite[Lemma 2.7.13]{BMShinIG2} (cf.\ \cite[Lemma 6.2]{Shi3}, \cite[Lemma 2.42]{BM2}) there is at most one $\mf{e}_L$ appearing on the right hand side of \eqref{eqn: HHLascent}. If $\gamma_{H_L}$ for such $\mf{e}_L$ does not transfer to some $\gamma_{G_b} \in G_b(F)$ whose image $\gamma_L$ in $L(F)$ is $\nu_b$-acceptable, then the original $SO^H_{\gamma_H}(f_H)$ is $0$. Otherwise, we get
\begin{align*}
    SO^H_{\gamma_H}(f_H)
    &=
    \delta^{-1/2}_{P_{\nu_{\mf{e}_L}}}(\gamma_{H_L}) \cdot SO^{H_L}_{\gamma_{H_L}}(f_{\mf{e}_L})\\
    &=
    \sum\limits_{\gamma'_{G_b} \sim_{\st}  \gamma_{G_b}} \Delta[\mf{w}_L, z](\gamma_{H_L}, \gamma'_{G_b}) \delta^{-1/2}_{P_{\nu_{\mf{e}_L}}}(\gamma_{H_L})\ov{\delta}^{1/2}_{P_{\nu_b}}(\gamma'_{G_b})O^{G_b}_{\gamma'_{G_b}}(f_b).
\end{align*}
The formula 
\[
\frac{|\det( \Ad(\gamma_L) -1 \mid \Lie(G) / \Lie(L))|^{1/2}}{|\det( \Ad(\gamma_{H_L}) -1 \mid \Lie(H) / \Lie(H_L))|^{1/2}}\Delta[\mf{w}_L,z](\gamma_{H_L}, \gamma_L) =  \Delta[\mf{w},z](\gamma_{H_L}, \gamma_L)
\]
(see \cite[Proposition 5.3]{BM2} for instance) and the facts that 
\begin{itemize}
    \item
    $|\delta_{P_{\nu_{\mf{e}_L}}}(\gamma_{H_L})| = |\det( \Ad(\gamma_{H_L}) -1 \mid \Lie(H) / \Lie(H_L))|$ and
    \item 
    $|\delta_{P_{\nu_b}}(\gamma_L)| = |\det( \Ad(\gamma_L) -1 \mid \Lie(G) / \Lie(L))|$
\end{itemize}
(see \cite[Lemma 3.4]{Shi3}) imply that finally:
\begin{equation}{\label{eqn: matchingfn}}
    SO^H_{\gamma_H}(f_H) = \sum\limits_{\gamma'_{G_b} \sim_{\st}  \gamma_{G_b}} \Delta[\mf{w},z](\gamma_{H_L}, \gamma_L)\langle \inv[z](\gamma_L, \gamma_{G'_b}), \wh{\varphi}_{\gamma_{H_L},\gamma_L}(s) \rangle^{-1}O^{G_b}_{\gamma'_{G_b}}(f_b),
\end{equation}
where $\wh{\varphi}_{\gamma_{H_L},\gamma_L}$ is the dual of the admissible isomorphism taking $Z_{H_{L}}(\gamma_{H_L})$ to $Z_{L}(\gamma_L)$ (cf.\ \cite[\S 4.1]{BM2}).
This is our notion of \textit{matching function}.
Corresponding to this notion of matching function, we get a \textit{transfer} of distributions; we say that an invariant distribution $D_b$ on $G_b(F)$ is a transfer of a stable distribution $D_H$ on $H(F)$ if they satisfy $D_b(f_b)=D_H(f_H)$ for any $f_b\in C_{c,\acc}^\infty(G_b(F))$ and any its matching $f_H\in C_c^\infty(H(F))$.

We remark that this definition of a transfer of distributions does not induce a map from the set of stable distributions on $H(F)$ to the set of invariant distributions on $G_b(F)$.
The problem is that the subspace of $\nu_b$-acceptable functions $C_{c,\acc}^\infty(G_b(F))$ is too small to specify an invariant distribution on $G_b(F)$ uniquely.
The following example was given by the anonymous referee:

\begin{example}
Let $G=\GL_2$ over $\Q_p$.
We consider the case where $\nu\in X_\ast(T)$ is given by $\nu(x)=\mathrm{diag}(x,1)$, hence $G_b$ is the diagonal maximal torus $T$.
Let $D\colon C_c^\infty(T(\Q_p))\rightarrow\C$ be the following distribution:
\[
D(f):=\int_{T_1}f(x)\,dx,
\]
where $T_1:=\{\mathrm{diag}(x,y)\in T(\Q_p)\mid x,y\in 1+p\Z_p\}$.
Then $D$ is obviously invariant (even stable) since $T(\Q_p)$ is abelian.
Moreover, $D$ maps any $\nu$-acceptable function to $0$.
In other words, we cannot distinguish $D$ from $0$ by looking at the values on $C_{c,\acc}^\infty(G_b(\Q_p))$.
\end{example}

The point of the above example is that the distribution $D$ considered there is not a virtual character.
As mentioned in the paragraph before Proposition \ref{prop: shinascentlem}, any virtual character is determined uniquely by its values on the set of  $\nu$-acceptable functions.
In other words, for a given stable distribution $D_H$ on $H(F)$, its transfer to $G_b(F)$ \textit{which is a virtual character} is unique if it exists.

In fact, for the stable distribution $S\Theta^H_{\phi_H}$ on $H(F)$, we can construct its unique transfer to $G_b(F)$ which is a virtual character by hand as follows.

\begin{definition}
We define a virtual character $\Trans^{G_b}_H S\Theta^H_{\phi_H}$ of $G_b(F)$ by
\begin{equation}{\label{eqn: endoscopiccharacter}}
     \Trans^{G_b}_H S\Theta^H_{\phi_H}
     :=
     \sum\limits_{\mf{e}_L \in X^{\mf{e}}_L} (\Trans^{G_b}_{H_L} J^H_{P^{\op}_{\nu_{\mf{e}_L}}}S\Theta^H_{\phi_H}) \otimes \ov{\delta}^{1/2}_{P_{\nu_b}}.
\end{equation}
Here, note that the right-hand side makes sense since the normalized Jacquet functor preserves the stability \cite[Lemma 3.3]{Hir1} (and also virtual characters), hence $J^H_{P^{\op}_{\nu_{\mf{e}_L}}}S\Theta^H_{\phi_H}$ is a stable distribution on $H_L(F)$, to which the endoscopic transfer in the basic case is applicable.
\end{definition}
\begin{remark}
    One could instead define $ \Trans^{G_b}_H S\Theta^H_{\phi_H}$ omitting $\ov{\delta}^{1/2}_{P_{\nu_b}}$. This has the effect of removing a number of modulus twists, for instance in the statement of Theorem \ref{thm: ECI}. However, one would have to modify the construction of $f_H$, by deleting the first step, and then adding a twist to Equation \eqref{eqn: matchingfn}. The function $f_H$ and Equation \eqref{eqn: matchingfn} as they appear in this article show up naturally in the stable trace formula for Igusa varieties and are compatible with \cite{Shi3}, which explains our slightly more complicated definition.
\end{remark}

\begin{lemma}\label{lem:Trans-Jac}
The virtual character $\Trans^{G_b}_H S\Theta^H_{\phi_H}$ of $G_b(F)$ is a transfer of the stable distribution $S\Theta^H_{\phi_H}$ on $H(F)$.
\end{lemma}

\begin{proof}
We fix $f_b \in C^{\infty}_{c, \acc}(G_b(F))$ and its transfer $f_H \in C_c^{\infty}(H(F))$.
If we let $f_b^0$, $f_{\mf{e}_L}$, $f_{H,\mf{e}_L}$ be intermediate test functions as explained above, then we have
\begin{align*}
\sum\limits_{\mf{e}_L \in X^{\mf{e}}_L} (\Trans^{G_b}_{H_L} J^H_{P^{\op}_{\nu_{\mf{e}_L}}}S\Theta^H_{\phi_H}) \otimes \ov{\delta}^{1/2}_{P_{\nu_b}}(f_b)
&=
\sum\limits_{\mf{e}_L \in X^{\mf{e}}_L} (\Trans^{G_b}_{H_L} J^H_{P^{\op}_{\nu_{\mf{e}_L}}}S\Theta^H_{\phi_H})(f_b^0)\\
&=
\sum\limits_{\mf{e}_L \in X^{\mf{e}}_L} (J^H_{P^{\op}_{\nu_{\mf{e}_L}}}S\Theta^H_{\phi_H})(f_{\mf{e}_L})\\
&=
\sum\limits_{\mf{e}_L \in X^{\mf{e}}_L} S\Theta^H_{\phi_H}(f_{H,\mf{e}_L})
=
S\Theta^H_{\phi_H}(f_{H}),
\end{align*}
where we used Proposition \ref{prop: shinascentlem} (2) in the third equality.
\end{proof}

A naive expectation is that the identity \eqref{eq:ECI-basic} holds also for non-basic $b\in B(G)$ with this definition of $\Trans_H^{G_b}S\Theta_{\phi_H}^H$.
However, this is not true.
Let us explain the difficulty.

Consider the simplest case where $(H,\mc{H},s,\eta)=(G,{}^{L}G,1,\mathrm{id})$.
In this case, the set $X^{\mf{e}}_L$ is a singleton whose unique element can be taken to be $(L,{}^{L}L,G,{}^{L}G,1,\mathrm{id})$.
Note that the standard parabolic subgroup $P_\nu^\op=P_{\nu_b}^\op$ associated to this unique embedded endoscopic datum is given by $Q$.
Hence, by Lemma \ref{lem:Trans-Jac}, the identity \eqref{eq:ECI-basic} would become
\begin{align}\label{eq: ECI-fake}
(\Trans_{L}^{G_b}J^G_{Q} S\Theta^G_{\phi})\otimes\ov{\delta}^{-1/2}_Q
=
\sum_{\pi\in\Pi_{\phi}(G_{b})}
\langle\pi,1\rangle\Theta_\pi.
\end{align}
Let us explain how this identity fails in the following two examples.

\begin{example}\label{ex:first}
Let $G=\GL_2$.
We take $\phi$ to be the direct sum $\mathbbm{1}\oplus\mathbbm{1}$ of two trivial representations of $W_F\times\SL_2(\C)$.
Then we have $S_\phi=\GL_2(\C)$.
Suppose that $\rho$ is an irreducible representation of $S_\phi$ which is not $1$-dimensional.
Then the element $b\in B(G)$ associated to $\rho$ is non-basic and $G_b=L=T$.
Since $\Pi_\phi(G_b)$ is a singleton consisting of $\mathbbm{1}\boxtimes\mathbbm{1}$, we have
\[
\sum_{\pi\in\Pi_{\phi}(G_{b})}
\langle\pi,1\rangle\Theta_\pi
=
\dim(\rho)\Theta_{\mathbbm{1}\boxtimes\mathbbm{1}}.
\]
On the other hand, $\Pi_\phi(G)$ is a singleton consisting of $I_B^G(\mathbbm{1}\boxtimes\mathbbm{1})$.
We have $Q=B$ and can check that 
\[
(\Trans_{T}^{G_b}J^G_{B} S\Theta^G_{\phi})\otimes\ov{\delta}^{-1/2}_B
=
2\Theta_{\mathbbm{1}\boxtimes\mathbbm{1}}\otimes\ov{\delta}^{-1/2}_B
\]
(for example, by the geometric lemma (\cite[p.\ 448]{Ber1})).
Thus, firstly, this example suggests that it would be better to twist the $G_b$-side $\sum_{\pi\in\Pi_{\phi}(G_{b})}\langle\pi,1\rangle\Theta_\pi$ via the character $\ov{\delta}_{B}^{-1/2}$.
Secondly, even if we make this modification, the equality \eqref{eq: ECI-fake} does not hold unless $\dim(\rho)=2$.
\end{example}

\begin{example}\label{ex:second}
Let $G=\GL_4$.
We take $\phi$ to be the direct sum $\mathrm{Std}\oplus\mathrm{Std}$ of two standard representations of $\SL_2(\C)$ (trivial on the $W_F$-part).
Then we have $S_\phi\cong\GL_2(\C)$.
Suppose that $\rho$ is an irreducible representation of $S_\phi$ which is not $1$-dimensional.
Then the element $b\in B(G)$ associated to $\rho$ is non-basic and $G_b$ is given by an inner form of the standard Levi subgroup $L=\GL_2\times\GL_2$ of $G$.
Since $\Pi_\phi(G_b)$ is a singleton consisting of $\Trans_L^{G_b}\mathrm{St}_2\boxtimes\mathrm{St}_2$, we have
\[
\sum_{\pi\in\Pi_{\phi}(G_{b})}
\langle\pi,1\rangle\Theta_\pi
=
\dim(\rho)\Trans_L^{G_b}\Theta_{\mathrm{St}_2\boxtimes\mathrm{St}_2},
\]
where $\mathrm{St}_2$ denotes the Steinberg representation of $\GL_2(F)$.
On the other hand, $\Pi_\phi(G)$ is a singleton consisting of $I_Q^G(\mathrm{St}_2\boxtimes\mathrm{St}_2)$, where $Q$ is the standard parabolic subgroup of $G$ with Levi part $L$.
By using the geometric lemma as before, we can check that 
\[
J^G_Q S\Theta^G_{\phi}
=
2\Theta_{\mathrm{St}_2\boxtimes\mathrm{St}_2}
+
\Theta_{I_B^{\GL_2}(|-|^{\frac{1}{2}}\boxtimes|-|^{\frac{1}{2}}) \boxtimes I_B^{\GL_2}(|-|^{-\frac{1}{2}}\boxtimes|-|^{-\frac{1}{2}})}.
\]
Thus the equality \eqref{eq: ECI-fake} cannot hold even if we twist the $G_b$-side via $\ov{\delta}_Q^{-1/2}$ and if $\dim(\rho)=2$ because of an extra term in $J^G_Q S\Theta^G_{\phi}$.
\end{example}

What we will do in the following is to modify the identity \eqref{eq:ECI-basic} so that the problems as in the above examples are resolved.

On the $G_b$-side, we introduce a quantity $\langle\pi,-\rangle_\reg$ and replace $\langle\pi,-\rangle$ in $\Theta_\phi^{G_b,\eta(s)}$ with $\langle\pi,-\rangle_\reg$.
In the cases of Examples \ref{ex:first} and \ref{ex:second}, we get $\langle\pi,1\rangle_\reg=2$ for any $\pi$ whose $\rho\in \Irr(S_\phi)$ is not $1$-dimensional.
Moreover, we consider the character twist via $\ov{\delta}_{P_{\nu_b}}^{1/2}$.

On the $H$-side, we define the \textit{regular part} $[\Trans^{G_b}_{H}S\Theta_{\phi_H}^H]_{\reg}$ of $\Trans^{G_b}_{H}S\Theta_{\phi_H}^H$ by simply cutting off some part of the sum obtained after applying the geometric lemma (Definition \ref{def: endregpart}).
In the case of Example \ref{ex:first}, nothing changes by this procedure; in the case of Example \ref{ex:second}, the second term of $J^G_Q S\Theta^G_{\phi}$ is non-regular and thrown away.

The following is the main result of this section, which will be proved in \S \ref{ss:proof of ECI}.

\begin{theorem}\label{thm: ECI}
For any $b\in B(G)$, we have the following equality as distributions on $G_b(F)$:
\begin{align}\label{eq:B(G)-ECI}
[\Trans^{G_b}_{H}S\Theta_{\phi_H}^H]_{\reg}
=
e(G_b)\sum_{\pi\in\Pi_{\phi}(G_{b})}\langle\pi,\eta(s)\rangle_\reg \Theta_{\pi}\otimes\ov{\delta}^{1/2}_{P_{\nu_b}}
\end{align}
\end{theorem}

\begin{remark}
It is a natural attempt to seek a formulation of the endoscopic character identity such that the non-regular part is not discarded.
However, we do not pursue this direction in this paper.
Note that it is expected that the $L$-packet of a supercuspidal $L$-parameter (i.e., discrete and trivial on $\SL_2(\C)$-part) contains only supercuspidal representations (cf.\ \cite[Proposition 4.27]{Haines}).
This implies that when $\phi$ has trivial $\SL_2$-part, the regular part is everything.
For general $\phi$, we just remark that the non-regular part can be quite complicated (cf.\ \cite{AtobeJacquet}).
\end{remark}

\subsection{Preliminaries on the Weyl groups}

For any $F$-rational standard Levi subgroups $L_1$ and $L_2$ of $G$, we put
\begin{itemize}
\item 
$W^{\rel}(L_1,L_2):=\{w\in W^\rel \mid w(A_{L_1})\supset A_{L_2}\}=\{w\in W^\rel \mid w(L_1)\subset L_2\}$,
\item
$W^{\rel,L_1,L_2}:=\{ w \in W^{\rel} \mid w(L_1 \cap B) \subset B , w^{-1}(L_2 \cap B) \subset B\}$, and
\item 
$W^\rel[L_1,L_2]:=W^\rel(L_1,L_2) \cap W^{\rel,L_1,L_2}$.
\end{itemize}
We note that $W^\rel[L_1,L_2]$ gives a complete set of representatives of the double cosets $W^{\rel}_{L_2} \backslash W^\rel(L_1,L_2) / W^{\rel}_{L_1}$ (see \cite[Lemma 2.11]{Ber1}).
Also note that we have $W^{\rel}_{L_2}wW^{\rel}_{L_1}=W^{\rel}_{L_2}wW^{\rel}_{L_1}w^{-1}w=W^{\rel}_{L_2}w$ for any $w \in W^\rel(L_1,L_2)$, hence we have $W^{\rel}_{L_2} \backslash W^\rel(L_1,L_2) / W^{\rel}_{L_1}=W^{\rel}_{L_2} \backslash W^\rel(L_1,L_2)$.
%\abm{Im nervous that we have switched to $W^{\rel}_L \setminus W^{\rel}(M,L)$ but we should really be using $W^{\rel}(M,L) \cap W^{M,L}$.}

On the dual side, similarly, we put
\[
\wh{W}^{\rel}(L_1,L_2):=\{w\in \wh{W}^\rel \mid w(A_{\wh{L}_{1}})\supset A_{\wh{L}_{2}}\}
\]
for any standard Levi subgroups ${}^{L}L_1$ and ${}^{L}L_2$ of ${}^{L}G$.
The condition $w(A_{\wh{L}_{1}})\supset A_{\wh{L}_{2}}$ is equivalent to $w({}^{L}L_1)\subset{}^{L}L_2$ by \cite[\S0.4.1]{KMSW}.

Note that the identification $W^\rel\cong\wh{W}^\rel$ induces $W^\rel(L_1,L_2)\cong\wh{W}^\rel(L_1,L_2)$ for any standard Levi subgroups $L_1$, $L_2$.

\begin{lemma}\label{lem: Weylgroups}
The image of the map $W_{\wh{G}}(A_{\wh{M}}) \hookrightarrow \wh{W}^\rel$ (see Lemma \ref{lem: weylinj}) is contained in $\wh{W}^{\rel}(M,M)$.
%and has trivial intersection with $\wh{W}^\rel_M:=W_{\wh{M}}(A_{\wh{T}})\subset \wh{W}^\rel$.
In particular, for any standard Levi subgroup $L$ of $G$, the set $\wh{W}^{\rel}(M,L)$ is stable under the right $W_{\wh{G}}(A_{\wh{M}})$-translation.
\end{lemma}

\begin{proof}
Let $w$ be an element of $W_{\wh{G}}(A_{\wh{M}})$ with a lift $n\in N_{\wh{G}}(A_{\wh{M}})$.
Recall that the image of $w$ in $\wh{W}^\rel=W_{\wh{G}}(A_{\wh{T}})$ is given by the class of $m^{-1}n\in N_{\wh{G}}(A_{\wh{T}})$, where $m\in\wh{M}$ is an element such that $m^{-1}n$-conjugation preserves the Borel pair $(\wh{T}, \wh{B}_{\wh{M}})$ of $\wh{M}$.
In particular, $m^{-1}n$-conjugation preserves $A_{\wh{M}}$.
Hence we get the first assertion.

Since $\wh{W}^{\rel}(M,L)$ is stable under right $\wh{W}^{\rel}(M,M)$-translation, the second assertion follows from the first one.
%Moreover, if $w\in W_{\wh{M}}(A_{\wh{T}})\subset W_{\wh{G}}(A_{\wh{T}})$ belongs to the image of $W_{\wh{G}}(A_{\wh{M}})$, then $w$ must be trivial as $w$ preserves $(\wh{T}, \wh{B}_{\wh{M}})$.
%Thus we get the second assertion.
\end{proof}

\subsection{Definition of the regular part on the endoscopic side}\label{ss: endo-reg-part}

We continue with the fixed data from \S\ref{ss: setup} and \S\ref{ss: motivation}.
Let $P$ be a standard parabolic subgroup of $G$ with standard Levi $M$ for a fixed tempered $L$-parameter $\phi$ as in \S \ref{ss: S-group}, i.e., ${}^{L}M$ is a smallest Levi subgroup of ${}^{L}G$ such that $\phi$ factors through ${}^{L}M\hookrightarrow{}^{L}G$.
Then $X^{\mf{e}}_L$ is in bijection with $W_L \backslash W(L,H) / W_H$ where we identify $W_H$ with a subgroup of $\wh{W}_G$ via $\eta$ and $W(L,H)$ consists of $w \in \wh{W}_G$ such that for each $\gamma \in \Gamma$, there exists $h_{\gamma} \in \wh{H}$ such that $\Int(h_{\gamma}) \circ \gamma$ centralizes $(w \circ \eta)^{-1}(A_{\wh{M}})$ (see \cite[\S2.7]{BM2}).

\begin{lemma}{\label{lem: M-HM}}
Suppose $(H, \mc{H}, s, \eta)$ is a refined endoscopic datum through which $\phi$ factors as $\phi_H$ and let $H_M \subset H$ be a minimal Levi through which $\phi_H$ factors. Then $\eta^{-1}(\wh{M})$ and $\wh{H_M}$ are conjugate in $N_{\wh{H}}(\wh{T}_{H})$.  
\end{lemma}
\begin{proof}
    We have that $A_{\wh{M}}$ is a maximal torus of $S^{\circ}_{\phi}$ and note that $A_{\wh{M}} \subset \wh{T} \subset \eta(\wh{H})$. Since $\eta(S^{\circ}_{\phi_H}) \subset S^{\circ}_{\phi}$, we have that $\eta^{-1}(A_{\wh{M}})$ is a maximal torus of $S^{\circ}_{\phi_H}$. But if $H_M$ is a minimal Levi through which $\phi_H$ factors, then $A_{\wh{H_M}}$ is a maximal torus of $S^{\circ}_{\phi_H}$ and hence there exists $h \in S^{\circ}_{\phi_H} \subset \wh{H}$ conjugating $A_{\wh{H_M}}$ to  $\eta^{-1}(A_{\wh{M}})$. Let $\wh{T}' = \Int(h)(\wh{T}_{H})$. Then $\wh{T}'$ and $\wh{T}_{H}$ are two maximal tori in $Z_{\wh{H}}(\eta^{-1}(A_{\wh{M}}))$ and hence are conjugate. Thus, we may as well assume $h \in N_{\wh{H}}(\wh{T}_{H})$.
\end{proof}

We assume $\phi_H$ and $\phi$ are chosen such that $H_M$ and $M$ can be chosen to be standard Levi subgroups. Each $\mf{e}_L \in X^{\mf{e}}_L$ determines a Borel subgroup $\wh{B}^{\mf{e}_L} \subset \wh{H}$ via $\wh{B}^{\mf{e}_L} = (\Int(h) \circ \eta)^{-1}(\wh{B})$. There is a unique standard parabolic subgroup for $\wh{H_M}$ containing $\wh{B}^{\mf{e}_L}$, which we call $P^{\mf{e}_L}$. Similarly, there is a standard parabolic for $\wh{H_L}$ containing $\wh{B}^{\mf{e}_L}$, which is exactly $P^{\op}_{\nu_{\mf{e}_L}}$. 

By Assumption \ref{assumption:LIR-basic} and the geometric lemma of \cite{Ber1}, we have that the term $J^H_{P^{\op}_{\nu_{\mf{e}_L}}} S\Theta_{\phi_H}^H$ which appears in the expression in \eqref{eqn: endoscopiccharacter} becomes
\begin{align}
     J^H_{P^{\op}_{\nu_{\mf{e}_L}}} S\Theta_{\phi_H}^H
     &=
     J^H_{P^{\op}_{\nu_{\mf{e}_L}}}I^H_{P^{\mf{e}_L}}S\Theta_{\phi_{H}}^{H_M}\\
     &=
     \sum\limits_{w \in W^{\rel,H_M, H_L}} I^{H_L}_{P_2} \circ w^\ast \circ J^{H_M}_{P_1} S\Theta_{\phi_{H}}^{H_M}, \nonumber
\end{align}
where $P_1$ (resp.\ $P_2$) is the standard parabolic subgroup of $H_M \cap w^{-1}(H_L)$ inside $H_M$ (resp.\ $w(H_M) \cap H_L$ inside $H_L$) and $w^\ast$ denotes the pull-back via the $w$-conjugation from $H_M \cap w^{-1}(H_L)$ to $w(H_M) \cap H_L$.

Note that when $w \in W^\rel[H_M, H_L]$, we have $I^{H_L}_{P_2} \circ w^\ast \circ J^{H_M}_{P_1} S\Theta_{\phi_{H}}^{H_M}=S\Theta_{{}^w\phi_{H}}^{H_L}$.
Indeed, $H_M \cap w^{-1}(H_L) = H_M$ and so the $J^{H_M}_{P_1}$ is just the identity map. 
Moreover, by Assumption \ref{assumption:twist}, we have $w^\ast S\Theta_{\phi_{H}}^{H_M}=S\Theta_{{}^w\phi_{H}}^{H_M}$.
Finally, by Assumption \ref{assumption:LIR-basic}, we get $I_{P_2}^{H_L}S\Theta_{{}^w\phi_{H}}^{H_M}=S\Theta_{{}^w\phi_{H}}^{H_L}$

This motivates the following definition.

\begin{definition}{\label{def: endregpart}}
    We define the \textit{regular part} of $J^H_{P^{\op}_{\nu_{\mf{e}_L}}} S\Theta_{\phi_H}^H$ to be
    \begin{equation*}
             [J^H_{P^{\op}_{\nu_{\mf{e}_L}}} S\Theta_{\phi_H}^H]_{\reg}
             :=
             \sum\limits_{w \in W^\rel[H_M, H_L]} S\Theta_{\co{w}{\phi_H}}^{H_L}.
    \end{equation*}
    We define the \textit{regular part} of $\Trans^{G_b}_H S\Theta^H_{\phi_H}$ by replacing $J^H_{P^{\op}_{\nu_{\mf{e}_L}}} S\Theta_{\phi_H}^H$ in the expression \eqref{eqn: endoscopiccharacter} with $[J^H_{P^{\op}_{\nu_{\mf{e}_L}}} S\Theta_{\phi_H}^H]_{\reg}$:
    \begin{align*}
         [\Trans^{G_b}_H S\Theta^H_{\phi_H}]_\reg
         &:=
         \sum\limits_{\mf{e}_L \in X^{\mf{e}}_L} (\Trans^{G_b}_{H_L} [J^H_{P^{\op}_{\nu_{\mf{e}_L}}}S\Theta^H_{\phi_H}]_\reg)\otimes\ov{\delta}^{1/2}_{P_{\nu_b}}\\
         &=
         \sum\limits_{\mf{e}_L \in X^{\mf{e}}_L} \Bigl(\Trans^{G_b}_{H_L} \sum\limits_{w \in W^\rel[H_M, H_L]} S\Theta_{\co{w}{\phi_H}}^{H_L}\Bigr)\otimes\ov{\delta}^{1/2}_{P_{\nu_b}}.
    \end{align*}
\end{definition}

\subsection{Parametrization of members of \texorpdfstring{$\Pi_{\phi}(G_{b})$}{Pi(phi)(Gb)}}

In \S \ref{s: theconstruction}, we constructed a bijective map $\iota_{\mf{w}}$ between $\coprod_{b\in B(G)}\Pi_\phi(G_b)$ and $\Irr(S_\phi)$.
For convenience, for any $\rho\in\Irr(S_\phi)$, we write $\pi_\rho:=\iota_{\mf{w}}^{-1}(\rho)$.
Our aim in here is to, for each $b\in B(G)$, describe and parametrize $\rho\in\Irr(S_\phi)$ satisfying $\pi_\rho\in\Pi_\phi(G_b)$.

In the following, we fix a standard parabolic subgroup $Q$ of $G$ with Levi part $L$ and fix $b_L\in B(L)^+_\bas$ such that $\alpha_L(\lambda_L)\in\mf{A}_Q^+$, where $\lambda_L:=\kappa_L(b_L)|_{A_{\wh{L}}}$.
We put $b\in B(G)$ to be the image of $b_L$ in $B(G)$.

\begin{lemma}\label{lem: highest-weight-weyl}
Let $\rho=\mc{L}(\lambda,E)\in\Irr(S_\phi)$ be the irreducible representation of $S_\phi$ with highest weight $\lambda\in X^\ast(A_{\wh{M}})^+$ and a simple $\mc{A}^\lambda$-module $E$.
If $\pi_\rho$ belongs to $\Pi_\phi(G_b)$, then there exists an element $w\in \wh{W}^{\rel}(M,L)$ satisfying $\alpha_{{}^{w}M}({}^{w}\lambda)=\alpha_L(\lambda_L)$, or equivalently, $\lambda=\alpha_{M}^{-1}\circ w^{-1}\circ\alpha_L(\lambda_L)$.
\end{lemma}

\begin{proof}
Let us recall our construction of $\pi_\rho$.
We first choose an element $w\in W^\rel$ satisfying ${}^{w}\alpha_M(\lambda)\in \mf{A}_{Q_\lambda}^+$ for a (unique) standard parabolic subgroup $Q_\lambda$.
Let $L_\lambda$ be the Levi part of $Q_\lambda$ (thus we have $\mf{A}_{Q_\lambda}^+\subset X_\ast(A_{L_\lambda})_\R$). 
We have ${}^{w}\alpha_M(\lambda)=\alpha_{{}^{w}M}(\co{w}{\lambda})$.
Note that $\co{w}{M} \subset L_\lambda$ since we have $\alpha_{{}^{w}M}({}^{w}\lambda)\in\mf{A}_Q^+$, hence $w$ belongs to $W^{\rel}(M,L_\lambda)$.
We apply the $B(L_\lambda)_\bas$-LLC to $({}^{w}\phi,\rho_{L_\lambda})$ to obtain $b_{L_\lambda}\in B(L_\lambda)^+_\bas$ and $\pi_{b_{L_\lambda}}\in\Pi_{\co{w}{\phi}}(L_{b_{L_\lambda}})$, where $\rho_{L_\lambda}:=\mc{L}_{L_\lambda}({}^{w}\lambda,{}^{w}E_{L_\lambda})$ (see \S\ref{ss: themap}). 
Then $\pi_\rho$ is defined to be $\pi_{b_{L_\lambda}}$.
Hence, the assumption that $\pi_\rho\in\Pi_\phi(G_b)$ is equivalent to that $b\in B(G)$ is the image of $b_{L_\lambda}\in B(L_\lambda)_\bas^+$.
By our definition of $b\in B(G)$, this is furthermore equivalent to that $L=L_\lambda$ and $b_L=b_{L_\lambda}$.

By the commutative diagram \eqref{eqn: basic correspondence}, $\kappa_L(b_L)|_{A_{\wh{L}}}$ is given by the $A_{\wh{L}}$-central character of $\rho_L$, which equals ${}^{w}\lambda|_{A_{\wh{L}}}$.
On the other hand, by definition, $\lambda_L=\kappa_L(b_L)|_{A_{\wh{L}}}$.
Hence we get $\lambda_L={}^{w}\lambda|_{A_{\wh{L}}}$.
Now we note the following commutative  diagram:
\begin{equation*}
    \begin{tikzcd}
    & X^{\ast}(Z({}^{w}\wh{M})^\Gamma) \ar[r,->>, "\res"] \arrow[d,->>, "\res"] & X^{\ast}(A_{{}^{w}\wh{M}}) \arrow[r,hook] \arrow[d,->>, "\res"] & X^{\ast}(A_{{}^{w}\wh{M}})_{\R} \arrow[d,->>, "\res"] \arrow[r,"\alpha_{{}^{w}M}"] & \mf{A}_{{}^{w}M} \\
    B(L)_\bas \arrow[r, "\kappa_L"] & X^{\ast}(Z(\wh{L})^\Gamma) \ar[r,->>, "\res"] & X^{\ast}(A_{\wh{L}}) \arrow[r,hook] &X^{\ast}(A_{\wh{L}})_{\R} \arrow[r,"\alpha_L"] & \mf{A}_L \arrow[u, hook]
    \end{tikzcd}
\end{equation*}
Since $\alpha_{{}^{w}M}({}^{w}\lambda)$ belongs to $\mf{A}_L\subset \mf{A}_{{}^{w}M}$, we have ${}^{w}\lambda|_{A_{\wh{L}}}=\alpha_L^{-1}\circ\alpha_{{}^{w}M}({}^{w}\lambda)$.
Hence we obtain $\lambda_L=\alpha_L^{-1}\circ\alpha_{{}^{w}M}({}^{w}\lambda)$.
\end{proof}

In the following, for $w\in \wh{W}^{\rel}(M,L)$, we shortly write $\lambda_{L}^{w}$ for $\alpha_{M}^{-1}\circ w^{-1}\circ\alpha_L(\lambda_L)\in X^{\ast}(A_{\wh{M}})$.
(Hence what we have proved in Lemma \ref{lem: highest-weight-weyl} is that the highest weight of any $\rho\in\Irr(S_\phi)$ satisfying $\pi_\rho\in\Pi_\phi(G_b)$ must be of the form $\lambda_{L}^{w}$ for some $w\in \wh{W}^{\rel}(M,L)$.)
We also put 
\begin{align*}
\lambda_{L,w}
:=w\circ\alpha_{M}^{-1}\circ w^{-1}\circ\alpha_L(\lambda_L)
=\alpha_{{}^{w}M}^{-1}\circ\alpha_L(\lambda_L)\in X^{\ast}(A_{{}^{w}\wh{M}}).
\end{align*}
Note that ${}^{w}M$, ${}^{w}\phi$, and $\lambda_{L,w}$ depend only on the right $W_\phi$-coset of $w\in\wh{W}^{\rel}(M,L)$.
(Recall that we have a map $W_\phi\rightarrow W_{\wh{G}}(A_{\wh{M}})\hookrightarrow\wh{W}^\rel$ by \eqref{map:Wphi} and Lemma \ref{lem: weylinj}, hence $\wh{W}^{\rel}(M,L)$ is stable under right $W_\phi$-translation by Lemma \ref{lem: Weylgroups}.)

We let $\mc{I}(\phi,b)$ denote the set of pairs $(w,E_{L,w})$, where
\begin{itemize}
\item
$w\in W^{\rel}(M,L)/W_\phi$, and 
\item
$E_{L,w}$ is a simple $\mc{A}_{L}^{\lambda_{L,w}}$-module such that the $Z(\wh{L})^\Gamma$-central character of the irreducible representation $\mc{L}_L(\lambda_{L,w},E_{L,w})$ of $S_{{}^{w}\phi,L}$ is given by $\kappa_L(b_L)$.
\end{itemize}
Here, note that we need to specify the dominance in $S_{{}^{w}\phi,L}^{\circ}$ so that the notation $\mc{L}_L(-,-)$ makes sense in general.
However, since $\lambda_{L,w}$ extends to a $1$-dimensional character of $S_{{}^{w}\phi,L}^{\circ}$ as shown in the proof of Lemma \ref{lem: naturalfactoring}, the representation $\mc{L}_L(\lambda_{L,w},E_{L,w})$ is determined independently of the choice of the dominance.
We define an equivalence relation on $\mc{I}(\phi,b)$ as follows: $(w_{1},E_{L,w_{1}})\sim(w_{2},E_{L,w_{2}})$ if and only if there exists an element $w_{L}\in W^{\rel}_{L}$ such that
\begin{itemize}
\item
$w_{2}=w_{L}w_{1}$ (i.e., $w_{1}$ and $w_{2}$ belong to the same double coset in $W^{\rel}_{L}\backslash W^{\rel}(M,L)/W_\phi$) and
\item
$E_{L,w_{1}}$ and $E_{L,w_{2}}$ are identified under the isomorphism $S_{{}^{w_{1}}\phi,L}^{\lambda_{L,w_{1}}}\cong S_{{}^{w_{2}}\phi,L}^{\lambda_{L,w_{2}}}$ given by $\Int(w_{L})$.
\end{itemize}

\begin{proposition}\label{prop: members-of-b}
We have a natural bijection between the sets $\{\rho\in\Irr(S_{\phi}) \mid \pi_{\rho}\in\Pi_{\phi}(G_{b})\}$ and $\mc{I}(\phi,b)/{\sim}$.
\end{proposition}

\begin{proof}
By Theorem \ref{thm: acharmainthm} and Lemma \ref{lem: disconn-irrep-paramet}, the set $\Irr(S_{\phi})$ is bijective to the set of pairs $(\lambda,E)$, where
\begin{itemize}
\item 
$\lambda$ is (a representative of) an element of $X^\ast(A_{\wh{M}})^+/R_\phi$,
\item
$E$ is (the isomorphism class of) a simple $\mc{A}^{\lambda}$-module,
\end{itemize}
by $\mc{L}(\lambda,E)\leftrightarrow (\lambda,E)$.
Thus, by Lemma \ref{lem: highest-weight-weyl}, the set of elements $\rho\in\Irr(S_{\phi})$ satisfying $\pi_{\rho}\in\Pi_{\phi}(G_{b})$ is in bijection with the set of pairs $(\lambda,E)$, where
\begin{itemize}
\item 
$\lambda$ runs over a complete set of representatives of
\[
\{\lambda_{L}^{w}\in X^\ast(A_{\wh{M}})^+ \mid w\in W^{\rel}(M,L)\}/R_\phi,
\]
\item
$E$ runs over the isomorphism classes of simple $\mc{A}^\lambda$-modules such that the $Z(\wh{L})^\Gamma$-central character of $\mc{L}_L(\lambda_{L,w},E_{L,w})$ is given by $\kappa_L(b_L)$, where $E_{L,w}$ is the simple $\mc{A}_{L}^{\lambda_{L,w}}$-module which is identified with $E$ under the isomorphism $\mc{A}^{\lambda}\cong\mc{A}^{{}^{w}\lambda}\cong\mc{A}_{L}^{{}^{w}\lambda}$.
\end{itemize}
Since we have $W_\phi=W_\phi^\circ\rtimes R_\phi$ (Lemma \ref{lem: pi0embedweyl}) and each $W_\phi^\circ$-orbit in $X^{\ast}(A_{\wh{M}})$ contains a unique dominant element, we have
\[
X^{\ast}(A_{\wh{M}})/W_\phi\cong X^{\ast}(A_{\wh{M}})^{+}/R_\phi.
\]
Thus, by noting that the stabilizer of $\lambda_{L}$ in $W^{\rel}$ is given by $W^{\rel}_{L}$, we see that the map $W^{\rel}(M,L)\rightarrow X^{\ast}(A_{\wh{M}})\colon w\mapsto \lambda_{L}^{w}$ induces a bijection 
\[
W^{\rel}_{L}\backslash W^{\rel}(M,L)/W_\phi
\xrightarrow{1:1}
\{\lambda_{L}^{w}\in X^\ast(A_{\wh{M}})^+ \mid w\in W^{\rel}(M,L)\}/R_\phi.
\]
Therefore the set of pairs $(\lambda,E)$ as above can be identified with $\mc{I}(\phi,b)/{\sim}$.
\end{proof}

\subsection{Definition of \texorpdfstring{$\langle\pi,\eta(s)\rangle_\reg$}{<pi,s>reg}}{\label{ss: regular-pairing-def}}

\begin{lemma}\label{lem: endoscopic-s-Levi}
Let $\phi$ be an $L$-parameter of $G$.
Suppose that $(H, \mc{H}, s,\eta)$ is a refined endoscopic datum which $\phi$ factors through as $\phi_{H}$ (i.e., $\phi=\eta\circ\phi_{H}$).
Then, for any standard Levi subgroup $L$ of $G$ such that $\phi$ factors through ${}^{L}L\hookrightarrow {}^{L}G$, we have $\eta(s)$ belongs to $S_{\phi,L}$.
\end{lemma}

\begin{proof}
We first note that $\eta(s)$ belongs to $S_{\phi}$.
Indeed, by definition, $\eta(s)\in S_{\phi}$ if and only if $\eta(s)\cdot\phi(\sigma)\cdot\eta(s)^{-1}=\phi(\sigma)$ for any $\sigma\in L_{F}$.
As we have $\phi=\eta\circ\phi_{H}$ and $\eta$ is an $L$-embedding ${}^{L}H\hookrightarrow{}^{L}G$, this is equivalent to that $s\cdot\phi_{H}(\sigma)\cdot s^{-1}=\phi_{H}(\sigma)$ for any $\sigma\in L_{F}$, which is true since $s\in Z(\wh{H})^{\Gamma}$.

Thus our task is to show that $\eta(s)$ belongs to $\wh{L}$.
Let $M$ be a minimal Levi subgroup of $G$ such that $M\subset L$ and $\phi$ factors through $M$.
It is enough to show that $\eta(s)$ belongs to $\wh{M}$.
Let $H_{M}$ be a minimal Levi subgroup of $H$ which $\phi_{H}$ factors through.
As $\wh{H_{M}}\subset\wh{H}$, we have $Z(\wh{H_{M}})^{\Gamma}\supset Z(\wh{H})^{\Gamma}$.
Since $s\in Z(\wh{H})^{\Gamma}$ and $Z(\wh{H_{M}})^{\Gamma}\subset \wh{H_{M}}$, we get $s\in \wh{H_{M}}$.
By Lemma \ref{lem: M-HM}, there exists an element $h\in \wh{H}$ satisfying $h\eta^{-1}(\wh{M})h^{-1}=\wh{H_{M}}$.
Hence we get $\eta(h^{-1}sh)\in\wh{M}$.
Again noting that $s\in Z(\wh{H})^{\Gamma}$, we get $\eta(h^{-1}sh)=\eta(s)$, which completes the proof.
\end{proof}

Now suppose that $(H, \mc{H}, s,\eta)$ is a refined endoscopic datum for $G$ which $\phi$ factors through as $\phi_{H}$ (i.e., $\phi=\eta\circ\phi_{H}$).
Let $\rho=\mc{L}(\lambda,E)\in\Irr(S_\phi)$ be an element satisfying $\pi_{\rho}\in\Pi_{\phi}(G_{b})$.
We define a quantity $\langle\pi_{\rho},\eta(s)\rangle_\reg\in\C$ in the following manner.

Let $(w,E_{L,w})\in\mc{I}(\phi,b)/{\sim}$ be an element corresponding to $\rho$ as in Proposition \ref{prop: members-of-b}.
We take a representative of $(w,E_{L,w}) \in \mc{I}(\phi,b)/{\sim}$ in $\mc{I}(\phi,b)$ and furthermore a representative of $w\in W^{\rel}(M,L)/W_\phi$ in $W^{\rel}(M,L)$. We use the same notations ($(w, E_{L,w})$ and $w$) to refer to these representatives.
We put $\rho_{L}:=\mc{L}_{L}(\lambda_{L,w},E_{L,w})$, which is an irreducible representation of $S_{{}^{w}\phi,L}$.
For any element $w'\in W_{\co{w}\phi}$, we have ${}^{w'w}\eta(s)\in S_{{}^{w}\phi,L}$ by applying Lemma \ref{lem: endoscopic-s-Levi} to the refined endoscopic datum $(H,s,\Int(w'w)\circ\eta)$ and the $L$-parameter $\co{w'w}{\phi} \,(=\co{w}{\phi})$.
Here, we implicitly fix a representative of $w\in W^\rel\cong\wh{W}^\rel$ in $N_{\wh{G}}(A_{\wh{T}})$ (resp.\ $w'\in W_{{}^w\phi}$ in $N_{S_{{}^w\phi}}(A_{{}^w\wh{M}})$) and again write $w$ (resp.\ $w'$) for it by abuse of notation.
We put
\[
\langle\pi_{\rho},\eta(s)\rangle_\reg
:=
\sum_{w'\in W_{{}^w\phi,L}\backslash W_{{}^w\phi}} \tr({}^{w'w}\eta(s) \mid \rho_{L}). 
\]
Here, note that the trace of $\rho_L$ is invariant under the $S_{\co{w}\phi,L}$-conjugation, hence the quotienting by $W_{\co{w}\phi,L}=W_{\co{w}\phi}\cap W^\rel_L$ in the index set makes sense.

%\[
%\langle\pi_{\rho},\eta(s)\rangle_\reg
%:=
%\frac{|W^{\rel}_{L}\backslash W^{\rel}_{L}wW_\phi|}{|W_\phi|}\cdot
%\sum_{w'\in W_\phi} \tr({}^{ww'}\eta(s) \mid \rho_{L}). 
%\]
\begin{remark}
When $L=G$, the index set of the above sum is trivial and also $\rho_L=\rho$, hence we simply have $\langle\pi_{\rho},\eta(s)\rangle_\reg=\langle\pi_{\rho},\eta(s)\rangle$.
\end{remark}
%\abm{ I wonder if there's some clifford theory explanation of this. The right-hand side  a bit like the character formula for $\chi_N$ in the wikipedia article for clifford theory}

\begin{lemma}
The quantity $\langle\pi_{\rho},\eta(s)\rangle_\reg$ is well-defined, i.e., independent of the choices of representatives of $(w,E_{L,w})\in \mc{I}(\phi,b)/{\sim}$ in $\mc{I}(\phi,b)$ and $w\in W^{\rel}(M,L)/W_\phi$ in $W^{\rel}(M,L)$.
\end{lemma}

\begin{proof}
By noting that $W_{\co{w}\phi}w=wW_\phi$ and that $|W_{\co{w}{\phi},L}|=|W_{\phi,L}|$, we have
\[
\langle\pi_{\rho},\eta(s)\rangle_\reg
=
|W_{\phi,L}|^{-1}
\sum_{w''\in wW_{\phi}} \tr({}^{w''}\eta(s) \mid \rho_{L}). 
\]
Thus the independence of the choice of a representative of $w\in W^{\rel}(M,L)/W_\phi$ in $W^{\rel}(M,L)$ is clear from this expression.
If $(w_{1},E_{L_{1},w})\in\mc{I}(\phi,b)$ and $(w_{2},E_{L_{2},w})\in\mc{I}(\phi,b)$ represent $(w,E_{L,w})\in\mc{I}(\phi,b)/{\sim}$, then there exists an element $w_{L}\in W^{\rel}_{L}$ such that $w_{2}=w_{L}w_{1}$ and $E_{L,w_{1}}$ and $E_{L,w_{2}}$ are identified under the isomorphism $S_{{}^{w_{1}}\phi,L}^{\lambda_{L,w_{1}}}\cong S_{{}^{w_{2}}\phi,L}^{\lambda_{L,w_{2}}}$ given by $\Int(w_{L})$.
In particular, the representations $\mc{L}_{L}(\lambda_{L,w_{1}},E_{L,w_{1}})$ of $S_{{}^{w_{1}}\phi,L}$ and $\mc{L}_{L}(\lambda_{L,w_{2}},E_{L,w_{2}})$ of $S_{{}^{w_{2}}\phi,L}$ are identified under the isomorphism $\Int(w_{L})\colon S_{{}^{w_{1}}\phi,L}\cong S_{{}^{w_{2}}\phi,L}$.
Moreover, $\Int(w_{L})$ maps the set $\{{}^{w''}\eta(s)\mid w''\in w_1 W_\phi\}$ to $\{{}^{w''}\eta(s)\mid w''\in w_2 W_\phi\}$ bijectively.
Thus we get
\begin{align*}
\sum_{w''\in w_1 W_\phi} \tr ({}^{w''}\eta(s) \mid \mc{L}_{L}(\lambda_{L,w_{1}},E_{L,w_{1}}))
=
\sum_{w''\in w_2 W_\phi} \tr ({}^{w''}\eta(s) \mid \mc{L}_{L}(\lambda_{L,w_{2}},E_{L,w_{2}})).
\end{align*}
This completes the proof.
\end{proof}

\begin{proposition}\label{prop:Gb-side-unstable-sum}
We have
\[
e(G_b)\sum_{\pi\in\Pi_{\phi}(G_{b})}\langle\pi,\eta(s)\rangle_\reg \Theta_\pi
=
\sum_{w\in W^{\rel}_{L}\backslash W^{\rel}(M,L)}\Theta^{L_{b_{L}},{}^{w}\eta(s)}_{{}^{w}\phi}.
\]
\end{proposition}

\begin{proof}

By our construction of $\Pi_{\phi}(G_b)$, we have
\[
\sum_{\pi\in\Pi_{\phi}(G_{b})}\langle\pi,\eta(s)\rangle_\reg \Theta_\pi
=
\sum_{\begin{subarray}{c}\rho\in\Irr(S_{\phi}) \\ \pi_\rho\in\Pi_\phi(G_b) \end{subarray}}\langle\pi_{\rho},\eta(s)\rangle_\reg \Theta_{\pi_{\rho_{L}}},
\]
where the sum on the right-hand side is over $\rho\in\Irr(S_{\phi})$ associated to $b\in B(G)$ and $\pi_{\rho_{L}} \in \Pi(L)$ corresponds to $\rho_{L}$ under the $B(L)_\bas$-LLC (see \S \ref{ss: themap}).
By Proposition \ref{prop: members-of-b} and the definition of $\langle\pi_\rho,\eta(s)\rangle_\reg$, we have
\begin{align*}
\sum_{\begin{subarray}{c}\rho\in\Irr(S_{\phi}) \\ \pi_\rho\in\Pi_\phi(G_b)\end{subarray}}\langle\pi_{\rho},\eta(s)\rangle_\reg \Theta_{\pi_{\rho_{L}}}
=
\sum_{\begin{subarray}{c}(w,E_{L,w})\\ \in\mc{I}(\phi,b)/{\sim}\end{subarray}} 
\sum_{w'\in W_{{}^w\phi,L}\backslash W_{{}^w\phi}} \tr({}^{w'w}\eta(s) \mid \rho_{L}) \Theta_{\pi_{\rho_{L}}}.
\end{align*}
Note that the order of the equivalence class of $(w,E_{L,w})\in\mc{I}(\phi,b)$ is given by $|W^{\rel}_{L}wW_\phi/W_\phi|$.
Hence, the right-hand side equals
\[
\sum_{(w,E_{L,w})\in\mc{I}(\phi,b)}
|W_{{}^w\phi,L}|^{-1}\cdot |W^{\rel}_{L}wW_\phi/W_\phi|^{-1}
\sum_{w'\in W_{\co{w}\phi}} \tr ({}^{w'w}\eta(s) \mid \rho_{L}) \Theta_{\pi_{\rho_{L}}}.
\]
By noting that the association $w_L\mapsto w_LwW_\phi$ induces a bijection $W^{\rel}_{L}/W_{{}^w\phi,L}\xrightarrow{1:1} W^{\rel}_{L}wW_\phi/W_\phi$, this equals
\[
\sum_{(w,E_{L,w})\in\mc{I}(\phi,b)}
|W^{\rel}_{L}|^{-1}\sum_{w'\in W_\phi} \tr ({}^{ww'}\eta(s) \mid \rho_{L}) \Theta_{\pi_{\rho_{L}}}.
\]
By the definitions of $\mc{I}(\phi,b)$ and $\rho_{L}$, this equals
\begin{align}\label{eq:ECI-sum-over-W-E}
\sum_{w\in W^{\rel}(M,L)}
|W^{\rel}_{L}|^{-1}
\sum_{E_{L,w}} \tr({}^{w}\eta(s) \mid \mc{L}_{L}(\lambda_{L,w},E_{L,w})) \Theta_{\pi_{\rho_{L}}},
\end{align}
where $E_{L,w}$ runs over (the isomorphism classes of) simple $\mc{A}_{{}^{w}\phi,L}^{\lambda_{L,w}}$-modules such that the $Z(\wh{L})^\Gamma$-central character of $\mc{L}_L(\lambda_{L,w},E_{L,w})$ is given by $\kappa_L(b_L)$.
By Lemma \ref{lem: E-vs-E_L} (see below), \eqref{eq:ECI-sum-over-W-E} is equal to
\begin{align}\label{eq:ECI-sum-over-W-L}
\sum_{w\in W^{\rel}_{L}\backslash W^{\rel}(M,L)}
\sum_{\begin{subarray}{c} \rho_L\in\Irr(S_{{}^{w}\phi,L}^\natural)\\ \rho_L|_{Z(\wh{L})^\Gamma}=\kappa_L(b_L) \end{subarray}}
\tr({}^{w}\eta(s) \mid \rho_{L})\cdot\Theta_{\pi_{\rho_L}},
\end{align}
where the second sum is over irreducible representations $\rho_L$ of $S_{{}^{w}\phi,L}^\natural$ with $Z(\wh{L})^\Gamma$-central character $\kappa_L(b_L)$.
Since the product of $e(G_b)$ and the inner sum is nothing but $\Theta^{L_{b_{L}},{}^{w}\eta(s)}_{{}^{w}\phi}$, we get the desired equality.
\end{proof}

\begin{lemma}\label{lem: E-vs-E_L}
Let $w\in W^{\rel}(M,L)$.
The association $E_{L,w}\mapsto \mc{L}_{L}(\lambda_{L,w},E_{L,w})$ gives a bijection between
\begin{itemize}
\item 
the set of isomorphism classes of simple $\mc{A}_{{}^{w}\phi,L}^{\lambda_{L,w}}$-modules such that the $Z(\wh{L})^\Gamma$-central character of $\mc{L}_L(\lambda_{L,w},E_{L,w})$ is given by $\kappa_L(b_L)$, and
\item
the set of irreducible representations $\rho_L$ of $S_{{}^{w}\phi,L}^\natural=S_{{}^{w}\phi,L}/(\wh{L}_\der \cap S_{{}^{w}\phi,L})^\circ$ with $Z(\wh{L})^\Gamma$-central character $\kappa_L(b_L)$.
\end{itemize}
\end{lemma}

\begin{proof}
The well-definedness of the map is already discussed in Lemma \ref{lem: naturalfactoring}.
Here, we remark that $(\wh{L}_\der \cap S_{{}^{w}\phi,L})^\circ$ acts trivially on $\mc{L}_L(\lambda_{L,w},E_{L,w})$ and  $Z(\wh{L})^\Gamma$ acts on $\mc{L}_L(\lambda_{L,w},E_{L,w})$ as the character $\kappa_L(b_L)$, their product $(\wh{L}_\der \cap S_{{}^{w}\phi,L})^\circ\cdot Z(\wh{L})^\Gamma$ acts via a character.
For our convenience, we let $\tilde{\lambda}_{L,w}$ denote for this character, which does not depend on $E_{L,w}$.
Note that we have 
\[
(\wh{L}_\der \cap S_{{}^{w}\phi,L})^\circ\cdot Z(\wh{L})^\Gamma
=
S_{{}^{w}\phi,L}^\circ\cdot Z(\wh{L})^\Gamma
\]
by \cite[Lemma 0.4.13]{KMSW}.
In particular, the group $(\wh{L}_\der \cap S_{{}^{w}\phi,L})^\circ\cdot Z(\wh{L})^\Gamma$ contains $S_{{}^{w}\phi,{}^{w}M}^\circ=A_{{}^{w}\wh{M}}$.
The restriction of $\tilde{\lambda}_{L,w}$ to $A_{{}^{w}\wh{M}}$ equals $\lambda_{L,w}$.

The injectivity of the map is a part of the classification theorem of irreducible representations of a disconnected reductive group (Theorem \ref{thm: acharmainthm} together with Lemma \ref{lem: disconn-irrep-paramet}, applied to $S_{{}^{w}\phi,L}$).

To show the surjectivity, let us take an irreducible representation $\rho_L$ of $S_{{}^{w}\phi,L}^\natural$ with $Z(\wh{L})^\Gamma$-central character $\kappa_L(b_L)$.
It is enough to show that if we regard $\rho_L$ as an irreducible representation of $S_{{}^{w}\phi,L}$ by inflation, the highest weight of $\rho_L$ is given by $\lambda_{L,w}\in X^\ast(A_{{}^{w}\wh{M}})$.
By the discussion in the first paragraph, it suffices to check that the group $(\wh{L}_\der \cap S_{{}^{w}\phi,L})^\circ\cdot Z(\wh{L})^\Gamma$ acts on $\rho_L$ by the character $\tilde{\lambda}_{L,w}$.
This is obvious since $\rho_L$ is constructed by inflation from $S_{{}^{w}\phi,L}^\natural$ and the $Z(\wh{L})^\Gamma$-central character of $\rho_{L}$ is $\kappa_L(b_L)$.
\end{proof}

\subsection{Proof of main theorem}\label{ss:proof of ECI}

\begin{lemma}{\label{lem: endoindexinglem}}
    We have a natural bijection  
    \begin{equation*}
        \coprod\limits_{X^{\mf{e}}_L} W^{\rel}_{H_L} \backslash W^{\rel}(H_M, H_L) = W^{\rel}_L \setminus W^{\rel}(M,L).
    \end{equation*}
\end{lemma}
\begin{proof}
    Fix $h \in N_{\wh{H}}(\wh{T}_{H})$ conjugating $\eta^{-1}(A_{\wh{M}})$ to $A_{\wh{H_M}}$ as in Lemma \ref{lem: M-HM}. An element of $X^{\mf{e}}_L$ yields some $\dot{w}^{-1} \in N_{\wh{G}}(\wh{T})$ that takes $A_{\wh{L}}$ into $\eta(A_{\wh{H_L}})$ and an element of $W^{\rel}(H_M, H_L)$, whose inverse mapped into $\wh{W}^{\rel}$ via $\eta$ takes $\eta(A_{\wh{H_L}})$ to $\eta(A_{\wh{H_M}})$, which we identify with $A_{\wh{M}}$ via $\eta(h)^{-1}$. So in all we get a map $A_{\wh{L}} \to A_{\wh{M}}$. If we act on the element of $W^{\rel}(H_M, H_L)$ on the left by an element of $W^{\rel}_{H_L}$, then the resulting map $A_{\wh{L}} \to A_{\wh{M}}$ does not change. In particular, the corresponding elements of $W^{\rel}(M,L)$ agree up to an element of $W^{\rel}_L$. This constructs a map in one direction.

    Conversely, suppose we are given $w \in W^{\rel}(M,L)$ that therefore satisfies $w^{-1}(A_{\wh{L}}) \subset A_{\wh{M}}$. We take a lift $\dot{w} \in N_{\wh{G}}(\wh{T})$ of $w$ and then $\eta(h)\dot{w}^{-1}$ maps $A_{\wh{L}}$ into $\eta(A_{H_M})$. Then $(\Int(\dot{w}\eta(h)^{-1}) \circ \eta)^{-1}(A_{\wh{L}}) \subset A_{\wh{H_M}} \subset \wh{T}_{H}^{\Gamma}$, so $\dot{w}\eta(h)^{-1}$ induces an element of $W(L,H)$. By the proof of \cite[Proposition 2.24]{BM2},  $\Int(\dot{w} \eta(h)^{-1}) \circ \eta$ restricts to give an embedded endoscopic datum $(H'_L, H, s, \Int(\dot{w} \eta(h)^{-1}) \circ \eta)$. This datum is conjugate by some $h' \in N_{\wh{H}}(\wh{T}_{H})$ to some $(H_L, H, s, \Int(\dot{w} \eta(h^{-1}h')) \circ \eta) \in X^{\mf{e}}_L$. In particular, $\Int(\dot{w} \eta(h^{-1}h'))(\eta(A_{\wh{H_L}})) \supset A_{\wh{L}}$. Now,
    \begin{align*}
 (\Int(\dot{w} \eta(h)^{-1}) \circ \eta)(\co{L}{H_M}) &= \Int(\dot{w})((\Int(\eta(h)^{-1}) \circ \eta)(\co{L}{H}) \cap \co{L}{M})\\
 &\subset (\Int(\dot{w} \eta(h)^{-1}) \circ  \eta)(\co{L}{H}) \cap \co{L}{L}\\
 &= (\Int(\dot{w}\eta(h)^{-1}) \circ \eta)(\co{L}{H'_L})\\
 &=(\Int(\dot{w}\eta(h^{-1}{h'}^{-1})) \circ \eta)(\co{L}{H_L}),
    \end{align*}
and hence $\Int(h')(A_{\wh{H_M}}) \supset A_{\wh{H_L}}$. So $h'$ gives an element of $W^{\rel}(H_M, H_L)$. So we have given an element of $X^{\mf{e}}_L$ and $W^{\rel}(H_M, H_L)$ and we see that by the construction going in the other direction, we recover $w$ since we are supposed to compose $\eta(h)^{-1} \circ  \eta(h') \circ (\dot{w}\eta(h^{-1}h'))^{-1}$ and this is supposed to yield the inverse of the element of $W^{\rel}(M,L)$.  If we act on the original $w \in W^{\rel}(M,L)$ on the left by an element of $W^{\rel}_L$, then the embedded datum  $(H'_L, H, s, \Int(\dot{w} \eta(h)) \circ \eta)$ will be in the same inner class, and hence the new $h'$ will differ from the old one by an element of $W^{\rel}_{H_L}$. This completes the proof.
\end{proof}

\begin{lemma}\label{lem: H-side-unstable-sums}
We have
\[
[\Trans^{G_b}_{H}S\Theta_{\phi_H}^H]_{\reg}
=
\sum_{w\in W^{\rel}_{L}\backslash W^{\rel}(M,L)}\Theta^{L_{b_{L}},{}^{w}\eta(s)}_{{}^{w}\phi}\otimes\ov{\delta}^{1/2}_{P_{\nu_b}}.
\]
\end{lemma}

\begin{proof}
We recall that by Definition \ref{def: endregpart}, the left-hand side is 
\[
\sum\limits_{\mf{e}_L \in X^{\mf{e}}_L} \Bigl(\sum\limits_{w \in W^\rel[H_M, H_L]}\Trans^{G_b}_{H_L} S\Theta_{\co{w}{\phi_H}}^{H_L}\Bigr)\otimes\ov{\delta}^{1/2}_{P_{\nu_b}}.
\]
 Applying the endoscopic character identities from the basic correspondence (Assumption \ref{assumption:ECI-basic}), we have that the left-hand side equals a sum over terms of the form $\Theta^{L_{b_L},\co{w'}{\eta(s)}}_{\co{w'}{
\phi}}$ for $w' \in \wh{W}_G$. Moreover, each element $w'$ that we get can be chosen to be exactly the element of $W^{\rel}(H_M, H_L)$ constructed in Lemma \ref{lem: endoindexinglem} (since $h$ as in that lemma can be chosen to centralize $\phi_H$).  Hence, to show that the two sides are equal, we just need to show the indexing sets are the same.  But this is Lemma \ref{lem: endoindexinglem}.
\end{proof}

Now let us prove Theorem \ref{thm: ECI}.

\begin{proof}[Proof of Theorem \ref{thm: ECI}]
By Lemma \ref{lem: H-side-unstable-sums}, we have
\[
[\Trans^{G_b}_{H}S\Theta_{\phi_H}^H]_{\reg}
=
\sum_{w\in W^{\rel}_{L}\backslash W^{\rel}(M,L)}\Theta^{L_{b_{L}},{}^{w}\eta(s)}_{{}^{w}\phi}\otimes\ov{\delta}^{1/2}_{P_{\nu_b}}.
\]
By Proposition \ref{prop:Gb-side-unstable-sum}, we have
\[
e(G_b)\sum_{\pi\in\Pi_{\phi}(G_{b})}\langle\pi,\eta(s)\rangle_\reg \Theta_\pi
=
\sum_{w\in W^{\rel}_{L}\backslash W^{\rel}(M,L)}\Theta^{L_{b_{L}},{}^{w}\eta(s)}_{{}^{w}\phi}.
\]
Thus we obtain the desired identity \eqref{eq:B(G)-ECI}:
\[
[\Trans^{G_b}_{H}S\Theta_{\phi_H}^H]_{\reg}
=
e(G_b)\sum_{\pi\in\Pi_{\phi}(G_{b})}\langle\pi,\eta(s)\rangle_\reg \Theta_{\pi}\otimes\ov{\delta}^{1/2}_{P_{\nu_b}}.
\]
\end{proof}

\begin{remark}
We finally comment on the non-tempered case.
The fundamental issue beyond the tempered case is that the endoscopic character identity for the basic LLC (Assumption \ref{assumption:ECI-basic}) no longer holds.
This is because non-tempered L-packets are constructed by the Langlands classification; in general, there is no nice description of the character of the Langlands quotient, which is a unique irreducible quotient of the standard module.
However, it is believed that the standard module is irreducible if and only if its Langlands quotient is generic (e.g., see \cite{HM07}).
The standard module itself is just a parabolically induced representation, so its character can be described in terms of the character of the inducing representation (e.g., \cite{vD72}).
Hence it is reasonable to expect that Assumption \ref{assumption:ECI-basic} and also our discussion so far can be extended to non-tempered but generic $L$-packets.
\end{remark}

\appendix
\section{Interpretation of the regular part in the $\GL_n$ case}{\label{s: appendixA}}

In our formulation of the endoscopic character relation, we introduced the regular part $[\Trans^{G_b}_H S\Theta^H_{\phi_H}]_\reg$ on the endoscopic side by replacing $J^H_{P^{\op}_{\nu_{\mf{e}_L}}} S\Theta_{\phi_H}^H$ in the expression \eqref{eqn: endoscopiccharacter} with $[J^H_{P^{\op}_{\nu_{\mf{e}_L}}} S\Theta_{\phi_H}^H]_{\reg}$, whose definition essentially relies on the geometric lemma of Bernstein--Zelevinsky (see Definition \ref{def: endregpart}).
It is natural to seek a more conceptual explanation of the regular part.
In this appendix, we explore this problem in the $\GL_n$ case.

\subsection{Preliminaries on the Zelevinsky classification}

In the following, we appeal to the theory of Zelevinsky classification \cite{Zel80}.
Here, we briefly summarize some key points of the theory, particularly those needed in our later proof.

\subsubsection{Classification of discrete series via segments}

We use the notation $\mf{m}=[\rho; x, y]$ for a \textit{segment} (in the sense of \cite{Zel80}) determined by the data of a unitary irreducible supercuspidal representation $\rho$ of $\GL_{r}(F)$ (for some $r\in\Z_{>0}$) and real numbers $x,y\in\R$ satisfying $y-x\in\Z_{\geq0}$.
More explicitly, $[\rho; x, y]$ is the set $\{\rho|\det|^{x}, \rho|\det|^{x+1}, \ldots, \rho|\det|^{y} \}$ of irreducible supercuspidal representations of $\GL_{r}(F)$.
We say that a segment $\mf{m}=[\rho; x, y]$ is \textit{centered} if $x+y=0$.
For any segment $\mf{m}=[\rho; x, y]$, we define $\pi(\mf{m})$ by the following:
\[
\pi(\mf{m})
:= \rho|\det|^{x}\times\rho|\det|^{x+1}\times\cdots\times\rho|\det|^{y}.
\]
Here, $(-)\times\cdots\times(-)$ is an abbreviated symbol for the normalized parabolic induction with respect to the standard (upper-triangular) parabolic subgroup; so, from $\GL_{r}\times\cdots\times\GL_{r}$ to $\GL_{r(y-x+1)}$ in this case.

\begin{theorem}[{\cite[Theorem 9.3]{Zel80}}]\label{thm:Zel-ds}
\begin{enumerate}
\item
For any segment $\mf{m}$, the representation $\pi(\mf{m})$ has a unique irreducible quotient $\Delta(\mf{m})$, which is discrete series.
\item
Conversely, any irreducible discrete series representation of $\GL_{n}(F)$ is of the form $\Delta(\mf{m})$ for a unique segment $\mf{m}$.
\item
An irreducible discrete series representation $\Delta(\mf{m})$ is unitary if and only if $\mf{m}$ is centered.
\end{enumerate}
\end{theorem}

\subsubsection{Classification of irreducible admissible representations via multi-segments}

We use the symbol $\ul{\mf{m}}=\{\mf{m}_{1},\ldots,\mf{m}_{k}\}$ for denoting a \textit{multi-segment}, i.e., a multi-set of segments. 
We say that a multi-segment $\ul{\mf{m}}$ is \textit{centered} if each segment contained in $\ul{\mf{m}}$ is centered.

We say that two segments $\mf{m}_{1}=[\rho_{1}; x_{1}, y_{1}]$ and $\mf{m}_{2}=[\rho_{2}; x_{2}, y_{2}]$ are \textit{linked} if $\mf{m}_{1}\not\subseteq\mf{m}_{2}$, $\mf{m}_{2}\not\subseteq\mf{m}_{1}$, and $\mf{m}_{1}\cup\mf{m}_{2}$ is a segment (note that this condition necessarily implies that $\rho_{1}\cong\rho_{2}$).
We say that a segment $\mf{m}_{1}=[\rho_{1}; x_{1}, y_{1}]$ \textit{precedes} $\mf{m}_{2}=[\rho_{2}; x_{2}, y_{2}]$ if $\mf{m}_{1}$ and $\mf{m}_{2}$ are linked and $x_1< x_2$. 

For any multi-segment $\ul{\mf{m}}=\{\mf{m}_{1},\ldots,\mf{m}_{k}\}$, we put
\[
\pi(\ul{\mf{m}}):=\Delta(\mf{m}_{1})\times\cdots\times\Delta(\mf{m}_{k}),
\]
where $\mf{m}_{i}=[\rho_{i}; x_{i}, y_{i}]$ are segments ordered so that $\mf{m}_{i}$ does not precede $\mf{m}_{j}$ whenever $i>j$ (such an ordering may not be unique, hence we fix one).

\begin{theorem}[{\cite[Theorem 6.1]{Zel80}}]\label{thm:Zel-nontemp}
Let $\ul{\mf{m}}$ be a multi-segment.
\begin{enumerate}
\item
The representation $\pi(\ul{\mf{m}})$ has a unique irreducible quotient denoted by $\Delta(\ul{\mf{m}})$.
Moreover, $\Delta(\ul{\mf{m}})$ is independent of the choice of the ordering as above of $\mf{m}_1,\ldots,\mf{m}_k$.
\item
Any irreducible admissible representation of $\GL_{n}(F)$ is of the form $\Delta(\ul{\mf{m}})$ for a unique multi-segment $\ul{\mf{m}}$.
\end{enumerate}
\end{theorem}

\begin{remark}
We remark that the statement of \cite[Theorem 6.1]{Zel80} is that $\pi(\ul{\mf{m}})$ has a unique irreducible `subrepresentation' when $\mf{m}_i$'s are ordered so that $\mf{m}_{i}$ does not precede $\mf{m}_{j}$ whenever `$i<j$'.
These two different conventions can be translated into each other by the so-called Zelevinsky involution; see, e.g., \cite[\S3.3]{LM16}.
\end{remark}

\subsubsection{Description of Jordan--Holder constituents}

We say that a multi-segment $\ul{\mf{m}}'$ is obtained by an \textit{elementary operation} from another multi-segment $\ul{\mf{m}}$ if the following holds:
\begin{quote}
We may write $\ul{\mf{m}}=\{\mf{m}_{1},\ldots,\mf{m}_k\}$ and $\ul{\mf{m}}'=\{\mf{m}'_{1}, \mf{m}'_{2}, \mf{m}_{3},\ldots,\mf{m}_{k}\}$, where each $\mf{m}_{i}$ is a segment such that $\mf{m}_{1}$ and $\mf{m}_{2}$ are linked and satisfy $\mf{m}'_{1}=\mf{m}_{1}\cup\mf{m}_{2}$ and $\mf{m}'_{2}=\mf{m}_{1}\cap\mf{m}_{2}$.
\end{quote}

\begin{theorem}[{\cite[Theorem 7.1]{Zel80}}]\label{thm:Zel-JH}
Let $\ul{\mf{m}}$ be a multi-segment.
Then the set of Jordan--Holder constituents of $\pi(\ul{\mf{m}})$ contains $\Delta(\ul{\mf{m}}')$ for a multi-segment $\ul{\mf{m}}'$ if and only if $\ul{\mf{m}}'$ can be obtained from $\ul{\mf{m}}$ by a chain of elementary operations.
\end{theorem}

\subsection{Temperedness and centeredness}

\begin{proposition}
An irreducible admissible representation $\Delta(\ul{\mf{m}})$ is tempered if and only if $\ul{\mf{m}}$ is centered.
\end{proposition}

\begin{proof}
We believe that this proposition is well-known, but explain some details.
In general, an irreducible admissible representation $\pi$ of a $p$-adic reductive group is tempered if and only if it is realized in the normalized parabolic induction of a unitary discrete series representation of a Levi subgroup (see, e.g., \cite[Section VII.2.6]{Renard}); note that such a parabolically induced representation is unitary, hence semisimple.
Thus, in the case of $\GL_n$, $\pi$ is tempered if and only if $\pi$ is contained in $\Delta(\mf{m}_1)\times\cdots\times\Delta(\mf{m}_k)$ for some centered segments $\mf{m}_1,\ldots,\mf{m}_k$.
As any two centered segments are not linked, we cannot construct any new multi-segment from a centered multi-segment.
Hence, by \cite[Theorem 4.2]{Zel80}, $\Delta(\mf{m}_1)\times\cdots\times\Delta(\mf{m}_k)$ is irreducible and equal to $\Delta(\ul{\mf{m}})$, where $\ul{\mf{m}}:=\{\mf{m}_1,\ldots,\mf{m}_k\}$.
\end{proof}

\subsubsection{Jacquet modules of discrete series representations}
The following proposition says that the Jacquet module of a discrete series representation is simply described by ``dividing" the corresponding segment.

\begin{proposition}[{\cite[Proposition 9.5]{Zel80}}]\label{prop:ds-J}
Let $\mf{m}=[\rho;x,y]$ be a segment, where $\rho$ is a unitary supercuspidal representation of $\GL_{r}(F)$.
We put $n:=r(y-x+1)$, hence $\Delta(\mf{m})$ is a discrete series representation of $\GL_n(F)$.
For $0<l<n$, we let $P_{n-l,l}$ denote the standard parabolic subgroup of $\GL_{n}$ with standard Levi $\GL_{n-l}\times\GL_{l}$.
Then we have
\[
J_{P_{n-l,l}}^{\GL_{n}} (\Delta(\mf{m}))
=
\begin{cases}
0& \text{if $r\nmid l$,}\\
\Delta([\rho;x+k,y])\boxtimes\Delta([\rho;x,x+k-1])& \text{if $r\mid l$ (write $l=rk$).}
\end{cases}
\]
\end{proposition}

\subsubsection{Pseudo-centered multi-segments}

For a multi-segment $\ul{\mf{m}}=\{\mf{m}_{1},\ldots,\mf{m}_{k}\}$ and an irreducible unitary supercuspidal representation $\rho$ of $\GL_{r}(F)$, we define the \textit{$\rho$-part} $\ul{\mf{m}}_{\rho}$ of $\ul{\mf{m}}$ to be the multi-set consisting of $\mf{m}_{i}$ which is of the form $[\rho;x_{i},y_{i}]$.

\begin{definition}
Let $\ul{\mf{m}}$ be a multi-segment.
For an irreducible unitary supercuspidal representation $\rho$ of $\GL_{r}(F)$, we write $\ul{\mf{m}}_{\rho}=\{\mf{m}_{1},\ldots,\mf{m}_{k}\}$ and $\mf{m}_{i}=[\rho;x_{i},y_{i}]$.
We say that $\ul{\mf{m}}_{\rho}$ is \textit{pseudo-centered} if the following holds:
\begin{quote}
    For any $z\in\R$, the sum of the multiplicities of $\rho|\det|^{z}$ in $\mf{m}_i$ (over $1 \leq i \leq k$) equals that of $\rho|\det|^{-z}$.
\end{quote}
We say that $\ul{\mf{m}}$ is \textit{pseudo-centered} if so is $\ul{\mf{m}}_{\rho}$ for any $\rho$.
\end{definition}

Note that a centered segment is obviously pseudo-centered.

\begin{lemma}\label{lem:pseudo-cent}
Let $\mf{m}_{1},\ldots,\mf{m}_{r}$ be segments.
We put $\mf{m}:=\{\mf{m}_{1},\ldots,\mf{m}_{r}\}$.
If $\Delta(\mf{m}_{1})\times\cdots\times\Delta(\mf{m}_{r})$ contains a tempered irreducible subquotient, then $\mf{m}$ is pseudo-centered.
\end{lemma}

\begin{proof}
Note that $\mf{m}_1,\ldots,\mf{m}_r$ is not necessarily ordered so that $\mf{m}_{i}$ does not precede $\mf{m}_{j}$ whenever $i>j$, hence the parabolically induced representations $\Delta(\mf{m}_{1})\times\cdots\times\Delta(\mf{m}_{r})$ and $\pi(\mf{m})$ may be different.
However, they have the same sets of irreducible subquotients ignoring the multiplicity (see \cite[Theorem 2.9]{Zel1}), hence it is enough to discuss the claim for $\pi(\mf{m})$.
Suppose that the set of Jordan--Holder factors of $\pi(\mf{m})$ contains a tempered irreducible subquotient, which is written by $\pi(\ul{\mf{m}}')$ with a multi-segment $\ul{\mf{m}}'$.
By Theorem \ref{thm:Zel-JH}, $\ul{\mf{m}}'$ is obtained from $\ul{\mf{m}}$ by a chain of elementary operations.
This implies that $\ul{\mf{m}}'_{\rho}$ is obtained from $\ul{\mf{m}}_{\rho}$ by a chain of elementary operations for any $\rho$.
Since $\pi(\ul{\mf{m}}')$ is tempered, $\ul{\mf{m}}'$ is centered (Theorem \ref{thm:Zel-nontemp} (3)), hence so is $\ul{\mf{m}}'_{\rho}$.
In particular, $\ul{\mf{m}}'_{\rho}$ is pseudo-centered.
Noting that being pseudo-centered is preserved under the elementary operation, we conclude that $\ul{\mf{m}}_{\rho}$ must be pseudo-centered.
\end{proof}

\subsection{Non-temperedness of the non-regular part}

We let 
\begin{itemize}
\item 
$G=\GL_n$, and
\item
$(H,\mc{H},s,\eta)=(G,{}^L{G},1,\mathrm{id})$.
\end{itemize}
Let $\phi_H=\phi$ be a tempered $L$-parameter of $H=G$.
Let $b\in B(G)$ and $L$ be the standard Levi subgroup such that $b$ comes from $b_L\in B(L)_\bas^+$.
Let $Q$ be the standard parabolic subgroup of $G$ with standard Levi $L$.
As discussed in the paragraph above Example \ref{ex:first}, then we have
\begin{equation}{\label{eq: Appendixtranseqn}}
    \Trans_{H}^{G_b}S\Theta_{\phi_H}^{H}
=
(\Trans_{L}^{G_b}J^{G}_{Q}S\Theta_\phi^G)\otimes\overline{\delta}^{-1/2}_Q.
\end{equation}

Let us describe the regular part of this distribution following Section \ref{ss: endo-reg-part}.
By replacing $\phi$ via conjugation if necessary, we choose a minimal standard Levi subgroup $M$ of $G$ such that $\phi$ factors through a discrete $L$-parameter $\phi_M$ of $M$.
Let $P$ be the standard parabolic subgroup of $G$ with standard Levi $M$.
We write $\pi_M$ for the unique discrete series representation of $M(F)$ contained in $\Pi_{\phi_M}^M$.
Note that $\pi_M$ is unitary since $\phi$ is tempered.
Then, by Assumption \ref{assumption:LIR-basic} (this is indeed a theorem in this case), 
\[
S\Theta_\phi^G
=I_P^G(S\Theta_{\phi_M}^M)
=I_P^G(\pi_M).
\]
Hence the right-hand side of \eqref{eq: Appendixtranseqn} becomes
\[
\Trans_{L}^{G_b}(J^{G}_{Q}\circ I_P^G(\pi_M))\otimes\overline{\delta}^{-1/2}_Q.
\]

Recall that 
\begin{itemize}
\item
$W:=W_{G}(T)$,
\item
$W^{M,L}:=\{w\in W \mid w(M\cap B)\subset B,\, w^{-1}(L\cap B)\subset B \}$,
\item
$W(M,L):=\{w\in W \mid w(M)\subset L\}$,
\item
$W[M,L]:=W^{M,L}\cap W(M,L)$.
\end{itemize}
(Here, we are omitting the script ``$\rel$" from the notation).
Also recall that we often write ${}^{w}M$ in short for $w(M)=wMw^{-1}$.

By the geometric Lemma of Bernstein--Zelevinsky (\cite[448 page]{Ber1}), we have
\[
J^{G}_{Q}\circ I^{G}_{P}(\pi_{M})
=\sum_{w\in W^{M,L}} I^{L}_{P_{2}}\circ w^{\ast}\circ J^{M}_{P_{1}}(\pi_{M}),
\]
where
\begin{itemize}
\item
$P_{1}$ is the standard parabolic subgroup of $M$ with standard Levi $L_{1}:=M\cap w^{-1}(L)$, 
\item
$P_{2}$ is the standard parabolic subgroup of $L$ with standard Levi $L_{2}:=w(M)\cap L$.
\end{itemize}
Note that, for any $w\in W[M,L]$, we have $L_{1}=M$ and $L_{2}=w(M)$, hence the summand equals $I^{Q}_{P}({}^{w}\pi_{M})$, which is an irreducible tempered representation.
Recall that, by definition, 
\[
[J^{G}_{Q}\circ I^{G}_{P}(\pi_{M})]_{\reg}
:=\sum_{w\in W[M,L]} I^{Q}_{P}({}^{w}\pi_{M})
\]
and 
\[
[\Trans_{H}^{G_b}S\Theta_{\phi_H}^{H}]_\reg
:=
\Trans_{L}^{G_b}([J^{G}_{Q}\circ I_P^G(\pi_M)]_\reg).
\]

Our aim is to show the following:

\begin{proposition}\label{prop:ABM}
For any $w\in W^{M,L}\smallsetminus W[M,L]$, any irreducible subquotient of $I^{L}_{P_{2}}\circ w^{\ast}\circ J^{M}_{P_{1}}(\pi_{M})$ is non-tempered.
In particular, the regular part $[J^{G}_{Q}\circ I^{G}_{P}(\pi_{M})]_{\reg}$ is the projection of $J^{G}_{Q}\circ I^{G}_{P}(\pi_{M})$ to its tempered part.
\end{proposition}

%\begin{corollary}\label{cor:ABM}
%Let $\Groth(G)$ be the Grothendieck group of the category of finite-length admissible representations of $G(F)$.
%Let $e_{\pi_{M}}[M,L]$ be the projector 
%\[
%e_{\pi_{M}}[M,L]\colon
%\Groth(G)\twoheadrightarrow \sum_{w\in W[M,L]}\Z(I^{Q}_{P}({}^{w}\pi_{M}))\subset \Groth(G)
%\]
% with respect to the subspace of $\Groth(G)$ spanned by $\{w\in W[M,L] \mid I^{Q}_{P}({}^{w}\pi_{M})\}$.
%Then we have
%\[
%[J^{G}_{Q}\circ I^{G}_{P}(\pi_{M})]_{\reg}
%=e_{\pi_{M}}[M,L](J^{G}_{Q}\circ I^{G}_{P}(\pi_{M})).
%\]
%\end{corollary}

We suppose that $w\in W^{M,L}\smallsetminus W[M,L]$.
Thus, in particular, ${}^{w}M\not\subset L$.

We introduce some ad hoc terminology and notation for convenience.

\begin{definition}\label{defn:position}
\begin{enumerate}
\item
We say that a subgroup $M'$ of $\GL_n$ is a \textit{single-block} subgroup if it is of the following form:
\[
M'=\left\{g=(g_{ij})_{ij}\in \GL_n \,\Bigg\vert\, \begin{array}{l}
     \text{$g_{ii}=1$ if $i\notin [n',m']$}  \\
     \text{$g_{ij}=0$ if $i\notin [n',m']$ or $j\notin [n',m']$}
\end{array}   \right\}
\]
for some $1\leq n' \leq m' \leq n$.
We call $m'-n'+1$ the \textit{size} of $M'$.
We call $n'$ (resp.\ $m'$) of $M'$ the \textit{upper-left entry} (resp.\ \textit{lower-right entry}) of $M'$.
\item 
For single-block subgroups $M'$ and $M''$ of $\GL_n$, we write $M'\nwarrow M''$ if the lower-right entry of $M'$ is smaller than the upper-left entry of $M''$.
\end{enumerate}
\end{definition}

We write $M=M^{(1)}\times\cdots\times M^{(r)}$, where each $M^{(i)}$ is a general linear group which is identified with a single-block subgroup of $\GL_n$ such that $M^{(i)}\nwarrow M^{(j)}$ for any $i<j$. 

Note that, since $w\in W^{M,L}$, it follows $M\cap L^{w}$ is a standard Levi of $M$, hence also of $G$ (see \cite[Lemma 2.11]{Ber1}).
In particular, we may write $M^{(i)}\cap L^{w}=M^{(i)}_{1}\times\cdots\times M^{(i)}_{n_{i}}$, where each $M^{(i)}_{j}$ is a single-block subgroup such that $M^{(i)}_{j}\nwarrow M^{(i)}_{j'}$ whenever $j<j'$.
On the other hand, ${}^{w}M\cap L$ is also a standard Levi of $L$, hence of $G$ (see \cite[Lemma 2.11]{Ber1}).
In particular, each factor $M^{(i)}_{j}$ of $M^{(i)}\cap L^{w}$ is mapped to a single-block subgroup of $G$ under the $w$-conjugation.

By noting that $w(M\cap B)\subset B$, we can check that the $w$-conjugation preserves the relative positions of the blocks $M^{(i)}_{1},\ldots, M^{(i)}_{n_{i}}$ in each $M^{(i)}\cap L^{w}$.
To be more precise, the following holds:

\begin{lemma}\label{lem:relative-position}
For any $1\leq i \leq r$, we have ${}^{w}M^{(i)}_{j}\nwarrow {}^{w}M^{(i)}_{j'}$ whenever $j<j'$.
\end{lemma}

We write the unitary discrete series representation $\pi_{M}$ of $M(F)=M^{(1)}(F)\times\cdots\times M^{(r)}(F)$ as 
\[
\pi_M
=
\Delta(\mf{m}^{(1)})\boxtimes\cdots\boxtimes\Delta(\mf{m}^{(r)})
\]
with centered segments $\mf{m}^{(1)}, \ldots, \mf{m}^{(r)}$.
We put $\ul{\mf{m}}:=\{\mf{m}^{(1)},\ldots,\mf{m}^{(r)}\}$.

Let us fix a unitary irreducible supercuspidal representation $\rho$ of $\GL_{m}(F)$ for some $m\in\Z_{>0}$ such that $\ul{\mf{m}}_{\rho}\neq0$.
By permuting $M^{(i)}$'s if necessary, we may assume that $\ul{\mf{m}}_{\rho}=\{\mf{m}^{(1)},\ldots,\mf{m}^{(s)}\}$ for some $1\leq s \leq r$.
Let us write $\mf{m}^{(i)}=[\rho;-x_{i},x_{i}]$ for $1\leq i\leq s$.
Furthermore, by again permuting $M^{(i)}$'s and also replacing the choice of $\rho$ if necessary, we may also assume that
\begin{itemize}
\item
$x_{1}\geq\cdots \geq x_{s}$,
\item
there exists $1\leq i\leq s$ satisfying ${}^{w}M^{(i)}\not\subset L$. 
\end{itemize}
(If we cannot find $i$ satisfying the second condition for any $\rho$, then it means that ${}^wM\subset L$, which contradicts $w\notin W[M,L]$.)

Let $1\leq k\leq s$ be the index such that ${}^{w}M^{(k)}\not\subset L$ and $x_k$ is the largest among all such $k$'s.
Note that there might be multiple such indices $k$.
In that case, we choose $k$ so that ${}^{w}M^{(k)}_{1}\nwarrow{}^{w}M^{(k')}_1$ for any other such index $k'$.

Now we start the proof.
Recall that our goal is to show that any irreducible subquotient of $I^{L}_{P_{2}}\circ w^{\ast}\circ J^{M}_{P_{1}}(\pi_{M})$ is non-tempered.
For this, we may assume that $I^{L}_{P_{2}}\circ w^{\ast}\circ J^{M}_{P_{1}}(\pi_{M})\neq0$.

\begin{proof}[Proof of Proposition \ref{prop:ABM}]
We write $L=L^{(1)}\times\cdots\times L^{(t)}$, where $L^{(i)}$'s are single-block subgroups such that $L^{(i)}\nwarrow L^{(j)}$ for any $i<j$.
Then $w^{\ast}\circ J^{M}_{P_{1}}(\pi_{M})$ is a representation of ${}^{w}M\cap L=({}^{w}M\cap L^{(1)})\times\cdots\times({}^{w}M\cap L^{(s)})$.
By writing $w^{\ast}\circ J^{M}_{P_{1}}(\pi_{M})=\boxtimes_{i=1}^{s}\pi^{(i)}$ according to this product expression of ${}^{w}M\cap L$, we have
\[
I^{L}_{P_{2}}\circ w^{\ast}\circ J^{M}_{P_{1}}(\pi_{M})
=
\bigl(I^{L^{(1)}}_{P^{(1)}_{2}}\pi^{(1)}\bigr)\boxtimes\cdots\boxtimes\bigl(I^{L^{(t)}}_{P^{(t)}_{2}}\pi^{(t)}\bigr),
\]
where $P^{(i)}_{2}:=P_{2}\cap L^{(i)}$.

Let $L^{(l)}$ be the block containing the upper-left entry of ${}^{w}M^{(k)}_{1}$.
To complete the proof, it is enough to show the following:
\begin{claim}
Any irreducible subquotient of $I^{L^{(l)}}_{P^{(l)}_{2}}\pi^{(l)}$ is non-tempered.
\end{claim}

We write ${}^{w}M\cap L^{(l)}=L^{(l)}_{1}\times\cdots\times L^{(l)}_{m_{l}}$ as usual and $\pi^{(l)}=\pi^{(l)}_{1}\boxtimes\cdots\boxtimes\pi^{(l)}_{m_{l}}$.

Let $L^{(l)}_{m}$ be the block which contains the upper-left entry of ${}^{w}M^{(k)}_1$.
Note that both $L^{(l)}_{m}$ and ${}^{w}M^{(k)}_1$ are single-block subgroups constituting the standard Levi subgroup ${}^wM\cap L$, hence we have $L^{(l)}_{m}={}^{w}M_{1}^{(k)}$.
By the assumption that $I^{L}_{P_{2}}\circ w^{\ast}\circ J^{M}_{P_{1}}(\pi_{M})\neq0$, we have $\pi^{(l)}_m\neq0$.
Thus, Proposition \ref{prop:ds-J} and Lemma \ref{lem:relative-position} implies that $\pi^{(l)}_{m}=\Delta(\mf{m}_{m}^{(l)})$, where $\mf{m}_{m}^{(l)}$ is a segment of the form $[\rho;z,x_{k}]$, where $-x_{k}<z\leq x_{k}$.
Also, the other components $\pi^{(l)}_{i}$ must be discrete series, so let us write $\pi^{(l)}_{i}=\Delta(\mf{m}_{i}^{(l)})$ with a segment $\mf{m}_{i}^{(l)}$.
Hence, with this notation, we have
\[
I^{L^{(l)}}_{P^{(l)}_{2}}\pi^{(l)}
=\Delta(\mf{m}_{1}^{(l)})\times\cdots\times\Delta(\mf{m}_{m_{l}}^{(l)}).
\]

For the sake of contradiction, we suppose that this parabolically induced representation contains a tempered irreducible subquotient.
Then, by Lemma \ref{lem:pseudo-cent}, the multi-segment $\{\mf{m}_{1}^{(l)},\ldots,\mf{m}_{m_{l}}^{(l)}\}$ is pseudo-centered.
Hence, since $\mf{m}_{m}^{(l)}=[\rho;z,x_{k}]$, at least one of $\mf{m}_{i}^{(l)}$ ($1\leq i\leq m_{l}$, $i\neq m$) must be of the form $[\rho;x,y]$ with $x\leq -x_{k}\leq y$.
Let $\mf{m}_{m'}^{(l)}=[\rho;x,y]$ be such a segment.
Recall that the index ``$k$'' was chosen so that $x_{k}$ is the largest among all indices $i$ satisfying ${}^{w}M^{(i)}\not\subset L$.

\begin{enumerate}
\item
If $x<-x_{k}$, then the segment $\mf{m}_{m'}^{(l)}=[\rho;x,y]$ ``originates" from some $\mf{m}^{(i)}=[\rho;-x_{i},x_{i}]$ with $x_{i}>x_{k}$.
In this case, by the definition of $k$, ${}^{w}M^{(i)}\subset L$, which implies that the segment $\mf{m}^{(i)}$ is not divided (in the sense of Proposition \ref{prop:ds-J}) when $J_{P_{1}}^{M}$ is applied.
Hence, $\mf{m}_{m'}^{(l)}$ is necessarily $[\rho;-x_{i},x_{i}]$, which is centered itself.
Therefore, so that $\{\mf{m}_{1}^{(l)},\ldots,\mf{m}_{s_{l}}^{(l)}\}$ is pseudo-centered, there must be another segment $\mf{m}_{m''}^{(l)}$ of the form $[\rho;x',y']$ with $x'\leq -x_{k}\leq y'$.
\item
If $x\geq-x_{k}$, we have $x=-x_{k}$.
\begin{itemize}
\item[(i)]
If $y>x_{k}$, then the same argument as in (1) implies that $\mf{m}^{(l)}_{m'}$ must be centered.
However, as $x=-x_k$, this cannot happen.
\item[(ii)]
If $y=x_{k}$, the same argument as in (1) implies that there must be another segment $\mf{m}_{m''}^{(l)}$ of the form $[\rho;x',y']$ with $x'\leq -x_{k}\leq y'$.
\item[(iii)]
Suppose that $y<x_{k}$.
By the definition of $k$ and Lemma \ref{lem:relative-position}, we cannot have $1\leq m'<m$.
However, if $m<m'\leq m_{l}$, then again the definition of $k$ and Lemma \ref{lem:relative-position} imply that there must be $m<m^\circ<m'$ such that $\mf{m}^{(l)}_{m^\circ}=[\rho;z',x_{k}]$ for some $z'$.
But then $\rho|\det|^{x_k}$ is contained in $\mf{m}^{(l)}_{m}$ and $\mf{m}^{(l)}_{m^\circ}$ while $\rho|\det|^{-x_k}$ is contained in $\mf{m}^{(l)}_{m'}$.
Therefore, so that $\{\mf{m}_{1}^{(l)},\ldots,\mf{m}_{s_{l}}^{(l)}\}$ is pseudo-centered, there must be another segment $\mf{m}_{m''}^{(l)}$ of the form $[\rho;x',y']$ with $x'\leq -x_{k}\leq y'$.
\end{itemize}
\end{enumerate}
By repeating this procedure of finding a segment $\mf{m}_{m''}^{(l)}$, we arrive at a contradiction.
\end{proof}

%\printbibliography
%\bibliographystyle{my_amsalpha}
%\bibliography{biblio} 
\end{document}